\documentclass[12pt,a4paper,final]{amsart}
\usepackage{amssymb}
\usepackage{hyperref}
\usepackage[notcite, notref]{showkeys}
\usepackage[utf8]{inputenc}
\usepackage[figuresright]{rotating}

\theoremstyle{plain}
\newtheorem{theorem}{Theorem}[section]
\newtheorem{proposition}[theorem]{Proposition}

\newtheorem{lemma}[theorem]{Lemma}

\theoremstyle{definition}

\newtheorem{remark}[theorem]{Remark}

\newtheorem{notation}[theorem]{Notation}

\theoremstyle{remark}

\numberwithin{equation}{section}

\newcommand{\Z}{\mathbb Z}

\newcommand{\C}{\mathbb C}

\newcommand{\fe}{\mathfrak e}
\newcommand{\ff}{\mathfrak f}
\newcommand{\fg}{\mathfrak g}
\newcommand{\fh}{\mathfrak h}

\newcommand{\fl}{\mathfrak l}

\newcommand{\fp}{\mathfrak p}

\newcommand{\ft}{\mathfrak t}

\DeclareMathOperator{\GL}{GL}

\DeclareMathOperator{\SO}{SO}
\DeclareMathOperator{\SU}{SU}
\DeclareMathOperator{\Sp}{Sp}

\DeclareMathOperator{\Spin}{Spin}

\newcommand{\so}{\mathfrak{so}}
\newcommand{\spp}{\mathfrak{sp}}
\newcommand{\su}{\mathfrak{su}}

\newcommand{\op}{\operatorname}
\newcommand{\Id}{\textup{Id}}

\newcommand{\mi}{\mathrm{i}}

\DeclareMathOperator{\Span}{Span}
\newcommand{\inner}[2]{\langle {#1},{#2}\rangle }
\newcommand{\innerdots}{\langle {\cdot},{\cdot}\rangle }
\newcommand{\ee}{\varepsilon}

\DeclareMathOperator{\Sym}{Sym}

\DeclareMathOperator{\Ad}{Ad}
\DeclareMathOperator{\ad}{ad}
\DeclareMathOperator{\Cas}{Cas}

\DeclareMathOperator{\st}{st}

\DeclareMathOperator{\Spec}{Spec}

\newcommand{\kil}{\operatorname{B}}

\newcommand{\PP}{\mathcal{P}}

\newcommand{\DD}{\mathbb{D}}

\title[First Laplace eigenvalue]
{First Laplace eigenvalue of strongly isotropy irreducible spaces}

\author{Emilio~A.~Lauret$^1$}
\address{$^1$Instituto de Matemática (INMABB), Departamento de Matemática, Universidad Nacional del Sur (UNS)-CONICET, Bahía Blanca, Argentina.}
\email{emilio.lauret@uns.edu.ar, fiorela.rossi@uns.edu.ar}

\author{Fiorela Rossi Bertone$^1$}

\author{Alejandro Tolcachier$^2$}
\address{$^2$FAMAF, Universidad Nacional de Córdoba and CIEM-CONICET, Av. Medina Allende s/n, X5000HUA
	Córdoba, Argentina}
\email{atolcachier@unc.edu.ar}

\thanks{The first and second named authors were supported by a grant from SGCYT--UNS (PGI 24/ZL25). 
Additionally, the first and third named authors were supported by CONICET (PIP 11220210100343CO and 11220210100606CO, respectively).}

\date{\today}

\begin{document}

\begin{abstract}
We study the smallest positive eigenvalue $\lambda_1$ of the Laplace-Beltrami operator associated with any compact strongly isotropy irreducible space. 
We provide an explicit expression for all simply connected cases. 
Furthermore, every strongly isotropy irreducible space is automatically an Einstein manifold, and we prove for each of them that $E<\lambda_1\leq 16E$, where $E$ denotes the corresponding Einstein constant. 
\end{abstract} 

\maketitle


\section{Introduction}

The smallest positive eigenvalue $\lambda_1(M,g)$ of the Laplace--Beltrami operator associated with a compact Riemannian manifold $(M,g)$ is of fundamental importance in geometric analysis, playing a central role in understanding the interplay between analysis, geometry, and topology. 
For instance, the celebrated Lichnerowicz theorem asserts that a positive lower bound on the Ricci curvature yields an explicit lower bound for the first eigenvalue, while Obata’s theorem characterizes the equality case by showing that the sphere is, essentially, the only manifold attaining it. 
Similarly, Cheeger’s theorem relates this eigenvalue to the isoperimetric constant of the manifold, establishing a deep link between spectral analysis and global geometry. 
Moreover, the first eigenvalue also plays a central role in geometric and dynamical phenomena, such as the rate of convergence of the heat flow and the stability of certain metrics in variational problems.

A connected Riemannian manifold $(M,g)$ is said to be \emph{isotropy irreducible} if the isotropy group of all isometries at a point $p$ acts irreducibly on $T_pM$, for every $p\in M$. 
It is automatically homogeneous, say $M=G/H$, with $g$ a $G$-invariant metric. 
We assume that $G/H$ is compact, in which case $G$ can be taken to be compact and connected, and $H$ a closed subgroup of $G$. 
It turns out that the only $G$-invariant metric on $G/H$ is, up to positive scaling, the standard metric $g_{\st}$ induced by the Killing form $\kil_\fg$ of $\fg$. 
Moreover, $(G/H,g_{\st})$ is a compact Einstein manifold with positive scalar curvature.
We denote by $E(G/H,g_{\st})$ its Einstein constant, which satisfies $\frac14\leq E(G/H,g_{\st})\leq \frac12$. 

Although (non-symmetric) isotropy irreducible spaces play an important role in differential geometry (see e.g.\ \cite[Ch.~III]{Wolf68}, \cite[\S7.D]{Besse}, \cite{WangZiller93}, \cite{HeintzeZiller96}) 
and are currently a common source of examples for different kinds of problems (see e.g.\ \cite{BerestovskiiGorbatsevich14}, \cite{ChrysikosGustadWinther19},  \cite{Morimoto21}, \cite{ChenChenZhu21}, \cite{Narita24}, \cite{XuTan24}), 
they have not been studied from a spectral point of view until recently (see \cite{Schwahn-Lichnerowicz}, \cite{LauretTolcachier-nu-stab}, \cite{ProvenzanoStubbe25}).

This article is motivated by the following problem posed by Saloff-Coste in \cite[page 198]{Saloff-Coste10}:
\begin{quote}
\it Decide whether or not there is a constant $c$ such that 
$$\lambda_1(G/H,g)\leq c\, E(G/H,g)$$
for every compact isotropy irreducible space $G/H$ and any $G$-invariant metric $g$ on $G/H$. 
\end{quote}

Note that the functional $\lambda_1(G/H,g)/E(G/H,g)$ is invariant up to positive scaling of $g$. 
Consequently, we can assume without losing generality that $g=g_{\st}$, the standard metric on $G/H$.

Our main result (Theorem~\ref{thm1:Saloff-Coste} below) provides an answer for the large subfamily of (compact) \emph{strongly isotropy irreducible spaces}: homogeneous spaces $G/H$ such that the connected component $H_0$ of $H$ acts irreducibly on $T_{eH}G/H$. 
(Note that every simply connected strongly isotropy irreducible space $G/H$ with $G$ connected satisfies $H$ is connected, and therefore $G/H$ is also an isotropy irreducible space.) 
They were classified independently by Manturov~\cite{Manturov61a,Manturov61b,Manturov66}, Wolf~\cite{Wolf68}, and Krämer~\cite{Kramer}, consisting of ten infinite families and 33 isolated cases (see Tables~\ref{table4:isotropyirred-families},\ref{table4:isotropyirredexcepcion-classical},\ref{table4:isotropyirredexcepcion-exceptional} for the simply connected cases). 
The remaining (compact) isotropy irreducible spaces were classified by Wang and Ziller~\cite{WangZiller91}.

\begin{theorem}\label{thm1:Saloff-Coste}
If $G/H$ is any compact strongly isotropy irreducible space, then 
$$
1
<
\frac{\lambda_1(G/H,g_{\st})}{E(G/H,g_{\st})} 
\leq 
16 
. 
$$
\end{theorem}

\begin{remark}
The lower bound in Theorem~\ref{thm1:Saloff-Coste} is sharp. 
In fact, for spheres $G/H=\SO(n+1)/\SO(n)\simeq S^n$, one has that ${\lambda_1(G/H,g_{\st})}/{E(G/H,g_{\st})} =\frac{n}{2(n-1)}/\frac12\to1$ as $n\to+\infty$. 
One easily checks that if we omit the cases where $(G/H,g_{\st})$ is isometric to a round sphere 
	(e.g.\ $G/H$ equal to 
	$\SO(n+1)/\SO(n)$ for any $n\geq 1$, 
	or $\SU(2)/\{e\} \simeq S^3$, 
	or $\SU(4)/\Sp(2)\simeq S^5$, 
	or $\op{G}_2/\SU(3)\simeq S^6$, 
	or $\Spin(7)/\op{G}_2\simeq S^7$), 
we get ${\lambda_1(G/H,g_{\st})}/{E(G/H,g_{\st})} \geq \frac{13}{9}\approx 1.44$ attained only at $G/H=\op{E}_6/\op{F}_4$.

The authors do not know whether the upper bound in Theorem~\ref{thm1:Saloff-Coste} is sharp. 
It was established (see Remark~\ref{rem4:upperbound-Sp(4)/Sp(1)xSO(4)}) that, if $G'/H'$ is covered by $G/H= (\Sp(4)/\Z_2)/(\Sp(1)/\Z_2\times\SO(4)/\Z_2)$ (Family IV, $n=4$), then 
${\lambda_1(G'/H',g_{\st})}/{E(G'/H',g_{\st})}\leq 16$.
Curiously, $G/H$ is the only simply connected strongly isotropy irreducible space for which the group $N_G(H)/(Z(G)\cdot H)$ is non-abelian, where $N_G(H)$ is the normalizer of $H$ in $G$ and $Z(G)$ denotes the center of $G$. 
In general, the strongly isotropy irreducible spaces covered by a simply connected strongly isotropy irreducible space $G/H$ depend on the group $N_G(H)/(Z(G)\cdot H)$ (see Remark~\ref{rem:universalcover}). 
\end{remark}

In the course of the proof, we obtained the explicit expression of $\lambda_1(G/H,g_{\st})$ (and its multiplicity) for each simply connected strongly isotropy irreducible space $G/H$ (see Tables~\ref{table4:isotropyirred-families2}--\ref{table4:isotropyirredexcepcion-exceptional}). 
The case of irreducible compact symmetric spaces was completed by Urakawa~\cite{Urakawa86}. 
Most of the isolated cases were computed in \cite{LauretTolcachier-nu-stab} with the help of a computer. 
Here we determine $\lambda_1(G/H,g_{\st})$ for the ten infinite families, as well as for six isolated cases that could not be handled computationally, as the required memory exceeded what was available due to the high rank of $\fg_\C$.
Tables~\ref{table4:isotropyirred-families2}--\ref{table4:isotropyirredexcepcion-exceptional} also show, in each row, lower and upper bounds for $\lambda_1(G'/H',g_{\st})$, for an arbitrary strongly isotropy irreducible space $G'/H'$ covered by $G/H$.

\subsection*{Organization of the paper}
Section~\ref{sec:preliminaries} provides the Lie-theoretical preliminaries used in the computation of the first eigenvalue for the simply connected non-symmetric strongly isotropy irreducible spaces carried out in Section~\ref{sec:lambda_1(simplyconnected)}. 
The last section is devoted to proving Theorem~\ref{thm1:Saloff-Coste}, which provides lower and upper bounds for $\lambda_1(G'/H',g_{\st})$ for an arbitrary strongly isotropy irreducible space $G'/H'$.

\subsection*{Acknowledgments}
The authors wish to express their thanks to Jorge Lauret and Marco Radeschi for several helpful comments on an earlier version of this article.

\begin{table}
	\caption{The 10 families of non-symmetric simply connected strongly isotropy irreducible spaces.}\label{table4:isotropyirred-families}
	{\renewcommand{\arraystretch}{1.5}\small
		\begin{tabular}{cccccc} 
			No. & $G/H$  & $(d\rho)_\C$ &Cond. & $Z(G)$ 
			& $N_G(H)/ZH$
			\\ [1ex]
			\hline \hline \\ [-4ex]
			I & 
			$\frac{\SU\big(\frac{n(n-1)}{2}\big)/\Z_d}{\SU(n)/\Z_n}$ 
			& $\sigma_{\eta_2}$ &
						{\tiny 
							\begin{tabular}{c}
								$n\geq 5$,\\[-4pt]
								$d=\frac{n}{2}$, $n$ even\\[-4pt]
								$d=n$, $n$ odd
							\end{tabular}
						}
			& $\Z_{\frac{n(n-1)}{2d}}$  & $e$
			\\ 
			\hline \\ [-4ex]
			
			II & 
			$\frac{\SU\big(\frac{n(n+1)}{2}\big)/\Z_d}{\SU(n)/\Z_n}$
			& $\sigma_{2\eta_1}$  &
			{\tiny 
				\begin{tabular}{c}
					$n\geq 3$,\\[-4pt]
					$d=\frac{n}{2}$, $n$ even\\[-4pt]
					$d=n$, $n$ odd
				\end{tabular}
						}
				& $\Z_{\frac{n(n+1)}{2d}}$  & $e$
			\\ \hline \\ [-4ex]
			III	& 
			$\frac{\SU(pq)/\Z_d}{\big(\SU(p)/\Z_p\big)\times \big(\SU(q)/\Z_q\big)}$ &$\sigma_{\eta_1}\widehat\otimes \,\, \sigma_{\eta_1'}'$ & 
			{\tiny $\begin{array}{c} 
				2\leq p\leq q, \\[-4pt]
				p+q\neq 4 \\[-4pt] 
				d=\op{lcm}(p,q)
			\end{array}$ }
			& $\Z_{\frac{pq}{d}}$  & $e$
			\\ \hline \\ [-3.5ex]
			IV & 
			$\begin{array}{c}
				\frac{\Sp(n)/\Z_2}{(\Sp(1)/\Z_2)\times (\SO(n)/\Z_2)}, \\[-6pt]
				{\text{\tiny $n$ even}} \\
				\frac{\Sp(n)/\Z_2}{(\Sp(1)/\Z_2)\times \SO(n)}, \\[-6pt]
				{{\text{\tiny $n$ odd}}} 
			\end{array}$
			 & 
			$\sigma_{\eta_1}\widehat\otimes \,\, \sigma_{\eta_1'}'$ & 
			$n\geq 3$
			& $e$ & \begin{tabular}{c}
				$e$, \tiny{$n$ odd}\\[-3pt]
				$\mathbb{S}_3$, \tiny{$n=4$}\\[-3pt]
				$\Z_2$, \tiny{$n>4$ even}
			\end{tabular}
			\\ \hline \\ [-3ex]
			
			V & 
			$\begin{array}{c}
				\frac{\Spin(n^2-1)}{\SU(n)/\Z_n}, \\[-6pt]
				{\text{\tiny $n$ odd}} \\
				\frac{\SO(n^2-1)}{\SU(n)/\Z_n}, \\[-6pt]
				{\text{\tiny $n$ even}} 
			\end{array}$ 
			&
			$\sigma_{\eta_1+\eta_{n-1}}$ &
			$n\geq3$
			& \begin{tabular}{c}
				$\Z_2\oplus\Z_2$, \tiny{$n$ odd}\\[-3pt]
				$e$, \tiny{$n$ even}
			\end{tabular} & \begin{tabular}{c}
			$e$, \tiny{$n\equiv 1,2\pmod 4$}\\[-3pt]
			$\Z_2$, \tiny{$n\equiv 0,3\pmod 4$}
			\end{tabular}
			\\ \hline \\ [-3ex]
			VI & 
			$\begin{array}{c}
				\frac{\Spin(2n^2-n-1)}{\Sp(n)/\Z_2},\\[-6pt]
				 {\text{\tiny $n\equiv 0,1\pmod 4$}} \\
				\frac{\SO(2n^2-n-1)}{\Sp(n)/\Z_2}, \\[-6pt]
				{\text{\tiny $n\equiv 2,3\pmod 4$}} 
			\end{array}$  &
			$\sigma_{\eta_2}$  & 
						$n\geq 3$ & 
						\begin{tabular}{c}
							$\Z_2$, \tiny{$n\equiv 0,3\pmod 4$}\\[-3pt]
							$\Z_2\oplus \Z_2$, \tiny{$n\equiv 1\pmod 4$}\\[-3pt]
							$e$, \tiny{$n\equiv 2\pmod 4$}
						\end{tabular} & $e$
			\\ \hline \\ [-3ex]
						VII  & 
			$\begin{array}{c}
				\frac{\Spin(2n^2+n)}{\Sp(n)/\Z_2}, \\[-6pt]
				{\text{\tiny $n\equiv 0,3\pmod 4$}} \\
				\frac{\SO(2n^2+n)}{\Sp(n)/\Z_2}, \\[-6pt]
				{\text{\tiny $n\equiv 1,2\pmod 4$}} 
			\end{array}$ &
			$\sigma_{2\eta_1}$  & 
									$n\geq 3$ & 
									\begin{tabular}{c}
										$\Z_2\oplus \Z_2$, \tiny{$n\equiv 0,3\pmod 4$}\\[-3pt]
										$e$, \tiny{$n\equiv 1\pmod 4$}\\[-3pt]
										$\Z_2$, \tiny{$n\equiv 2\pmod 4$}
									\end{tabular} & $e$
			\\
			\hline \\ [-3ex]
				VIII & 
			$\frac{\SO(4n)/\Z_2}{(\Sp(1)/\Z_2)\times(\Sp(n)/\Z_2)}$ & 
			$\sigma_{\eta_1}\widehat\otimes \,\, \sigma_{\eta_1'}'$  & 
				$n\geq 2$ & $e$ & $e$
			\\[3pt]
			\hline \\ [-3ex]
			
			IX   & 
			$\begin{array}{c}
				\frac{\Spin\big(\frac{n(n-1)}{2}\big)}{\SO(n)/\Z_2},\\[-6pt]
				 {\text{\tiny $n\equiv 0\pmod 4$}} \\
				 	\frac{\SO\big(\frac{n(n-1)}{2}\big)}{\SO(n)/\Z_2},\\[-6pt]
				 	{\text{\tiny $n\equiv 2 \pmod 4$}}\\
				\frac{\SO\big(\frac{n(n-1)}{2}\big)}{\SO(n)},\\[-6pt]
				 {\text{\tiny $n\equiv 1,3 \pmod 4$}} 
			\end{array}$ &
			$\begin{array}{cc} 
					\sigma_{2\eta_2} \\
					\sigma_{\eta_2}
				\end{array}$  & 
				$\begin{array}{cc} n=5, \\ n\geq 7 \end{array}$  & 
				\begin{tabular}{c}
					$e$, \tiny{$n\equiv 2,3\pmod 4$}\\[-3pt]
					$\Z_2$, \tiny{$n\equiv 1\pmod 4$}\\[-3pt]
					$\Z_2\oplus \Z_2$, \tiny{$n\equiv 0\pmod 8$}\\[-3pt]
					$\Z_4$, \tiny{$n\equiv 4\pmod 8$}
				\end{tabular} & 
				\begin{tabular}{c}
					$e$, \tiny{$n\equiv 0,1,3\pmod 4$}\\[-3pt]
					$\Z_2$, \tiny{$n\equiv 2\pmod 4$}
				\end{tabular}
			\\ 			\hline \\ [-3ex]
			X    & 
			$\begin{array}{c}
				\frac{\Spin\big(\frac{(n-1)(n+2)}{2}\big)}{\SO(n)/\Z_2},\\[-5pt]
				 {\text{\tiny $n\equiv 0\pmod 4$}} \\[2pt]
				 	\frac{\SO\big(\frac{(n-1)(n+2)}{2}\big)}{\SO(n)/\Z_2},\\[-5pt]			 	{\text{\tiny $n\equiv 2 \pmod 4$}}\\[2pt]
					\frac{\SO\big(\frac{(n-1)(n+2)}{2}\big)}{\SO(n)},\\[-5pt]
				 {\text{\tiny $n\equiv 1,3 \pmod 4$}} 
				\end{array}$  & 
			$\sigma_{2\eta_1}$
			& 	$n\geq 5$
				&
				\begin{tabular}{c}
					$\Z_2$, \tiny{$n\equiv 0,1,2\pmod 4$}\\[-3pt]
					$e$, \tiny{$n\equiv 2\pmod 4$}
				\end{tabular} & $e$
			\\ 
			\hline
		\end{tabular}
		
\medskip 

See Remark~\ref{rem:embedding} for details of the embedding $H\subset G$. 

	}
\end{table}

\begin{table}
	\caption{Bounds for $\lambda_1(G'/H',g_{\st})$ for the 10 families of non-symmetric strongly isotropy irreducible spaces.}\label{table4:isotropyirred-families2}
	{\renewcommand{\arraystretch}{1.5}\small
\begin{tabular}{ccccc} 
No. & $\mathfrak{g} /\mathfrak{h}$  & $E$ &
\begin{tabular}{c}
\text{Lower bound}\\[-8pt]
$\lambda_1(G/H,g_{\st})$ 
\end{tabular} 
& \text{Upper bound}
			\\ [1ex]
			\hline \hline \\ [-4ex]
			I & 
			$\frac{\su\big(\frac{n(n-1)}{2}\big)}{\su(n)}$ 
			& $\frac{1}{4}+\frac{2}{n(n-2)}$
			&\begin{tabular}{c}
				$\lambda^{\pi_{3\omega_1}}=\frac{42}{25}$, \, \tiny{$n=6$}\\
				$\lambda^{\pi_{2\omega_1+2\omega_{N-1}}}= 2+\frac{4}{n(n-1)}$
			\end{tabular}
			&$ \lambda^{\pi_{2\omega_1+2\omega_{N-1}}}=2+\frac{4}{n(n-1)}$ 
			\\ 
			\hline \\ [-4ex]
			
			II & 
		$\frac{\su\big(\frac{n(n+1)}{2}\big)}{\su(n)}$
			& $\frac{1}{4}+\frac{2}{n(n+2)}$
			&			\begin{tabular}{c}
				$\lambda^{\pi_{3\omega_1}}=\frac{15}{8}$, \, \tiny{$n=3$}\\
				$\lambda^{\pi_{2\omega_1+2\omega_{N-1}}}= 
				2+\frac{4}{n(n+1)}$
			\end{tabular}
			&$ \lambda^{\pi_{2\omega_1+2\omega_{N-1}}}=2+\frac{4}{n(n+1)}$ 
			\\ \hline \\ [-4ex]
			III	& 
			$\frac{\su(pq)}{\su(p)\oplus \su(q)}$ &
			$\frac{1}{2}+ \frac{p^2+q^2}{p^2q^2}$
			&  
			$\lambda^{\pi_{\omega_2+\omega_{pq-2}}}=2-\frac{2}{pq}$
			&  
			$\lambda^{\pi_{\omega_2+\omega_{pq-2}}}=2-\frac{2}{pq}$
			\\[4pt] \hline \\ [-3.5ex]
			IV & 
			$	\frac{\spp(n)}{\spp(1)\oplus \so(n)}$
			 & 
			$\frac{3}{8}+\frac{8-n}{8n(n+1)}$
			&
			$\lambda^{\pi_{2\omega_2}}= 2-\frac{1}{n+1}$
				&
				\begin{tabular}{c}
					$\lambda^{\pi_{8\omega_1}}=\frac{32}{5}$, \, \tiny{$n=4$}\\
			$ \lambda^{\pi_{4\omega_2}}=4+\frac{2}{n+1}$
			\end{tabular}
			\\ \hline \\ [-3ex]
			
			V & 
			$\frac{\so(n^2-1)}{\su(n)}$
			&
			$\frac{1}{4}+\frac{1}{n^2-3}$ &
			$\lambda^{\pi_{\ee_1+\ee_2+\ee_3}}= \frac{3}{2}-\frac{3}{2(n^2-3)}$
			& $\lambda^{\pi_{2\ee_1+2\ee_2+2\ee_3}}=3$
			\\[4pt] \hline \\ [-3ex]
			VI & 
			$\frac{\so(2n^2-n-1)}{\spp(n)}$  &
			$\frac{1}{4}+\frac{1}{(n^2-1)(2n-3)}$
			& $\lambda^{\pi_{3\omega_1}}= \frac{3}{2}+\frac{9}{2(2n^2-n-3)}$
			& $ \lambda^{\pi_{6\omega_1}}=3+\frac{18}{2n^2-n-3}$
			\\[4pt] \hline \\ [-3ex]
						VII  & 
		$\frac{\so(2n^2+n)}{\spp(n)}$ &
			$\frac{1}{4}+\frac{1}{2n^2+n-2}$ &
			$\lambda^{\pi_{\omega_3}}= \frac{3}{2}-\frac{3}{2(2n^2+n-2)}$
			& $\lambda^{\pi_{2\omega_3}}=3$
			\\[4pt]
			\hline \\ [-3ex]
				VIII & 
			$\frac{\so(4n)}{\spp(1)\oplus\spp(n)}$ & 
			$\frac{3}{8}+\frac{n+4}{8n(2n+1)}$ &
			$\lambda^{\pi_{\ee_1+\ee_2+\ee_3+\ee_4}}= 2-\frac{2}{2n-1}$
				& $\lambda^{\pi_{\ee_1+\ee_2+\ee_3+\ee_4}}=2-\frac{2}{2n-1}$
			\\[4pt]
			\hline \\ [-3ex]
			
			IX   & 
			$\frac{\so\big(\frac{n(n-1)}{2}\big)}{\so(n)}$  &
			 $\frac{1}{4}+\frac{2}{n^2-n-4}$
			& 	$\lambda^{\pi_{\omega_3}}= \frac{3}{2}-\frac{3}{n^2-n-4}$
				& $\lambda^{\pi_{2\omega_3}}=3$
			\\[4pt] 			\hline \\ [-3ex]
			X    & 
			$\frac{\so\big(\frac{(n-1)(n+2)}{2}\big)}{\so(n)}$  & 
			$\frac{1}{4}+\frac{2n}{(n^2-4)(n+3)}$
			& 	$\lambda^{\pi_{3\omega_1}}= \frac{3}{2}+\frac{9}{n^2+n-6}$
				& $\lambda^{\pi_{6\omega_1}}=3+\frac{18}{n^2+n-4}$
			\\[4pt] 
			\hline
		\end{tabular}
		
\medskip 

$E$ abbreviates $E(G/H,g_{\st})$, the Einstein constant of $(G/H,g_{\st})$.

The lower and upper bounds correspond to $\lambda_1(G'/H', g_{\mathrm{st}})$ for an arbitrary strongly isotropy irreducible space $G'/H'$ associated with each family (i.e.\ $\mathfrak{g}' = \mathfrak{g}$ and $\mathfrak{h}' = \mathfrak{h}$).

The lower bound equals $\lambda_1(G/H, g_{\mathrm{st}})$, where $G/H$ is the simply connected presentation in Table~\ref{table4:isotropyirred-families}. 

The upper bound follows by Theorem~\ref{thm4:Lambda+Lambda^*}, Proposition~\ref{prop4:N=ZH+pi(Z)trivial}, or  Remark~\ref{rem4:Z=eN_G(H)=H}, \ref{rem4:upperbound-Sp(4)/Sp(1)xSO(4)} or \ref{rem4:improved-upperbound}.

	}
\end{table}

{
\renewcommand{\arraystretch}{1.5}
\begin{table}[!ht]
\caption{Isolated simply connected strongly isotropy irreducible spaces with $G$ classical.}\label{table4:isotropyirredexcepcion-classical}

\begin{tabular}[t]{c c c c c r@{\,}lr@{\,}lcc} 
\hline 
\hline
No & $G$
 & $H$
 & $(d\rho)_\C$ & $E$& 
\multicolumn{2}{c}{
\begin{tabular}{c}
{\small \text{Lower bound}}\\[-8pt]
$\lambda_1(G/H,g_{\st})$ 
\end{tabular}
}
& \multicolumn{2}{c}{{\small \text{Upper bound}}}
&$Z(G)$&$\frac{N_G(H)}{ZH}$ 
\\ [0.5ex]
\hline 
\hline 
 1&$\SU(16)/\Z_4$ & $\SO(10)/\Z_2$  & $\sigma_{\eta_4}$  &$\frac{11}{32}$  
 &$\lambda^{\pi_{2\omega_2}}$&$=\frac{63}{32}$ 
 & $\lambda^{\pi_{4\omega_2}}$&$=\frac{35}{8}$
 & $\Z_4$ & $e$
 \\  \hline 
 2&$\SU(27)/\Z_3$ & $\op{E}_6/\Z_3$    & $\sigma_{\eta_1}$ & $\frac{11}{36}$ 
 &$\lambda^{\pi_{3\omega_1}}$&$=\frac{130}{81}$ 
 &$\lambda^{\pi_{6\omega_1}}$&$=\frac{29}{9}$
 & $\Z_9$ & $e$
 \\ \hline
 3&$\Spin(7)$  & $\op{G}_2$    & $\sigma_{\eta_1}$ & $\frac{9}{20}$  
 &$\lambda^{\pi_{\omega_3}}$&$=\frac{21}{40}$ 
 &$\lambda^{\pi_{2\omega_3}}$&$=\frac{6}{5}$
 & $\Z_2$ & $e$
  \\ \hline
 4& $\SO(133)$ & $\op{E}_7/\Z_2$   & $\sigma_{\eta_3}$ & $\frac{135}{524}$ 
 & $\lambda^{\pi_{\omega_3}}$&$=\frac{195}{131}$
 & $\lambda^{\pi_{\omega_3}}$&$=\frac{195}{131}$ 
 & $e$ & $e$
 \\ \hline
 5&$\Sp(2)/\Z_2$  & $\SO(3)$  & $\sigma_{3\eta_1}$ & $\frac{9}{20}$ 
 &$\lambda^{\pi_{4\omega_1}}$&$=\frac{8}{3}$ 
 &$\lambda^{\pi_{4\omega_1}}$&$=\frac{8}{3}$
 & $e$ & $e$
 \\ \hline
 6&$\Sp(7)/\Z_2$  & $\Sp(3)/\Z_2$ & $\sigma_{\eta_3}$  & $\frac{29}{80}$  
 &$\lambda^{\pi_{4\omega_1}}$&$=\frac{9}{4}$ 
 &$\lambda^{\pi_{4\omega_1}}$&$=\frac{9}{4}$ 
 & $e$ & $e$
 \\ \hline
 7&$\Sp(10)/\Z_2$ & $\SU(6)/\Z_6$  & $\sigma_{\eta_3}$  & $\frac{15}{44}$  
 &$\lambda^{\pi_{4\omega_1}}$&$=\frac{24}{11}$ 
 & $\lambda^{\pi_{8\omega_1}}$&$=\frac{56}{11}$
 & $e$ & $\Z_2$
 \\ \hline
 8&$\Sp(16)/\Z_2$ & $\SO(12)/\Z_2$ & $\sigma_{\eta_5}$ & $\frac{43}{136}$   
 &$\lambda^{\pi_{4\omega_1}}$&$=\frac{36}{17}$ 
 &$\lambda^{\pi_{4\omega_1}}$&$=\frac{36}{17}$ 
 & $e$ & $e$
 \\ \hline
 9&$\Sp(28)/\Z_2$ & $\op{E}_7/\Z_2$  &  $\sigma_{\eta_7}$ & $\frac{17}{58}$  
 &$\lambda^{\pi_{4\omega_1}}$&$=\frac{60}{29}$ 
 &$\lambda^{\pi_{4\omega_1}}$&$=\frac{60}{29}$ 
 & $e$ & $e$
 \\ \hline
10&$\Spin(14)$ & $\op{G}_2$  &  $\sigma_{\eta_2}$ & $\frac{1}{3}$    
&$\lambda^{\pi_{\omega_3}}$&$=\frac{11}{8}$ 
& $\lambda^{\pi_{2\omega_3}}$&$=3$
& $\Z_4$ & $e$
\\ \hline
11& $\Spin(16)/\Z_2 \;\dag$
& $\SO(9)$ &  $\sigma_{\eta_4}$ & $\frac{23}{56}$  
&$\lambda^{\pi_{2\omega_2}}$&$=\frac{15}{7}$ 
& $\lambda^{\pi_{2\omega_2}}$&$=\frac{15}{7}$ 
& $\Z_2$ & $e$
\\ \hline
12&$\Spin(26)$ & $\op{F}_4$  &  $\sigma_{\eta_4}$ & $\frac{1}{3}$    
&$\lambda^{\pi_{3\omega_1}}$&$=\frac{27}{16}$ 
& $\lambda^{\pi_{6\omega_1}}$&$=\frac{15}{4}$    
& $\Z_4$ & $e$
\\ \hline
13&$\Spin(42)$ & $\Sp(4)/\Z_2$& $\sigma_{\eta_4}$  & $\frac{19}{70}$   
&$\lambda^{\pi_{2\omega_2}}$&$=\frac{41}{20}$ 
& $\lambda^{\pi_{2\omega_2}}$&$=\frac{41}{20}$ 
& $\Z_4$ & $e$
\\ \hline
14&$\Spin(52)$ & $\op{F}_4$&  $\sigma_{\eta_1}$  & $\frac{27}{100}$  
& $\lambda^{\pi_{\omega_3}}$&$=\frac{147}{100}$ 
& $\lambda^{\pi_{2\omega_3}}$&$=3$
& $\Z_2\oplus\Z_2$ & $e$
\\ \hline
15&$\SO(70)/\Z_2$ & $\SU(8)/\Z_8$ & $\sigma_{\eta_4}$  & $\frac{179}{680}$ 
& $\lambda^{\pi_{2\omega_2}}$&$=\frac{69}{34}$ 
& $\lambda^{\pi_{4\omega_2}}$&$=\frac{71}{17}$
& $e$ & $\Z_2$
\\ \hline
16& $\Spin(78)$ & $\op{E}_6/\Z_3$ & $\sigma_{\eta_2}$   & $\frac{5}{19}$  
& $\lambda^{\pi_{\omega_3}}$&$=\frac{225}{152}$ 
& $\lambda^{\pi_{2\omega_3}}$&$=3$
& $\Z_4$ & $e$
\\ \hline
17&$\Spin(128)/\Z_2 \; \dag $
& $\SO(16)/\Z_2$ & $\sigma_{\eta_7}$ & $\frac{173}{672}$ 
& $\lambda^{\pi_{2\omega_2}}$&$=\frac{127}{63}$ 
& $\lambda^{\pi_{2\omega_2}}$&$=\frac{127}{63}$ 
& $\Z_2$ & $e$
\\ \hline
18&$\Spin(248)$ & $\op{E}_8$ &  $\sigma_{\eta_8}$  & $\frac{125}{492}$ 
& $\lambda^{\pi_{\omega_3}}$&$=\frac{245}{164}$ 
& $\lambda^{\pi_{2\omega_3}}$&$=3$
& $\Z_2\oplus\Z_2$ & $e$
\\  \hline 
\end{tabular}

\

$\dag$ The quotient $\Spin(N)/\Z_2$ is not isomorphic to $\SO(N)$. 

Analogous references as in Table~\ref{table4:isotropyirred-families2} hold for this table. 

\end{table}

\begin{table} 
\caption{Isolated simply connected strongly isotropy irreducible spaces with $G$ exceptional} \label{table4:isotropyirredexcepcion-exceptional}

\begin{tabular}[t]{cc  c cr@{\;}lr@{\;}lcc}
 \hline 
 \hline
No & $G$ & $H$ 
&$E$ & 
\multicolumn{2}{c}{
\begin{tabular}{c}
{\small \text{Lower bound}}\\[-8pt]
$\lambda_1(G/H,g_{\st})$ 
\end{tabular}
}
& \multicolumn{2}{c}{{\small \text{Upper bound}}}
& $Z(G)$ & $\frac{N_G(H)}{ZH}$
\\ 
 \hline 
 \hline 
19&
	$\op{E}_6$ &  
	$\SU(3)/\Z_3$             
	 &
	$\frac{11}{36}$ &
	$\lambda^{\pi_{2\omega_1}}$&$=\frac{14}{9}$ 
	& $\lambda^{\pi_{2\omega_1+2\omega_6}}$&$=\frac{10}{3}$ 
	& $\Z_3$ & $e$
\\  \hline 
20&
	$\op{E}_6/\Z_3$ & 
	$\SU(3)^3/(\Z_3\times\Z_3)$  & 
	$\frac{5}{12}$ &
	$\lambda^{\pi_{\omega_1+\omega_6}}$&$=\frac{3}{2}$
	& $\lambda^{\pi_{\omega_1+\omega_6}}$&$=\frac{3}{2}$ 
	& $e$ & $e$
\\ \hline 
21&
	$\op{E}_6$ & 
	$\op{G}_2$               & 
	$\frac{25}{72}$  &
	$\lambda^{\pi_{2\omega_1}}$&$=\frac{14}{9}$
	& $\lambda^{\pi_{2\omega_1+2\omega_6}}$&$=\frac{10}{3}$
	& $\Z_3$ & $e$
\\ \hline  
22&
	$\op{E}_6/\Z_3$ & 
	$(\SU(3)/\Z_3) \times \op{G}_2$   & 
	$\frac{19}{48}$ &
	$\lambda^{\pi_{\omega_1+\omega_6}}$&$=\frac{3}{2}$
	& $\lambda^{\pi_{\omega_1+\omega_6}}$&$=\frac{3}{2}$ 
	& $e$ & $e$
\\ \hline 
23&
	$\op{E}_7$ & 
	$\SU(3)/\Z_3$ & 
	$\frac{71}{252}$  &
	$\lambda^{\pi_{2\omega_7}}$&$=\frac{5}{3}$
	& $\lambda^{\pi_{4\omega_7}}$&$=\frac{11}{3}$
	& $\Z_3$ & $\Z_2$
\\ \hline 
24&
	$\op{E}_7/\Z_2$ & 
	$(\SU(6)\times\SU(3))/\Z_6$& 
	$\frac{5}{12}$ &
	$\lambda^{\pi_{\omega_6}}$&$=\frac{14}{9}$
	& $\lambda^{\pi_{2\omega_6}}$&$=\frac{10}{3}$
	& $e$ & $\Z_2$
\\  \hline 
25&
	$\op{E}_7/\Z_2$ & 
	$\op{G}_2\times (\Sp(3)/\Z_2)$  & 
	$\frac{7}{18}$  &
	$\lambda^{\pi_{\omega_6}}$&$=\frac{14}{9}$
	& $\lambda^{\pi_{\omega_6}}$&$=\frac{14}{9}$
	& $e$ & $e$
\\ \hline 
26&
	$\op{E}_7/\Z_2$ & 
	$\SO(3)\times\op{F}_4$   & 
	$\frac{47}{108}$  &
	$\lambda^{\pi_{\omega_6}}$&$=\frac{14}{9}$
	& $\lambda^{\pi_{\omega_6}}$&$=\frac{14}{9}$
	& $e$ & $e$
\\ \hline 
27&
	$\op{E}_8$ & 
	$\SU(9)/\Z_3$      & 
	$\frac{5}{12}$ &
	$\lambda^{\pi_{\omega_7}}$&$=2$
	& $\lambda^{\pi_{2\omega_7}}$&$=\frac{21}{5}$
	& $e$ & $\Z_2$
\\ \hline 
28&
	$\op{E}_8$ & 
	$(\op{E}_6\times\SU(3))/\Z_3$  & 
	$\frac{5}{12}$ &
	$\lambda^{\pi_{\omega_1}}$&$=\frac{8}{5}$
	& $\lambda^{\pi_{2\omega_1}}$&$=\frac{10}{3}$
	& $e$ & $\Z_2$
\\ \hline 
29&
	$\op{E}_8$ & 
	$\op{G}_2\times\op{F}_4$    & 
	$\frac{23}{60}$  &
	$\lambda^{\pi_{\omega_1}}$&$=\frac{8}{5}$
	& $\lambda^{\pi_{\omega_1}}$&$=\frac{8}{5}$ 
	& $e$ & $e$
\\ \hline 
30&
	$\op{F}_4$ & 
	$(\SU(3)\times\SU(3))/\Z_3$ & 
	$\frac{5}{12}$  &
	$\lambda^{\pi_{\omega_3}}$&$=\frac{4}{3}$
	& $\lambda^{\pi_{2\omega_3}}$&$=3$ 
	& $e$ & $\Z_2$
\\ \hline 
31&
	$\op{F}_4$ & 
	$\SO(3)\times\op{G}_2$ & 
	$\frac{29}{72}$  &
	$\lambda^{\pi_{2\omega_3}}$&$=\frac{13}{9}$
	& $\lambda^{\pi_{2\omega_3}}$&$=\frac{13}{9}$ 
	& $e$ & $e$
\\ \hline 
32&
	$\op{G}_2$ & 
	$\SO(3)$          & 
	$\frac{43}{112}$  &
	$\lambda^{\pi_{2\omega_2}}$&$=\frac{5}{2}$
	& $\lambda^{\pi_{2\omega_2}}$&$=\frac{5}{2}$ 
	& $e$ & $e$
\\ \hline 
33&
	$\op{G}_2$ & 
	$\SU(3)$           & 
	$\frac{5}{12}$    &
	$\lambda^{\pi_{\omega_1}}$&$=\frac{1}{2}$
	& $\lambda^{\pi_{2\omega_1}}$&$=\frac{7}{6}$
	& $e$ & $\Z_2$
\\  \hline 
\end{tabular}

\smallskip 

The embedding $(d\rho)_\C$ is omitted;
see \cite[Table~5]{LauretTolcachier-nu-stab}.

Analogous references as in Table~\ref{table4:isotropyirred-families2} hold for this table. 

\end{table} 
}

\section{Preliminaries}\label{sec:preliminaries}

This section is devoted to introducing Lie-theoretical preliminaries that will be used mainly in Section~\ref{sec:lambda_1(simplyconnected)}.

\subsection{The first Laplace eigenvalue of standard homogeneous manifolds}

A Riemannian manifold $(M,g)$ is called \emph{homogeneous} if its isometry group acts transitively on $M$. 
If there is an action of a Lie group $G$ by isometries on $M$ which is transitive, then $(M,g)$ is called \emph{$G$-homogeneous}. 
In this case, $M$ is diffeomorphic to $G/H$, where $H$ is the isotropy subgroup of the $G$-action at some point of $M$, and $(M,g)$ is isometric to $(G/H,g')$ for some $G$-invariant metric $g'$ on $M$. 

A $G$-homogeneous Riemannian manifold $(G/H,g)$ (here, $g$ is $G$-invariant) is called \emph{standard} if the inner product $g_{eH}$ on the tangent space at the point $eH$ coincides with $-\kil_{\fg}|_{\fp}$. 
(Note that $G$ must be compact and semisimple in order for the Killing form  $\kil_{\fg}|_{\fp}$ to be negative definite.) 
In addition, we assume that $G$ is connected. 

Every Riemannian manifold has a distinguished differential operator, the \emph{Laplace-Beltrami operator}. 
We next describe its spectrum on a standard Riemannian manifold $(G/H,g_{\st})$. 
The Casimir element $\Cas_\fg$ of $\fg$ is in the center of $\mathcal U(\fg_\C)$, thus $(d\pi)_\C(\Cas_\fg)$ commutes with $\pi(a)$ for all $a\in G$. 
Schur's Lemma implies that $(d\pi)_\C(\Cas_\fg)$ acts on $V_\pi$ by a scalar $\lambda^\pi$, for any $\pi\in\widehat G$, that is,
\begin{equation*}
(d\pi)_\C(\Cas_\fg)\cdot v=\lambda^\pi\, v
\qquad\forall\, v\in V_\pi.
\end{equation*}  
Moreover, Freudenthal's formula yields that 
\begin{equation}\label{eq:Freudenthal}
\lambda^\pi=
\kil_\fg^*(\Lambda_\pi,\Lambda_\pi+2\rho_\fg)
.
\end{equation} 
Here, $\Lambda_\pi$ denotes the highest weight of $\pi$ with respect to a maximal torus $T$ of $G$, and $\kil_\fg^*$ stands for the extension of $\kil_\fg$ to $\fg_\C^*\supset \ft_\C^*$. 
We set $\widehat G_H=\{\pi\in\widehat G: V_\pi^H\neq0\}$, the \emph{spherical representations} of the pair $(G,H)$. 
In the sequel, we will abbreviate $\innerdots=\kil_\fg^*(\cdot,\cdot)$ restricted to $(\mi\ft)^*\times (\mi\ft)^*$.

The next result is well known (see e.g.\ \cite[\S 5.6]{Wallach-book}).
\begin{theorem}\label{thm:Spec(standard)} 
Let $(G/H,g_{\st})$ be a connected, standard homogeneous space. 
The spectrum of its associated Laplace-Beltrami operator is given by 
$$
\Spec(G/H,g_{\st})
:=\Big\{\!\!\Big\{ 
	\underbrace{\lambda^\pi,\dots,\lambda^\pi}_{d_\pi\times d_\pi^H\text{-times} } 
	: \pi\in\widehat G_H 
\Big\}\!\!\Big\}
,
$$
where $d_\pi=\dim V_\pi$ and $d_\pi^H=\dim V_\pi^H$. 
In particular, the smallest positive eigenvalue is given by 
$$
\lambda_1(G/H,g_{\st})
=\min \left\{\lambda^\pi: \pi\in\widehat G_H\smallsetminus\{1_G\}\right\}
,
$$
and its multiplicity is given by 
\begin{equation*}
\sum_{\pi\in\widehat G_H\,:\, \lambda^\pi=\lambda_1(G/H,g_{\st})} 
\dim V_{\pi}\cdot\dim V_{\pi}^H.
\end{equation*}
\end{theorem}

Note that any $\pi'\in\widehat G_H\smallsetminus\{1_G\}$  gives an upper bound for $\lambda_1(G/H,g_{\st})$, namely $\lambda_1(G/H,g_{\st})\leq \lambda^{\pi'}$. 
Moreover, one has that  $\lambda_1(G/H,g_{\st})=\lambda^{\pi'}$ if and only if $\pi'\in\widehat G_H\smallsetminus\{1_G\}$ and 
\begin{equation}\label{eq:sufficient-lambda_1}
V_{\pi}^H=0
\text{ for all }
\pi\in\widehat G\smallsetminus\{1_G\}
\text{ satisfying }\lambda^{\pi}<\lambda^{\pi'}. 
\end{equation}

\begin{lemma}\label{lem3:dimV_pi^H}
	Let $H$ be a closed subgroup of a compact Lie group $G$, and let $\pi$ be a finite-dimensional representation of $G$.  
	Then $V_\pi^H=0$ if  and only if the trivial representation $1_H$ of $H$ does not occur in the decomposition in irreducible $H$-representations of the restriction $\pi|_H$ of $\pi$ to $H$. 
\end{lemma}

\begin{proof}
	This follows immediately from the fact that $\dim V_\pi^H$ is precisely the number of times that $1_H$ occurs in $\pi|_H$. 
\end{proof}

\subsection{Strongly isotropy irreducible spaces}

A homogeneous space $G/H$ is called \emph{isotropy irreducible} if the isotropy representation on the tangent space $T_{eH}G/H$ of the isotropy group $H$ is irreducible. 
Similarly, $G/H$ is called \emph{strongly isotropy irreducible} if the identity connected component $H^o$ of $H$ acts irreducibly on $T_{eH}G/H$. 
Of course, both notions coincide when $H$ is connected. 

Strongly isotropy irreducible spaces were classified independently by Manturov~\cite{Manturov61a,Manturov61b,Manturov66}, Wolf~\cite{Wolf68}, and Krämer~\cite{Kramer}. 
The remaining isotropy irreducible spaces were classified by Wang and Ziller~\cite{WangZiller91}. 

All non-symmetric simply connected strongly isotropy irreducible spaces are listed in Tables~\ref{table4:isotropyirred-families}, \ref{table4:isotropyirredexcepcion-classical}, and  \ref{table4:isotropyirredexcepcion-exceptional}. 
They consist of 10 families, 18 isolated cases with $\fg$ classical, and 15 isolated cases with $\fg$ exceptional. 
For each entry $G/H$ in the tables, $G$ is a compact connected simple Lie group, $H$ is a connected closed subgroup of $G$, $G/H$ is simply connected, $G$ acts effectively on $G/H$, and the linear isotropy action of $H$ on the tangent space of $G/H$ is irreducible. 

\begin{remark}\label{rem:Kconexo}
Let $G'/H'$ be an arbitrary presentation of any of the spaces listed in Tables~\ref{table4:isotropyirred-families},%
\ref{table4:isotropyirredexcepcion-classical},%
\ref{table4:isotropyirredexcepcion-exceptional}. 
If $G'$ is connected, then since $G/H = G'/H'$ is simply connected, the subgroup $H'$ must also be connected. 
Consequently, $G'/H'$ is indeed an isotropy irreducible space.
\end{remark}

We fix $T$ a maximal torus of $G$ such that $T\cap H$ is a maximal torus of $H$, and also a simple root system in the corresponding root system $\Phi(\fg_\C,\ft_\C)$. 
The Highest Weight Theorem ensures a correspondence between irreducible representations of $G$ (resp.\ of $\fg_\C$ of finite dimension) and the set of dominant $G$-integral weights $\PP^+(G)$ (resp.\ algebraically integral weights $\PP^+(\fg_\C)$).
For $\Lambda$ in $\PP^+(G)$ or $\PP^+(\fg_\C)$, we denote by $\pi_{\Lambda}$ the corresponding irreducible representation. 
We next identify those $\Lambda\in\PP^+(\fg_\C)$ satisfying $\lambda^{\pi_{\Lambda}} \leq 1$ for each complex simple Lie algebra $\fg_\C$ (see \cite[\S3]{SemmelmannWeingart22}). 
We take the ordering of simple roots given in Bourbaki's book \cite[\S{}VI.4]{Bourbaki-Lie4-6} for each complex simple Lie algebra $\fg_\C$.
This determines an ordering of the fundamental weights, which can be found in \cite[Table~1]{LauretTolcachier-nu-stab}.

\begin{notation}\label{notation:G_rho}
Let $H'$ be a compact and connected Lie group. 
For a complex representation $\rho$ of $H'$ of dimension $N$, we denote by $G_\rho$ the subgroup of $\GL(V_\rho)$ such that $\rho(H')\subset G_\rho$, where $G_\rho\simeq\SU(N)$, $\SO(N)$, or $\Sp(N/2)$ according to whether $\rho$ is of complex, real or quaternionic type. 
In all these cases, $(\fg_\rho)_\C$ is a classical complex simple Lie algebra, and the highest weight of its standard representation $\st_{(\fg_\rho)_\C}$ is always the first fundamental weight $\omega_1$. 
\end{notation}

\begin{remark}[Embedding of $H$ into $G$] \label{rem:embedding}
Let $G/H$ be any entry in Tables~\ref{table4:isotropyirred-families} or \ref{table4:isotropyirredexcepcion-classical}. 
Note that $\fg_\C$ is a complex simple Lie algebra of classical type.

Let $H'$ denote the connected and simply connected Lie group with Lie algebra $\fh'=\fh$. 
The Lie algebra representation $(d\rho)_\C$ of $\fh_\C'$ corresponding to $G/H$ (indicated in the tables) defines a classical Lie algebra $\fg_\rho\simeq \su(N),\so(N),\spp(N/2)$ according to whether $(d\rho)_\C$ is of complex, real or quaternionic type, where $N$ is the dimension of $(d\rho)_\C$. 

One can check that $\fg=\fg_\rho$.  
If $\rho:H'\to G$ is the only Lie group homomorphism with $(d\rho)_\C$ indicated in the table, then $H=\rho(H')\simeq H'/\op{Ker}(\rho)$. 
Note that $G$ and $G_\rho$ share the same universal cover, but they do not necessarily coincide. 
\end{remark}

We are now in position to give general conditions ensuring $V_\pi^H=0$ for certain $\pi\in\widehat G$. 
Note that, since $H$ is connected (see Remark~\ref{rem:Kconexo}), 
\begin{equation}
V_\pi^H=V_\pi^{\fh_\C}=V_\pi^{\fh_\C'}=\{v\in V_\pi: (d\pi)_\C(X)\cdot v=0\quad\forall \, X\in\fh_\C'\} 
\qquad\forall\,\pi\in\widehat G. 
\end{equation}

\begin{lemma}\label{lem:omega2-no-esferica}
Let $G/H$ be an entry in Tables~\ref{table4:isotropyirred-families} or \ref{table4:isotropyirredexcepcion-classical}, and let $\rho:H'\to G$ be as in Remark~\ref{rem:embedding}. 
Then
\begin{enumerate}
\item\label{item:st^H=0} 
The standard representation $\st_{\fg_\C}=\pi_{\omega_1}$,  when restricted to $\fh_\C$, satisfies $V_{\pi_{\omega_1}}^{\fh_\C'}=0$. 

\item\label{item:Ad^H=0} 
The adjoint representation $\Ad:G\to\SO(\fg)\simeq\SO(\dim\fg)$ of $G$ satisfies $V_{\Ad}^{H}=0$. 

\item\label{item:omega_2^H=2omega1=0} 
If $\fg_\C$ is of type $B_n$, $C_n$, or $D_n$, then $V_{\pi_{2\omega_1}}^{\fh_\C'}=V_{\pi_{\omega_2}}^{\fh_\C'}=0$.

\end{enumerate}
\end{lemma}

We prove this result after recalling some Lie-theoretical tools. 
Let $\fl$ be an arbitrary compact Lie algebra, and let us denote by $\fl_\C$ its complexification.

\begin{remark}[Exterior and symmetric products of the standard representation] \label{rem:ExtSym-standard}
We denote by $\bigwedge^p(\theta)$  (resp.\ $\Sym^k(\theta)$) the $p$-exterior (resp.\ the $k$-symmetric) power of $\theta$, which has underlying vector space equal to $\bigwedge^p (V_\theta)$ (resp.\ $\Sym^k(V_\rho)$), for any integer $0\leq p\leq \dim V_\theta$ (resp.\ $k\geq0$). 
Notice that $\Lambda^0(\theta) =\Sym^0(\theta) =1_{\fl_\C}$ and $\Lambda^1(\theta) =\Sym^1(\theta) =\theta$.

If $\fl_\C$ is simple of type $A_n$, we have that $\pi_{\omega_p}\simeq \textstyle{\bigwedge^p}(\pi_{\omega_1})$,  $\pi_{k\omega_1}\simeq \Sym^k(\pi_{\omega_1})$, and $\pi_{\omega_p}^*\simeq\pi_{\omega_{n+1-p}}$ for all integers $k\geq1$, $1\leq p\leq n$. 

If $\fl_\C$ is simple of type $B_n$ or $D_n$, $\omega_p=\ee_1+\dots+\ee_p$ if $1\leq p\leq n-1$ for type $B_n$ and if $1\leq p\leq n-2$ for type $D_n$, 
$\Sym^k(\pi_{\omega_1})\simeq \pi_{k\omega_1}\oplus \Sym^{k-2}(\pi_{\omega_1})$ 
for all integers $k\geq2$,
and 
\begin{equation*}
\textstyle{\bigwedge^p}(\pi_{\omega_1})
\simeq 
\begin{cases}
\pi_{\ee_1+\dots+\ee_p} &\quad\text{if }1\leq p\leq n-1, \\
\pi_{\ee_1+\dots+\ee_n} &\quad\text{if }p=n\text{ and $\fl_\C$ is of type $B_n$}, \\
\pi_{\ee_1+\dots+\ee_n}\oplus \pi_{\ee_1+\dots+\ee_{n-1}-\ee_n} &\quad\text{if }p=n\text{ and $\fl_\C$ is of type $D_n$}.
\end{cases}
\end{equation*}

If $\fl_\C$ is simple of type $C_n$,  $\textstyle{\bigwedge^p}(\pi_{\omega_1}) \simeq \pi_{\omega_p}\oplus \textstyle{\bigwedge^{p-2}}(\pi_{\omega_1})$, and
$\Sym^k(\pi_{\omega_1})\simeq \pi_{k\omega_1}$ 
for all integers $k\geq0$, $1\leq p\leq n$. 
\end{remark}

\begin{remark}\label{rem:rhoxrho^*}
Assume that $\fl_\C$ is simple. 
For finite-dimensional irreducible representations $\theta_1,\theta_2$ of $\fl_\C$, it is well known that $\theta_1\otimes\theta_2$ contains the trivial representation $1_{\fl_\C}$ of $\fl_\C$ if and only if $\theta_1^*\simeq\theta_2$.
In this case, $1_{\fl_\C}$ occurs exactly once. 
Moreover, when $\theta:=\theta_1=\theta_2$, 
$\theta\otimes\theta \simeq \Sym^2(\theta) \oplus {\textstyle\bigwedge^2}(\theta)$ and, 
	$1_{\fh_\C}$ occurs in $\Sym^2(\theta)$ if and only if $\theta$ is of real type, and 
	$1_{\fh_\C}$ occurs in ${\textstyle\bigwedge^2}(\theta)$ if and only if $\theta$ is of quaternionic type. 
\end{remark}

\begin{remark}[Tensor formulas]\label{rem:KoikeTerada}
Assume $\fl_\C$ is simple of type $B_n$, $C_n$ or $D_n$, with $n\geq3$. 
For non-negative integers $s\leq r$ and $q\leq p\leq n-2$, one has that (see e.g.\ \cite[p.~509--510]{KoikeTerada87} or \cite[\S1.64--1.66]{DictionaryOnLieAlgebras})
\begin{align*}
	\pi_{r\omega_1}\otimes \pi_{s\omega_1}&= \bigoplus_{j=0}^{s}\bigoplus_{i=0}^{j} \pi_{(r+s-2j)\omega_1+(j-i)\omega_2},\\
	\pi_{r\omega_1}\otimes \pi_{\omega_p}&= \pi_{r\omega_1+\omega_p}\oplus \pi_{(r-1)\omega_1+\omega_{p+1}} \oplus \pi_{(r-2)\omega_1+\omega_{p}}\oplus \pi_{(r-1)\omega_1+\omega_{p-1}},\\
	\pi_{\omega_p}\otimes \pi_{\omega_q}&= \bigoplus_{j=0}^{q}\bigoplus_{i=0}^{j} \pi_{\omega_{p+q-j-i}+\omega_{j-i}}.
\end{align*}
Here, $\omega_k=0$ if $k>n$. 
The assumption $p\leq n-2$ is not essential. 
Formulas also hold for $p=n-1$ for type $B_n$ and $p=n-1,n$ for type $C_n$.
Moreover, they hold for $p=n$ for type $B_n$ replacing $\omega_n$ by $\ee_1+\dots+\ee_n=2\omega_n$, and for $p=n-1,n$ for type $D_n$ replacing $\omega_{n-1}$ and $\omega_n$ by $\ee_1+\dots+\ee_{n-1}=\omega_{n-1}+\omega_n$ and $\ee_1+\dots+\ee_n=2\omega_n$ respectively. 

When $\fl_\C$ is of type $A_n$, similar identities are provided by Pieri's formula. 
We will directly cite \cite[Prop.~15.25]{FultonHarris-book} each time we use it to avoid stating it here. 
\end{remark}

\begin{proof}[Proof of Lemma~\ref{lem:omega2-no-esferica}]
For \eqref{item:st^H=0}, since $\pi_{\omega_1}$ is the standard representation of $\fg_\C$, we clearly have that $\pi_{\omega_1}|_{\fh_\C'} \simeq (d\rho)_\C$. 
Now, since $(d\rho)_\C$ is irreducible and non-trivial in all cases, we deduce $V_{\pi_{\omega_1}}^{\fh_\C'}=0$ by Lemma~\ref{lem3:dimV_pi^H}.

We now prove \eqref{item:Ad^H=0}. 
The underlying vector space of $\Ad$ is $\fg$. 
Its decomposition into irreducible $H$-modules is $\fg=\fh\oplus \fp$ (because $H$ is connected), where $\fh$ is the adjoint representation of $H$ and $\fp$ is an $H$-irreducible module of dimension $>1$ by assumption. 
Neither summand is trivial, thus $V_{\Ad}^H=0$ by Lemma~\ref{lem3:dimV_pi^H}.

We next prove \eqref{item:omega_2^H=2omega1=0}.
Assume that $\fg_\C$ is of type $B_n$, $C_n$, or $D_n$. 
It follows from Remark~\ref{rem:KoikeTerada} that
$
\pi_{\omega_1}\otimes \pi_{\omega_1}\simeq 
\pi_0\oplus \pi_{2\omega_1} \oplus \pi_{\omega_2}. 
$
Hence 
\begin{equation*}
1+\dim V_{\pi_{2\omega_1}}^{\fh_\C'} +\dim V_{\pi_{\omega_2}}^{\fh_\C'}
= \dim V_{\pi_{\omega_1}\otimes \pi_{\omega_1}}^{\fh_\C'}
.
\end{equation*}
It remains to show that the right-hand side equals 1 using Lemma~\ref{lem3:dimV_pi^H}. 

We have that 
$\big(\pi_{\omega_1}\otimes \pi_{\omega_1}\big)|_{\fh_\C'} 
= \pi_{\omega_1}|_{\fh_\C'}\otimes \pi_{\omega_1}|_{\fh_\C'}
\simeq\rho\otimes\rho
$.
Since $\rho$ is $\fh_\C'$-irreducible and self-adjoint (because it is of real or quaternionic type, since $G$ is of type $B_n,C_n,D_n$; see Remark~\ref{rem:embedding}), it follows from Remark~\ref{rem:rhoxrho^*} that $1_{\fh_\C'}$ occurs exactly once in $\rho\otimes \rho$. 
\end{proof}

\begin{remark}
Lemma~\ref{lem:omega2-no-esferica} cannot be extended to irreducible symmetric spaces. 
For instance, the embedding $\rho$ of the round sphere $S^n=\SO(n+1)/\SO(n)$ is not irreducible. 
Note that $\rho$ is irreducible in the cases considered in Lemma~\ref{lem:omega2-no-esferica}, and this fact was essential in the proof. 
Actually, in the sphere, one has $V_{\pi_{\omega_1}}^{\SO(n)}\neq0$, 
$V_{\pi_{2\omega_1}}^{\SO(n)}\neq0$, 
thus 
$\lambda^{\pi_{\omega_1}}, \lambda^{\pi_{2\omega_1}}\in\Spec(S^n,g_{\st})$.
Moreover, $\lambda_1(\SO(n+1)/\SO(n),g_{\st}) =\lambda^{\pi_{\omega_1}}$. 
\end{remark}

\begin{notation}\label{not:irrep-G-H}
We will set some conventions for each time we have a homogeneous space $G/H$ and we deal with representations of $G$ and $H$ simultaneously (e.g.\ when considering branching rules from $G$ to $H$). 

We will always fix a maximal torus $T$ of $G$ such that $T\cap H$ is a maximal torus of $H$. 
We use Bourbaki's choice (see \cite[\S{}VI.4]{Bourbaki-Lie4-6}) for the ordering of simple roots and fundamental weights. 
The next table describes the notation used for each group:
\begin{equation*}
\begin{array}{ccccc}
\text{group} &\text{irred.\ rep.} & \text{highest weight} & \text{fund.\ weights} 
\\ \hline
\rule{0pt}{12pt}
G & \pi & \Lambda & \omega_i \\
H & \sigma & \zeta & \eta_i \\
\end{array}
.
\end{equation*}
\end{notation}

When $\fh$ is semisimple, we will use prime symbols to designate the second factor.
For instance, when $\fh=\fh_1\oplus\fh_2$ is the decomposition of $\fh$ in simple ideals, 
any irreducible representation of $\fh$ is of the form $\sigma\widehat\otimes \,\sigma'$ with $\sigma$ an irreducible representation of $\fh_1$ and $\sigma'$ an irreducible representation of $\fh_2$.
Similarly, 
the highest weights are of the form $\zeta+\zeta'$, and the fundamental weights are $\eta_i,\eta_j'$.

\begin{remark}[Exterior and symmetric product]
\label{rem:ExtSym(suma)o(tensor)}

Let $\theta_1,\theta_2$ be finite-dimensional representations of $\fl_\C$. 
Then, as $\fl_\C$-modules, we have 
\begin{equation*}
\label{eq2:Lambda^2(V+W)-Sym^2(V+W)}
\begin{aligned}
{\textstyle \bigwedge^2}(\theta_1\oplus \theta_2) &
\simeq {\textstyle \bigwedge^2}(\theta_1)
\oplus {\textstyle \bigwedge^2}(\theta_2)
\oplus \theta_1\otimes \,\theta_2
,
\\
\Sym^2(\theta_1\oplus \theta_2) &
\simeq \Sym^2(\theta_1)
\oplus \Sym^2(\theta_2)
\oplus \theta_1\otimes \,\theta_2
.
\end{aligned}
\end{equation*}

Now, let $\fl_1, \fl_2$ be arbitrary compact Lie algebras and $\tau_1,\tau_2$ finite-dimensional representations of $(\fl_1)_\C,(\fl_2)_\C$ respectively. 
It is well known that
\begin{equation*}
\label{eq2:Lambda^2(VxW)-Sym^2(VxW)}
\begin{aligned}
{\textstyle \bigwedge^2}(\tau_1\widehat\otimes \, \tau_2) &
\simeq {\textstyle \bigwedge^2}(\tau_1) \widehat\otimes \, \Sym^2(\tau_2)
\oplus \Sym^2(\tau_1) \widehat\otimes \, {\textstyle \bigwedge^2}(\tau_2)
,
\\
\Sym^2(\tau_1\widehat\otimes \, \tau_2) &
\simeq \Sym^2(\tau_1) \widehat\otimes \, \Sym^2(\tau_2)
\oplus {\textstyle \bigwedge^2}(\tau_1) \widehat\otimes \, {\textstyle \bigwedge^2}(\tau_2)
.
\end{aligned}
\end{equation*}
\end{remark}

\begin{remark}\label{rem:autovalor_suma_dominantes}
For $\Lambda,\Lambda'\in\PP^+(\fl_\C)$, one has that  $\lambda^{\pi_{\Lambda+\Lambda'}}>\lambda^{\pi_{\Lambda}}+\lambda^{\pi_{\Lambda'}}$. 
Indeed, \eqref{eq:Freudenthal} gives
\begin{align*}
	\lambda^{\pi_{\Lambda+\Lambda'}} &=	\langle \Lambda+\Lambda'+2\rho, \Lambda+\Lambda' \rangle = \langle \Lambda+2\rho, \Lambda \rangle + \langle \Lambda'+2\rho, \Lambda' \rangle + 2\langle \Lambda,\Lambda' \rangle \\
	&= \lambda^{\pi_{\Lambda}}+\lambda^{\pi_{\Lambda'}} +  2\langle \Lambda,\Lambda' \rangle > \lambda^{\pi_{\Lambda}}+\lambda^{\pi_{\Lambda'}},
\end{align*}
because $\langle\Lambda,\Lambda'\rangle>0$, since both are dominant. 
In particular, $\lambda^{\pi_{\Lambda}}> \sum_j b_j\lambda^{\pi_{\omega_j}}$ if  $\Lambda= \sum_j b_j\omega_j$.
\end{remark}

\section{First Laplace eigenvalue}\label{sec:lambda_1(simplyconnected)}

In this section we provide an explicit expression for $\lambda_1(G/H,g_{\st})$ for each simply connected non-symmetric strongly isotropy irreducible space $G/H$. 
Many cases were already determined in \cite{LauretTolcachier-nu-stab} with the help of a computer (see \cite[\S4.7]{LauretTolcachier-nu-stab} for details). 
What remains are all the ten infinite families and the six isolated cases (Nos.~4, 14--18) where the calculations exceeded the available memory due to the large rank of $\fg_\C$.

The proof is carried out case by case, although some cases will be treated together.
For instance, Theorem~\ref{thm:lambda1(SO(N)/Ad(H))} simultaneously covers three families and five isolated cases.

\subsection{Isotropy subgroup embedded by the adjoint representation}

Let $H'$ be a compact, connected and simple Lie group of dimension $N$.
Its adjoint representation is clearly irreducible and of real type; thus $\Ad :H'\to\SO(N)$ (see Notation~\ref{notation:G_rho}). 
It turns out that 
$$
M_{H'}:=\SO(N)/\Ad(H')
$$ 
is an isotropy irreducible space. 
We omit the case $H'=\SU(2)$ since $\dim M_{H'}=0$. 

Let us denote by $\widetilde M_{H'}$ the universal cover of $M_{H'}$. 
Wolf also provided the realization 
$$
\widetilde M_{H'}=G/H
$$  
for each $H'$,  
with $G$ a connected Lie group acting effectively. 
The realization $G/H$ is the one appearing in the tables, namely, 
Families V, VII, and IX when $\fh'$ is unitary, symplectic, or orthogonal respectively, and the Isolated cases Nos.~4, 10, 14, 16, and 18 when $\fh'$ is of type $E_7$, $G_2$, $F_4$, $E_6$, and $E_8$ respectively.
Note that $G$ is $\SO(N)$ or $\Spin(N)$ in all cases and $\fh\simeq\fh'$.

\begin{theorem}\label{thm:lambda1(SO(N)/Ad(H))}
Let $H'$ be a compact, connected, and simple Lie group of dimension $N\geq8$. 
Then,  
$$
\lambda_1(\widetilde M_{H'},g_{\st})
= \lambda^{\pi_{\ee_1+\ee_2+\ee_3}} 
= \frac{3(N-3)}{2(N-2)}
= \frac32-\frac{3}{2(N-2)}
.
$$
Its multiplicity in $\Spec(G/H,g_{\st})$ is equal to $\dim V_{\pi_{\ee_1+\ee_2+\ee_3}} =\binom{N}{3}$ if $\fh'\not\simeq \su(3)$, and $168$ for $\fh'\simeq \su(3)$. 
\end{theorem}

\begin{proof}
Notice that $\fg=\so(N)$. 
For any $0\leq p\leq \lfloor\tfrac{N}{2}\rfloor$, we have that $V_{\pi_{\ee_1+\dots+\ee_p}} =\textstyle{\bigwedge^p} (\pi_{\omega_1}) =\textstyle{\bigwedge^p} (\C^N)$ (see Remark~\ref{rem:ExtSym-standard}). 
Hence
\begin{align*}
\pi_{\ee_1+\dots+\ee_p}|_{\fh}&
\simeq \pi_{\ee_1+\dots+\ee_p}|_{\fh_\C'}
= \pi_{\ee_1+\dots+\ee_p}\circ\ad_{\fh_\C'}
= \textstyle{\bigwedge^p} (\pi_{\ee_1+\dots+\ee_p}\circ\ad_{\fh_\C'}) 
\simeq  \textstyle{\bigwedge^p} (\fh_\C'). 
\end{align*}

Let $n$ denote the rank of $\fh_\C'$. 
The Poincaré polynomial associated with $\fh_\C'$ is given by 
\begin{equation}\label{eq:PoincarePolynomial}
P(1_{\fh_\C'}, {\textstyle\bigwedge}(\fh_\C'),x) :=\prod_{k\in D_{\fh_\C'}} (1+t^{2k-1}),
\end{equation} 
where $D_{\fh_\C'}$ consists of the degrees of the generators for the $H'$-invariant polynomials (see Table~\ref{table:exponents}). 
It is well known (see e.g.\ \cite[\S1]{DiTranti}) that the $k$-th term of $P$ equals the number of times the trivial representation $1_{\fh_\C'}$ occurs in $\textstyle{\bigwedge^k}(\fh_\C')$. 
We note from Table~\ref{table:exponents} that $2$ is the smallest exponent for any choice of $\fh_\C'$.
It follows immediately that the coefficient next to $t^3$ is exactly one. 
We conclude that $1_{\fh_\C'}$ occurs once in $\pi_{\ee_1+\ee_2+\ee_3}|_{\fh_\C'}$, which means by Lemma~\ref{lem3:dimV_pi^H} that $V_{\pi_{\ee_1+\ee_2+\ee_3}}^{\fh_\C'}\neq 0$.

\begin{table}[!htbp]
\renewcommand{\arraystretch}{1.4}
\caption{Degrees of the generators for the $H'$-invariant polynomials in $\fh_\C'$}\label{table:exponents}
$
\begin{array}[t]{cc}
\hline\hline
\fh' &  \text{degrees} 
\\ \hline\hline \rule{0pt}{14pt}
\su(n+1) & 2,3,\dots,n+1\\
\so(2n+1) & 2,4,\dots,2n\\
\spp(n) & 2,4,\dots,2n\\
\so(2n),\,n\geq3 &  n,2,4,\dots,2(n-1)\\
\end{array}
$
\qquad
$
\begin{array}[t]{cc}
\hline\hline
\fh' &  \text{degrees} 
\\ \hline\hline \rule{0pt}{14pt}
%
%
\fe_6  & 2,5,6,8,9,12\\
\fe_7  & 2,6,8,10,12,14,18\\
\fe_8 & 2,8,12,14,18,20,24,30\\
\ff_4 & 2,6,8,12\\
\fg_2 & 2,6 
\end{array}
$
\end{table}

Since $\ee_1+\ee_2+\ee_3 \in\PP^+(G)$ (because $G=\Spin(N)$ or $G=\SO(N)$), we have seen that $\pi_{\ee_1+\ee_2+\ee_3}\in\widehat G_H$, and Theorem~\ref{thm:Spec(standard)} gives 
$
\lambda^{\pi_{\ee_1+\ee_2+\ee_3}}
\in \Spec\big(\widetilde M_{H'}, g_{\st}\big)
.
$

In order to establish $\lambda_1(G/H,g_{\st}) =\lambda^{\pi_{\ee_1+\ee_2+\ee_3}}$, it remains to check \eqref{eq:sufficient-lambda_1}, i.e.\ $V_\pi^{H}=V_\pi^{\fh_\C'}=0$ for all $\pi\in\widehat G\smallsetminus\{1_G\}$ satisfying  $\lambda^\pi<\lambda^{\pi_{\omega_3}}$. 
Those irreducible representations $\pi$'s are explicitly described in Lemma~\ref{lem:SOlambda^pi<lambda^omega3} below. 
We already know that $V_{\pi}^{\fh_\C'}=0$ for $\pi=\pi_{\omega_1},\pi_{\omega_2},\pi_{2\omega_1}$ by Lemma~\ref{lem:omega2-no-esferica}. 
We next list the compact simple Lie algebras $\fh_\C'$ of dimension $N\leq 21$, and the corresponding branching rules for the spin representations (see Notation~\ref{not:irrep-G-H} for the notation of the irreducible representations of $\fh_\C'$): 
\begin{itemize}
\item
For $\fh'=\su(3)$ ($N=8$),  $\pi_{\omega_4}|_{\fh_\C'}\simeq \sigma_{\eta_1+\eta_2}$ and $\pi_{\omega_3}|_{\fh_\C'}\simeq \sigma_{\eta_1+\eta_2}$ by \cite[pp.\ 209]{LieART}. 

\item
For $\fh'=\so(5)\simeq\spp(2)$ ($N=10$),  $\pi_{\omega_5}|_{\fh_\C'}\simeq\pi_{\omega_4}|_{\fh_\C'} \simeq \sigma_{\eta_1+\eta_2}$ by \cite[pp.\ 229]{LieART}.

\item
For $\fh'=\fg_2$ ($N=14$),  $\pi_{\omega_7}|_{\fh_\C'}\simeq\pi_{\omega_6}|_{\fh_\C'} \simeq \sigma_{\eta_1+\eta_2}$
by \cite[pp.\ 255]{LieART}.

\item
For $\fh'=\su(4)\simeq \so(6)$ ($N=15$),  $\pi_{\omega_7}|_{\fh_\C'}\simeq 2 \sigma_{\eta_1+\eta_2+\eta_3}$
by \cite[pp.\ 262]{LieART}. 

\item
For $\fh'=\so(7)$ ($N=21$), 
$\pi_{\omega_{10}}|_{\fh_\C'}\simeq  2\sigma_{\eta_1+\eta_2+\eta_3}$
by \cite[pp.\ 284]{LieART}.

\item
For $\fh'=\spp(3)$ ($N=21$), 
$\pi_{\omega_{10}}|_{\fh_\C'}\simeq  2\sigma_{\eta_1+\eta_2+\eta_3}$
by \cite[pp.\ 284]{LieART}.
\end{itemize}
Since the trivial representation $1_{\fh_\C'}=\sigma_0$ never occurs, it follows that $V_{\pi_{\omega_m}}^{\fh_\C'}=0$ and $V_{\pi_{\omega_{m-1}}}^{\fh_\C'}=0$ in all of these cases, and $\lambda_1(G/H,g_{\st}) =\lambda^{\pi_{\ee_1+\ee_2+\ee_3}}$ follows.

{
It only remains to compute the multiplicity of $\lambda^{\pi_{\ee_1+\ee_2+\ee_3}}$ in $\Spec(G/H,g_{\st})$. 
According to Theorem~\ref{thm:Spec(standard)}, this number is equal to 
\begin{equation*}
\sum_{\pi\in\widehat G_H\,:\, \lambda^\pi=\lambda^{\pi_{\ee_1+\ee_2+\ee_3}}} 
\dim V_{\pi}\cdot\dim V_{\pi}^H.
\end{equation*}
The last part of Lemma~\ref{lem:SOlambda^pi<lambda^omega3} implies that the required multiplicity is given by
$$
\dim V_{\pi_{\ee_1+\ee_2+\ee_3}}\cdot \dim V_{\pi_{\ee_1+\ee_2+\ee_3}}^H 
=\dim V_{\pi_{\ee_1+\ee_2+\ee_3}}
=\dim {\textstyle\bigwedge^3}(\C^N)
=\tbinom{N}{3}
$$ 
if $N=\dim\fh\geq 10$. 
Furthermore, since there are no complex simple Lie algebras of dimension $9$, the remaining case is $N=8$, i.e.\ $\fh'= \su(3)$ and $G=\Spin(8)$.  
Now, \cite[pp.\ 209]{LieART} gives $\pi_{\omega_1+\omega_3}|_{\fh_\C'}\simeq \pi_{\omega_1+\omega_4}|_{\fh_\C'}\simeq  \sigma_{0}\oplus \sigma_{3\eta_1}\oplus \sigma_{\eta_1+\eta_2} \oplus \sigma_{3\eta_2} \oplus \sigma_{2\eta_1+2\eta_2}$, 
which implies $\dim V_{\pi_{\omega_1+\omega_3}}^{\fh_\C'} = \dim V_{\pi_{\omega_1+\omega_4}}^{\fh_\C'} = 1$ by Lemma~\ref{lem3:dimV_pi^H}. 
We conclude that the multiplicity in this case is given by $\dim V_{\pi_{\ee_1+\ee_2+\ee_3}}+\dim V_{\pi_{\omega_1+\omega_3}}+\dim V_{\pi_{\omega_1+\omega_4}}=3\cdot 56=168$, and the proof is complete. 
}
\end{proof}

\begin{remark}
The highest weight $\ee_1+\ee_2+\ee_3$ of $\pi_{\ee_1+\ee_2+\ee_3}$ in Theorem~\ref{thm:lambda1(SO(N)/Ad(H))} is equal to $\omega_3$ in all cases except for $\fh'=\su(3)$ when $N=8$ and $\ee_1+\ee_2+\ee_3=\omega_3+\omega_4$. 
\end{remark}

\begin{lemma}\label{lem:SOlambda^pi<lambda^omega3}
Let $N$ be an integer $\geq8$, $m=\lfloor N/2\rfloor$, and $\Lambda\in\PP^+(\so(N)_\C)$.

We have that $\lambda^{\pi_{\Lambda}}<\lambda^{\pi_{\ee_1+\ee_2+\ee_3}}$ if and only if $\Lambda$ is one of the following:
\begin{itemize}
\item $0$, $\omega_1$, $\omega_2$ for every $N$.

\item $2\omega_1$ if $N\geq 10$. 

\item $\omega_{m}$ if $N\leq 21$.
\item $\omega_{m-1}$ if $N\leq 20$ even. 
\end{itemize}

We have that $\lambda^{\pi_{\Lambda}} =\lambda^{\pi_{\ee_1+\ee_2+\ee_3}}$ if and only if $\Lambda$ is one of the following:
$\ee_1+\ee_2+\ee_3$, 
or also $\omega_1+\omega_3$, $\omega_1+\omega_4$ if $N=8$, 
or also $2\omega_1$, $\omega_1+\omega_4$ if $N=9$.
\end{lemma}

\begin{proof}
Write $\Lambda=\sum_{i=1}^{m} b_j\omega_j \in \PP^+(\so(N)_\C)$ with $b_1,\dots,b_n\in\Z_{\geq0}$. By Remark~\ref{rem:autovalor_suma_dominantes}, we have  $\lambda^{\pi_{\Lambda}}> \sum b_j\lambda^{\pi_{\omega_j}}$.
One has that $\rho_{\fg_\C}=\sum_{j=1}^m (\frac{N}{2}-j)\ee_j$	and $\inner{\ee_i}{\ee_j} =\delta_{i,j}\frac{1}{N-2}$.
Now, \eqref{eq:Freudenthal} gives
	\begin{align*}
		\lambda^{\pi_{\omega_k}} = \begin{cases}
		\frac{k(N-k)}{2(N-2)} \qquad \mbox{if $1\leq k\leq \lfloor \frac{N+1}{2} \rfloor-2$}, \\[2mm]
		\frac{N(N-1)}{16(N-2)} \qquad
		\mbox{if $k\geq \lfloor \frac{N+1}{2} \rfloor -1$}.
		\end{cases}
	\end{align*}

One can easily check that 
	$\lambda^{\pi_{\omega_1}} <\lambda^{\pi_{\omega_2}} <\lambda^{\pi_{\omega_3}}$,
	$\lambda^{\pi_{\omega_k}} >\lambda^{\pi_{\omega_3}}$ for all $4\leq k\leq \lfloor \frac{N+1}{2} \rfloor-2$, 
	$\lambda^{\pi_{\omega_k}} +\lambda^{\pi_{\omega_j}} >\lambda^{\pi_{\omega_3}}$ for all $1\leq k<j\leq \lfloor\tfrac{N+1}{2}\rfloor-2$,
	$\lambda^{\pi_{\omega_k}} +\lambda^{\pi_{\omega_m}} >\lambda^{\pi_{\omega_3}}$ for all $k\geq2$,
	$\lambda^{\pi_{k\omega_1}}> 3\lambda^{\pi_{\omega_1}} >\lambda^{\pi_{\omega_3}}$ for all $k\geq3$, 
	$\lambda^{\pi_{2\omega_2}}> 2\lambda^{\pi_{\omega_2}} >\lambda^{\pi_{\omega_3}}$ for all $k\geq2$, 
	and
	$\lambda^{\pi_{\omega_m}} >\lambda^{\pi_{\omega_3}}$ if and only if $N\geq22$ (and the same holds for $\lambda^{\pi_{\omega_{m-1}}}$ if $N$ is even).
In addition, if $N\in\{8,9\}$ (i.e. $m=4$), then
$$
\lambda^{\pi_{\omega_1+\omega_4}} =
\inner{\tfrac12(3\ee_1+\ee_2+\ee_3+\ee_4)} {\tfrac12(3\ee_1+\ee_2+\ee_3+\ee_4) +{\textstyle \sum\limits_{j=1}^4} (N-2j)\ee_j}
= \tfrac{3N-9}{2(N-2)}
= \lambda^{\pi_{\omega_3}}
,
$$
and the same holds for $\lambda^{\pi_{\omega_1+\omega_3}}$ for $N=8$.
Furthermore,
	$$\lambda^{\pi_{2\omega_1}} = \tfrac{N}{N-2}\leq \tfrac{3(N-3)}{2(N-2)}=  \lambda^{\pi_{\omega_3}}
	\quad \Longleftrightarrow \quad
	N\geq9,$$ and the equality is attained only at $N=9$.   
The proof is complete since all the cases have been considered. 
\end{proof}

\begin{remark}\label{rem:isolated-Ad(H)}
Some cases of Theorem~\ref{thm:lambda1(SO(N)/Ad(H))} had been established in \cite{LauretTolcachier-nu-stab} with the help of a computer. 
They are
$3\leq n\leq 8$ for Family V, 
$3\leq n\leq 6$ for Family VII, 
$7\leq n\leq 10$ for Family IX,
and the Isolated case No.~10 (i.e.\ $G/H=\Spin(14)/\op{G}_2$). 
\end{remark}

\subsection{Family VI}

We set $H'=\Sp(n)$ for some $n\geq3$. 
The irreducible representation $\rho:=\sigma_{\eta_2}$ of $H'$ is of real type, and has dimension $N:=2n^2-n-1$. 
The space $\SO(N)/\rho(H')$ is isotropy irreducible. 
Its universal cover is given by $G/H$ where 
\begin{align}\label{eq:G/H-fliaVI}
G&=
\begin{cases}
\Spin(N)&\text{ if }n\equiv 0,1\pmod 4,\\
\SO(N)&\text{ if }n\equiv 2,3\pmod 4,
\end{cases}
& H\simeq\Sp(n)/\Z_2.
\end{align}

\begin{lemma}\label{lem:Sym^2-Lambda^2(eta_2)}
For $\fh'=\spp(n)$, we have 
	$\displaystyle 
	\Sym^2(\sigma_{\eta_2})
	\simeq \sigma_{\eta_4} 
	\oplus \sigma_{2\eta_2}
	\oplus \sigma_{\eta_2}
	\oplus \sigma_{0}
	$ and $\displaystyle 
	\textstyle\bigwedge^2(\sigma_{\eta_2})
	\simeq \sigma_{\eta_1+\eta_3} 
	\oplus \sigma_{2\eta_1}	$.
\end{lemma}

\begin{proof}
We know that $\sigma_{\eta_2}\otimes \sigma_{\eta_2}=
	\Sym^2(\sigma_{\eta_2}) \oplus \displaystyle 
	\textstyle\bigwedge^2(\sigma_{\eta_2})$. 
Furthermore (see e.g.\ Remark~\ref{rem:KoikeTerada}) 
\begin{equation}\label{eq:eta2xeta2}
\sigma_{\eta_2}\otimes \sigma_{\eta_2}
\simeq \sigma_{\eta_4} 
	\oplus \sigma_{2\eta_2}
	\oplus 
	\sigma_{\eta_1+\eta_3} 
	\oplus \sigma_{2\eta_1} \oplus \sigma_{\eta_2}
	\oplus \sigma_{0}.
\end{equation}

The task is now to determine, for each irreducible term at the right-hand side, whether it lies in $\Sym^2(\sigma_{\eta_2})$ or in $\textstyle\bigwedge^2(\sigma_{\eta_2})$. 
The trivial representation $\sigma_0$ is in $\Sym^2(\sigma_{\eta_2})$ because $\sigma_{\eta_2}$ is of real type (see e.g.\ \cite[\S26.3]{FultonHarris-book}). 
The representation $\sigma_{2\eta_2}$ is also in $\Sym^2(\sigma_{\eta_2})$. 
Indeed, if $v$ denotes a highest weight vector of $\sigma_{\eta_2}$, then $v^2$ is a highest weight vector of $\Sym^2(\sigma_{\eta_2})$ with highest weight $\eta_2+\eta_2=2\eta_2$. 
The adjoint representation $\sigma_{2\eta_1}$ is included in $\textstyle\bigwedge^2(\sigma_{\eta_2})$ (see  \cite[Thm.~3.3]{PanyushevYakimova08}). 

Now, the only way to fill the number
\begin{align*}
\dim\textstyle\bigwedge^2(\sigma_{\eta_2})-\dim\sigma_{2\eta_1}&
= \tbinom{2n^2-n-1}{2}-\tbinom{2n+1}{2}
	=\tfrac{4n^4-4n^3-9n^2+n+2}{2}
\end{align*}
with the values 
$\dim\sigma_{\eta_4}= \binom{2n}{4}-\binom{2n}{2}=\frac{4n^4-12n^3-n^2+3n}{6}$, 
$\dim\sigma_{\eta_1+\eta_{3}} =\frac{4n^4-4n^3-9n^2+n+2}{2}$,  
and
$\dim\sigma_{\eta_2} = \binom{2n}{2}-1=2n^2-n-1$, 
is clearly with $\dim\sigma_{\eta_1+\eta_{3}}$. 
(All the dimensions can be computed using Weyl dimension formula \cite[Thm.~5.84]{Knapp-book-beyond}.)
We conclude that $\textstyle\bigwedge^2(\sigma_{\eta_2}) =\sigma_{\eta_1+\eta_3}\oplus\sigma_{2\eta_1}$, and the other identity follows. 
\end{proof}

\begin{theorem}\label{thm:flia6}
If $G$ and $H$ are as in \eqref{eq:G/H-fliaVI}, then  
$$
\lambda_1(G/H,g_{\st})=\lambda^{\pi_{3\omega_1}}
= \frac{3(N+1)}{2(N-2)}
= \frac{3}{2}+\frac{9}{2(N-2)}
.
$$ 
{
Its multiplicity in $\Spec(G/H,g_{\st})$ is $\dim V_{\pi_{3\omega_1}}=\frac{N+4}{N-2}\binom{N}{3}$ for $n\geq4$ and $2548$ for $n=3$.}
\end{theorem}
\begin{proof}
One can check (see e.g.\ Remark~\ref{rem:KoikeTerada}) that 
\begin{align}\label{eq:2omega1xomega1;omega2xomega1}
	\pi_{\omega_2}\otimes\pi_{\omega_{1}}
	\simeq \pi_{\omega_3}
	\oplus \pi_{\omega_1+\omega_2}
	\oplus \pi_{\omega_1}
\quad \text{ and } \quad
\pi_{2\omega_1}\otimes\pi_{\omega_{1}}
\simeq \pi_{3\omega_1}
\oplus \pi_{\omega_1+\omega_2}
\oplus \pi_{\omega_1}
.	
\end{align}
Restricting to $\fh_\C$, or equivalently to $\fh_\C'$, we obtain that
\begin{equation}\label{eq:omega2xomega1|_h}
\begin{aligned}
\pi_{\omega_2}\otimes\pi_{\omega_1}|_{\fh_\C}&
	\simeq \left({\textstyle\bigwedge^2}(\pi_{\omega_1}) \otimes \pi_{\omega_1} \right) |_{\fh_\C'}
	={\textstyle\bigwedge^2}(\pi_{\omega_1}|_{\fh_\C'}) \otimes \pi_{\omega_1} |_{\fh_\C'}
	\\ & 
	\simeq {\textstyle\bigwedge}^2(\sigma_{\eta_2}) \otimes\sigma_{\eta_2}
	\simeq (\sigma_{\eta_1+\eta_3} 
	\oplus \sigma_{2\eta_1}) \otimes \sigma_{\eta_2},
\end{aligned}
\end{equation}
and
\begin{equation}\label{eq:Sym^2xomega1|_h}
\begin{aligned}
\big(\pi_{2\omega_1}\oplus\pi_0\big) \otimes\pi_{\omega_{1}}|_{\fh_\C}
&\simeq \left(\Sym^2(\pi_{\omega_1}) \otimes\pi_{\omega_{1}}\right)|_{\fh_\C'}
= \Sym^2(\pi_{\omega_1}|_{\fh_\C'}) \otimes \pi_{\omega_{1}}|_{\fh_\C'}
\\ &
= \Sym^2(\sigma_{\eta_2}) \otimes \sigma_{\eta_2}
\simeq \big(\sigma_{\eta_4} 
	\oplus \sigma_{2\eta_2}
	\oplus \sigma_{\eta_2}
	\oplus \sigma_{0}\big) \otimes \sigma_{\eta_2}
	.
\end{aligned}
\end{equation}
In both cases, we used that $\pi_{\omega_1}|_{\fh_\C'}=\sigma_{\eta_2}$ (because $\pi_{\omega_1}$ is the standard representation of $\fg_\C$) and Lemma~\ref{lem:Sym^2-Lambda^2(eta_2)} in the last step.

Taking into account Remark~\ref{rem:rhoxrho^*}, it follows immediately from \eqref{eq:omega2xomega1|_h} and \eqref{eq:Sym^2xomega1|_h} that
\begin{align}
\dim V_{\pi_{\omega_2}\otimes\pi_{\omega_1}}^{\fh_\C'}=0
\qquad \mbox{ and } \qquad
\dim V_{(\pi_{2\omega_1}\oplus\pi_0) \otimes\pi_{\omega_{1}}}^{\fh_\C'}=1.
\end{align}
(Note that $\sigma_{\eta_2}^*\simeq\sigma_{\eta_2}$ for the identity at the right.)
Now, \eqref{eq:2omega1xomega1;omega2xomega1} gives
$
	\dim V_{\pi_{\omega_3}}^{\fh_\C'}
	+\dim V_{\pi_{\omega_1+\omega_2}}^{\fh_\C'}
	+\dim V_{\pi_{\omega_1}}^{\fh_\C'}
	=\dim V_{\pi_{\omega_2}\otimes\pi_{\omega_{1}}}^{\fh_\C'}=0
$, which implies 
$\dim V_{\pi_{\omega_3}}^{\fh_\C'}
= \dim V_{\pi_{\omega_1+\omega_2}}^{\fh_\C'}
= \dim V_{\pi_{\omega_1}}^{\fh_\C'}=0
$, and also 
\begin{align*}
	\dim V_{\pi_{3\omega_1}}^{\fh_\C'}&
	=
	\dim V_{\pi_{3\omega_1}}^{\fh_\C'}
	+\dim V_{\pi_{\omega_1+\omega_2}}^{\fh_\C'}
	+2\dim V_{\pi_{\omega_1}}^{\fh_\C'}
	\\ & 
	=
	\dim V_{\pi_{3\omega_1}\oplus \pi_{\omega_1+\omega_2}\oplus \pi_{\omega_1}}^{\fh_\C'}+ \dim V_{\pi_{\omega_1}}^{\fh_\C'}
	= \dim V_{(\pi_{2\omega_1}\oplus \pi_{0})\otimes \pi_{\omega_1}}^{\fh_\C'} =1
.
\end{align*}

We have seen that $\lambda^{\pi_{3\omega_1}}\in\Spec(G/H,g_{\st})$, and it remains to check \eqref{eq:sufficient-lambda_1}. 
Since $N=2n^2-n-1\geq 27$ for all $n\geq4$, and $N=14$ for $n=3$, then those $\pi$'s are described in Lemma~\ref{lem:lambda^pi<lambda^3omega1} below. 
We already proved that $V_\pi^{\fh_\C'}=0$ for $\pi=\pi_{\omega_1}, \pi_{\omega_3}, \pi_{\omega_1+\omega_2}$.
The case $V_{\pi_{\omega_2}}^{\fh_\C'}=0$ follows by  Lemma \ref{lem:omega2-no-esferica}.
This completes the cases $n\geq4$. 

We now assume $n=3$, thus $N=14$. 
From \cite[pp.\ 255]{LieART}, we have that 
$\pi_{\omega_4}|_{\fh_\C'} \simeq \sigma_{\eta_2}\oplus \sigma_{\eta_1+\eta_2}\oplus \sigma_{2\eta_2}\oplus \sigma_{2\eta_1+\eta_2}\oplus \sigma_{\eta_1+\eta_2+\eta_3}\oplus \sigma_{4\eta_1}$, 
$\pi_{\omega_{i}}|_{\fh_\C'} \simeq \sigma_{\eta_1+\eta_2}$ 
and
$\pi_{\omega_1+\omega_{i}}|_{\fh_\C'} \simeq \sigma_{\eta_1}\oplus \sigma_{\eta_3}\oplus \sigma_{3\eta_1}\oplus \sigma_{\eta_1+\eta_2}\oplus \sigma_{\eta_2+\eta_3}\oplus \sigma_{2\eta_1+\eta_3}\oplus \sigma_{\eta_1+2\eta_2}$ for any $i=6,7$. 
Since $\sigma_0=1_{\fh_\C'}$ does not occur in any of them, ${V_{\pi_{\omega_4}}^{\fh_\C'}=V_{\pi_{\omega_{i}}}^{\fh_\C'}= }V_{\pi_{\omega_1+\omega_{i}}}^{\fh_\C'}=0$ by Lemma~\ref{lem3:dimV_pi^H}.

We now compute the multiplicity of $\lambda_1(G/H,g_{\st})=\lambda^{\pi_{3\omega_1}}$ in $\Spec(G/H,g_{\st})$, which is given by 
$
\sum_{\pi} 
\dim V_{\pi}\cdot\dim V_{\pi}^H,
$
where the sum is over $\pi\in\widehat G_H$ satisfying $\lambda^\pi=\lambda^{\pi_{3\omega_1}}$. 
If $n\geq4$, then $N\geq27$ and Lemma~\ref{lem:lambda^pi<lambda^3omega1} implies that the multiplicity is 
	$
	\dim V_{\pi_{3\omega_1}}\cdot \dim V_{\pi_{3\omega_1}}^H 
	=\dim V_{\pi_{3\omega_1}}
	=\frac{N+4}{N-2}\binom{N}{3}
	$.
If $n=3$, then $N=14$ and Lemma~\ref{lem:lambda^pi<lambda^3omega1} ensures that the multiplicity is $\dim V_{\pi_{3\omega_1}}+\dim V_{\pi_{\omega_5}}\dim V_{\pi_{\omega_5}}^H =546+2002$, since $\dim V_{\pi_{\omega_5}}^H=1$ by \cite[pp.\ 255]{LieART}. 
\end{proof}

\begin{lemma}\label{lem:lambda^pi<lambda^3omega1}
Let $\Lambda\in\PP^+(\so(N)_\C)$ for some integer $N$. 

If $N\geq 26$, we have  $\lambda^{\pi_{\Lambda}}<\lambda^{\pi_{3\omega_1}}$ if and only if $\Lambda$ is one of 
$0$, $\omega_1$, $2\omega_1$, $\omega_2$, $\omega_1+\omega_2$, $\omega_3$.
 
If $N=14$, we have   $\lambda^{\pi_{\Lambda}} <\lambda^{\pi_{3\omega_1}}$ if and only if  $\Lambda$ is one of 
$0$, $\omega_1$, $2\omega_1$, $\omega_2$, $\omega_1+\omega_2$, $\omega_3$, $\omega_4$, $\omega_6$, $\omega_7$, $\omega_1+\omega_6$, $\omega_1+\omega_7$. 

If $N=20$, we have  $\lambda^{\pi_{\Lambda}} <\lambda^{\pi_{3\omega_1}}$ if and only if  $\Lambda$ is one of 
$0$, $\omega_1$, $2\omega_1$, $\omega_2$, $\omega_1+\omega_2$, $\omega_3$,  $\omega_9$, $\omega_{10}$. 

If $N\geq 14$, we have that $\lambda^{\pi_{\Lambda}} =\lambda^{\pi_{3\omega_1}}$ if and only if either $\Lambda=3\omega_1$ or $\Lambda=\omega_5$ for $N=14$, 
or $\Lambda=\omega_4$ for $N=19$.
\end{lemma}

\begin{proof}
The proof
is omitted as it parallels that of Lemma~\ref{lem:SOlambda^pi<lambda^omega3}.
\end{proof}

\subsection{Family X}

We set $H'=\SO(n)$ for some $n\geq5$. 
The irreducible representation $\rho:=\sigma_{2\eta_1}$ of $H'$ is of real type, and it has dimension $N:= \frac{(n-1)(n+2)}{2}=\frac{n^2+n-2}{2}$. 
The space $\SO(N)/\rho(H')$ is isotropy irreducible. 
Its universal cover is $G/H$, where
\begin{align}\label{eq:G/H-fliaX}
	G&=
	\begin{cases}
		\Spin(N)&\text{ if }n\equiv 0\pmod 4,\\
		\SO(N)&\text{ if }n\equiv 1,2,3\pmod 4,
	\end{cases}
	&
H&\simeq
\begin{cases}
\SO(n)/\Z_2&\text{ if }n\equiv0\pmod2,\\
\SO(n)&\text{ if }n\equiv1\pmod2.
\end{cases}
\end{align}

This case is very similar to Family VI, so we omit several details. 

\begin{lemma}\label{lem:Sym^2-Lambda^2(2eta_1)}
For $\fh'=\so(n)$, we have that 
	$\displaystyle 
\Sym^2(\sigma_{2\eta_1})
\simeq \sigma_{4\eta_1} 
\oplus \sigma_{2\eta_2}
\oplus \sigma_{2\eta_1}
\oplus \sigma_{0}
$ and $\displaystyle 
\textstyle\bigwedge^2(\sigma_{2\eta_1})
\simeq \sigma_{2\eta_1+\eta_2} 
\oplus \sigma_{\eta_2}	$.
\end{lemma}	  
\begin{proof}
One can check (see, e.g.\ Remark~\ref{rem:KoikeTerada}) that 
$$
\Sym^2(\sigma_{2\eta_1}) \oplus 
\textstyle\bigwedge^2(\sigma_{2\eta_1})
\simeq \sigma_{2\eta_1}\otimes \sigma_{2\eta_1}
\simeq \sigma_{4\eta_1} 
	\oplus \sigma_{2\eta_1}
	\oplus 
	\sigma_{2\eta_1+\eta_1} 
	\oplus \sigma_{2\eta_2} \oplus \sigma_{\eta_2}
	\oplus \sigma_{0}.
$$
As in the proof of Lemma~\ref{lem:Sym^2-Lambda^2(eta_2)}, $\sigma_0$ is in $\Sym^2(\sigma_{2\eta_1})$ because $\sigma_{2\eta_1}$ is of real type (see e.g.\ \cite[\S26.3]{FultonHarris-book}),
$\sigma_{4\eta_1}$ is also in $\Sym^2(\sigma_{2\eta_1})$,  
and
$\sigma_{\eta_2}$ is in $\textstyle\bigwedge^2(\sigma_{2\eta_1})$ by \cite[Thm.~3.3]{PanyushevYakimova08}. 
	
One can verify that the only way to fill the number
\begin{align*}
\dim\textstyle\bigwedge^2(\sigma_{2\eta_1})-\dim\sigma_{\eta_2}  &
= \tbinom{N}{2} -\tbinom{n}{2}
= \tfrac{n^4+2n^3-9n^2-2n+8}{8}
\end{align*}
with the values $\dim\sigma_{2\eta_2} = \frac{n^4-7n^2-6n}{12}$, $\dim\sigma_{2\eta_1+\eta_{2}} =\frac{n^4+2n^3-9n^2-2n+8}{8}$, and  $\dim \sigma_{2\eta_1} = \binom{n+1}{2}-1$ is the one described in the statement. 
\end{proof}

\begin{theorem}
If $G$ and $H$ are as in \eqref{eq:G/H-fliaX}, then  
$$
\lambda_1(G/H,g_{\st})=\lambda^{\pi_{3\omega_1}}
={\frac{3(N+1)}{2(N-2)}}
=\frac32 + \frac{9}{2(N-2)}
.
$$ 
{
Its multiplicity in $\Spec(G/H,g_{\st})$ is equal to $\dim V_{\pi_{3\omega_1}}=\frac{N+4}{N-2}\binom{N}{3}$ for $n\geq 6$, and is $2548$ when $n=5$.
}
\end{theorem}
\begin{proof}
Similarly to \eqref{eq:omega2xomega1|_h} and \eqref{eq:Sym^2xomega1|_h}, we obtain that 
\begin{align*}
\pi_{\omega_2}\otimes\pi_{\omega_1}|_{\fh_\C}&
\simeq {\textstyle\bigwedge}^2(\sigma_{2\eta_1}) \otimes\sigma_{2\eta_1}
\simeq (\sigma_{2\eta_1+\eta_2} \oplus \sigma_{\eta_2}) \otimes \sigma_{2\eta_1},
\\ 
\big(\pi_{2\omega_1}\oplus\pi_0\big) \otimes\pi_{\omega_{1}}|_{\fh_\C}
&
\simeq \Sym^2(\sigma_{2\eta_1}) \otimes \sigma_{2\eta_1}
\simeq \big(\sigma_{4\eta_1} \oplus \sigma_{2\eta_2}\oplus \sigma_{2\eta_1} \oplus \sigma_{0}\big) \otimes \sigma_{2\eta_1}
.
\end{align*} 
The last identity in both cases follows from Lemma~\ref{lem:Sym^2-Lambda^2(2eta_1)}.

We observe that $\sigma_0$ occurs exactly once in $(\sigma_{4\eta_1} \oplus \sigma_{2\eta_2} \oplus \sigma_{2\eta_1} \oplus \sigma_{0}) \otimes \sigma_{2\eta_1}$, and does not occur in $(\sigma_{2\eta_1+\eta_2} \oplus \sigma_{\eta_2}) \otimes \sigma_{2\eta_1}$.  
Proceeding analogously to the proof of Theorem~\ref{thm:flia6}, we prove $V_{\pi_{\omega_1}}^{\fh_\C'} = V_{\pi_{\omega_1+\omega_2}}^{\fh_\C'} = V_{\pi_{\omega_3}}^{\fh_\C'}= 0$ and $V_{\pi_{3\omega_1}}^{\fh_\C'}=1$.	

It remains to verify \eqref{eq:sufficient-lambda_1}. 
The set of $\pi\in\widehat G\smallsetminus\{1_G\}$ satisfying $\lambda^\pi<\lambda^{\pi_{3\omega_1}}$ is described in Lemma~\ref{lem:lambda^pi<lambda^3omega1} since $N=14,20$ for $n=5,6$ respectively, and $N\geq 27$ for $n\geq7$. 
The case $V_{\pi_{\omega_2}}^{\fh_\C}=0$ follows by  Lemma \ref{lem:omega2-no-esferica}.
This completes the cases $n\geq7$. 
The finitely many remaining cases for $n=5,6$ can be handled using the tables of branching rules in \cite{LieART}.

The multiplicity of $\lambda_1(G/H,g_{\st})$
can be determined by the same argument as in the proof of Theorem~\ref{thm:flia6} using Lemma~\ref{lem:lambda^pi<lambda^3omega1}, and we leave the computation to the reader.
\end{proof}

\subsection{Family IV}

We set $H'=\Sp(1)\times \SO(n)$ for some $n\geq3$, and
$\rho:= \sigma_{\eta_1}\widehat\otimes\,\sigma_{\eta_1'}'$, which is of symplectic type and has dimension $N:=2n$. 
The space $\Sp(N/2)/\rho(H')$ is isotropy irreducible, with universal cover $G/H$, where 
\begin{align}\label{eq:G/H-fliaIV}
	G&=\Sp(n)/\Z_2,
	&
	H&\simeq 
\begin{cases}
(\Sp(1)/\Z_2)\otimes(\SO(n)/\Z_2)
	&\text{ if }n\equiv0\pmod2,\\
(\Sp(1)/\Z_2)\otimes\SO(n)
&\text{ if }n\equiv1\pmod2. 
	\end{cases}
\end{align}

\begin{lemma}\label{lem-flia4:branchings}
For $\fg_\C$ and $\fh_\C'$ defined as above, we have that 
\begin{align*}
\pi_{2\omega_1}|_{\fh_\C'}
&\simeq \,
 \begin{cases}
 	\sigma_{2\eta_1}\widehat\otimes\,  \sigma_{0}'
 	\, \oplus \, 
 	\sigma_{2\eta_1}\widehat\otimes\,  \sigma_{4\eta_1'}' 
 	\, \oplus \, 
 	\sigma_0\widehat\otimes\,  \sigma_{2\eta_1'}', &\text{ if }n=3, \\
 	\sigma_{2\eta_1}\widehat\otimes\,  \sigma_{0}'
 	\, \oplus \, 
 	\sigma_{2\eta_1}\widehat\otimes\,  \sigma_{2\eta_1'+2\eta_2'}' 
 	\, \oplus \, 
 	\sigma_0\widehat\otimes\,  \sigma_{2\eta_1'}' \, \oplus \, 
 	\sigma_0\widehat\otimes\,  \sigma_{2\eta_2'}', &\text{ if }n=4, \\
 	\sigma_{2\eta_1}\widehat\otimes\,  \sigma_{0}'
 	\, \oplus \, 
 	\sigma_{2\eta_1}\widehat\otimes\,  \sigma_{2\eta_1'}' 
 	\, \oplus \, 
 	\sigma_0\widehat\otimes\,  \sigma_{2\eta_2'}', &\text{ if }n= 5,\\
 	\sigma_{2\eta_1}\widehat\otimes\,  \sigma_{0}'
 	\, \oplus \, 
 	\sigma_{2\eta_1}\widehat\otimes\,  \sigma_{2\eta_1'}' 
 	\, \oplus \, 
 	\sigma_0\widehat\otimes\,  \sigma_{\eta_2'+\eta_3'}', &\text{ if }n= 6,\\
 	\sigma_{2\eta_1}\widehat\otimes\,  \sigma_{0}'
 	\, \oplus \, 
 	\sigma_{2\eta_1}\widehat\otimes\,  \sigma_{2\eta_1'}' 
 	\, \oplus \, 
 	\sigma_0\widehat\otimes\,  \sigma_{\eta_2'}', &\text{ if }n\geq 7,
 \end{cases}	 
\\
\pi_{\omega_2}|_{\fh_\C'}
&\simeq  \, 
 \begin{cases}
		\sigma_{0}\widehat\otimes\,  \sigma_{4\eta_1'}'
	\, \oplus \,
	\sigma_{2\eta_1}\widehat\otimes\,  \sigma_{2\eta_1'}'
	, &\text{ if }n=3, \\
	\sigma_{0}\widehat\otimes\, \sigma_{2\eta_1'+2\eta_2'}' 
	\, \oplus \, 
	\sigma_{2\eta_1}\widehat\otimes\,  \sigma_{2\eta_1'}'
	\, \oplus \, 
	\sigma_0\widehat\otimes\,  \sigma_{2\eta_2'}', &\text{ if }n=4, \\
	\sigma_{0}\widehat\otimes\, \sigma_{2\eta_1'}' 
	\, \oplus \, 
	\sigma_{2\eta_1}\widehat\otimes\,  \sigma_{2\eta_2'}', &\text{ if }n= 5,\\
	\sigma_{0}\widehat\otimes\, \sigma_{2\eta_1'}' 
	\, \oplus \, 
	\sigma_{2\eta_1}\widehat\otimes\,  \sigma_{\eta_2'+\eta_3'}', &\text{ if }n= 6,\\
	\sigma_{0}\widehat\otimes\, \sigma_{2\eta_1'}' 
	\, \oplus \, 
	\sigma_{2\eta_1}\widehat\otimes\,  \sigma_{\eta_2'}', &\text{ if }n\geq 7,
\end{cases}
\\
\pi_{\omega_3}|_{\fh_\C'}
&\simeq \,
	 \begin{cases}
		\sigma_{3\eta_1}\widehat\otimes\,  \sigma_{0}'
		\, \oplus \,
		\sigma_{\eta_1}\widehat\otimes\,  \sigma_{4\eta_1'}'
		, &\text{ if }n=3, \\
		\sigma_{3\eta_1}\widehat\otimes\,  \sigma_{\eta_1'+\eta_2'}'
		\, \oplus \,
		\sigma_{\eta_1}\widehat\otimes\,  \sigma_{3\eta_1'+\eta_2'}'
		\, \oplus \,
		\sigma_{\eta_1}\widehat\otimes\,  \sigma_{\eta_1'+3\eta_2'}', &\text{ if }n=4, \\
		\sigma_{3\eta_1}\widehat\otimes\,  \sigma_{2\eta_2'}'
		\, \oplus \,
		\sigma_{\eta_1}\widehat\otimes\,  \sigma_{\eta_1'+2\eta_2'}', &\text{ if }n= 5,\\
		\sigma_{3\eta_1}\widehat\otimes\,  \sigma_{2\eta_2'}'
		\, \oplus \,
		\sigma_{3\eta_1}\widehat\otimes\,  \sigma_{2\eta_3'}'
		\, \oplus \,
		\sigma_{\eta_1}\widehat\otimes\,  \sigma_{\eta_1'+\eta_2'+\eta_3'}', &\text{ if }n= 6,\\
			\sigma_{3\eta_1}\widehat\otimes\,  \sigma_{2\eta_3'}'
		\, \oplus \,
		\sigma_{\eta_1}\widehat\otimes\,  \sigma_{\eta_1'+\eta_2'}', &\text{ if }n= 7,\\
			\sigma_{3\eta_1}\widehat\otimes\,  \sigma_{\eta_3'+\eta_4'}'
		\, \oplus \,
		\sigma_{\eta_1}\widehat\otimes\,  \sigma_{\eta_1'+\eta_2'}', &\text{ if }n= 8,\\
		\sigma_{3\eta_1}\widehat\otimes\,  \sigma_{\eta_3'}'
		\, \oplus \,
		\sigma_{\eta_1}\widehat\otimes\,  \sigma_{\eta_1'+\eta_2'}', &\text{ if }n\geq 9.
	\end{cases}
	\end{align*}
\end{lemma}

\begin{proof}
We provide a proof only for $n\geq9$. 
The cases $n=3,\dots,8$ can be checked directly, e.g.,\ by using \cite{LieART} or \cite{Sage}. 
	
From Remarks~\ref{rem:ExtSym-standard} and \ref{rem:ExtSym(suma)o(tensor)}, it follows that 
\begin{align*}
\pi_{2\omega_1}|_{\fh_\C'}
&= {\Sym^2}(\pi_{\omega_1}) |_{\fh_\C'}
= {\Sym^2}(\pi_{\omega_1}|_{\fh_\C'}) 
\simeq {\Sym}^2(\rho) 
= {\Sym}^2(\sigma_{\eta_1}\widehat\otimes \, \sigma_{\eta_1'}')\\
		&\simeq {\Sym}^2(\sigma_{\eta_1})\widehat\otimes \, \Sym^2(\sigma_{\eta_1'}')
		\, \oplus \,  {\textstyle\bigwedge}^2(\sigma_{\eta_1})\widehat\otimes \, {\textstyle\bigwedge}^2(\sigma_{\eta_1'}')\\
		&\simeq \sigma_{2\eta_1}\widehat\otimes \, (\sigma_{0}'\oplus \sigma_{2\eta_1'}') \, \oplus \, \sigma_0\widehat\otimes \, \sigma_{\eta_2'}',
\\
\pi_0\oplus \pi_{\omega_2}|_{\fh_\C'}&
= {{\textstyle\bigwedge}^2}(\pi_{\omega_1}) |_{\fh_\C'}
= {{\textstyle\bigwedge}^2}(\pi_{\omega_1}|_{\fh_\C'})
\simeq {{\textstyle\bigwedge}}^2(\rho) 
= {\textstyle\bigwedge^2}(\sigma_{\eta_1}\widehat\otimes \,\, \sigma_{\eta_1'}')\\
&\simeq {{\textstyle\bigwedge}}^2(\sigma_{\eta_1})\widehat\otimes \, \Sym^2(\sigma_{\eta_1'}')
		\, \oplus \,  {\Sym}^2(\sigma_{\eta_1})\widehat\otimes \,{\textstyle\bigwedge}^2(\sigma_{\eta_1'}')
\\&
\simeq 
	\sigma_{0}\widehat\otimes \, \sigma_{0}' 
	\, \oplus \, 
	\sigma_{0}\widehat\otimes \, \sigma_{2\eta_1'}' 
	\, \oplus \,
	\sigma_{2\eta_1}\widehat\otimes \, \sigma_{\eta_2'}'. 
	\end{align*}
This establishes the first two identities.

Using Remark~\ref{rem:KoikeTerada} several times, we obtain that
\begin{align*}
\big(\pi_{\omega_3} \oplus \pi_{\omega_1+\omega_2} \oplus \pi_{\omega_1} \big) |_{\fh_\C'}&
\simeq \big( \pi_{\omega_2}\otimes\pi_{\omega_{1}}\big) |_{\fh_\C'}
= (\pi_{\omega_2}|_{\fh_\C'}) \otimes (\pi_{\omega_{1}}|_{\fh_\C'})
\\ & 
\simeq 
\big(
	\sigma_{0}\widehat\otimes\,  \sigma_{2\eta_1'}' 
	\, \oplus \, 
	\sigma_{2\eta_1}\widehat\otimes\, \sigma_{\eta_2'}'
\big)
\otimes 
\big( 
	\sigma_{\eta_1}\widehat\otimes\,\sigma_{\eta_1'}' 
\big) 
\\& 
\simeq  
	2\,\sigma_{\eta_1}\widehat\otimes\, \sigma_{\eta_1'}' 
	\, \oplus \,
	2\,\sigma_{\eta_1}\widehat\otimes\, \sigma_{\eta_1'+\eta_2'}' 
	\, \oplus \,
	\sigma_{\eta_1}\widehat\otimes\,  \sigma_{3\eta_1'}' 
\\
&\quad \oplus
	\sigma_{ \eta_1}\widehat\otimes\,  \sigma_{\eta_3'}' 
	\, \oplus\, 
	\sigma_{3\eta_1}\widehat\otimes\,  \sigma_{\eta_1'}' 
	\, \oplus\, 
	\sigma_{3\eta_1}\widehat\otimes\,  \sigma_{\eta_1'+\eta_2'}' 
	\, \oplus\, 
	\sigma_{3\eta_1}\widehat\otimes\,  \sigma_{\eta_3'}',
\\
\big(\pi_{3\omega_1} \oplus \pi_{\omega_1+\omega_2}
	\oplus \pi_{\omega_1} \big)|_{\fh_\C'} &
=\big(\pi_{2\omega_1}\otimes\pi_{\omega_{1}}\big)|_{\fh_\C'}
=(\pi_{2\omega_1}|_{\fh_\C'}) \otimes (\pi_{\omega_{1}}|_{\fh_\C'})
\\ & 
\simeq \big(\sigma_{2\eta_1}\widehat\otimes\,  \sigma_{0}'\oplus \sigma_{2\eta_1}\widehat\otimes\, \sigma_{2\eta_1'}' \, \oplus \, \sigma_0\widehat\otimes\,  \sigma_{\eta_2'}'\big) \otimes \big( \sigma_{\eta_1}\widehat\otimes\,\sigma_{\eta_1'}'\big)
\\&
\simeq  
	3\,\sigma_{\eta_1}\widehat\otimes\, \sigma_{\eta_1'}' 
	\, \oplus\, 
	2\,\sigma_{\eta_1}\widehat\otimes\, \sigma_{\eta_1'+\eta_2'}' 
	\, \oplus\, 
	\sigma_{\eta_1}\widehat\otimes\, \sigma_{3\eta_1'}'
\\
&\quad 
\oplus
	\sigma_{\eta_1}\widehat\otimes\, \sigma_{\eta_3'}' 
	\, \oplus\, 
	2\,\sigma_{3\eta_1}\widehat\otimes\, \sigma_{\eta_1'}' 
	\, \oplus\, 
	\sigma_{3\eta_1}\widehat\otimes\, \sigma_{\eta_1'+\eta_2'}' 
	\, \oplus\, 
	\sigma_{3\eta_1}\widehat\otimes\, \sigma_{3\eta_1'}'.
\end{align*}
The term $\sigma_{3\eta_1}\widehat\otimes \, \sigma_{\eta_3'}'$ occurs in the first identity but not in the second one, thus it must appear in $\pi_{\omega_3}|_{\fh_\C'}$. 
One can check that the only way to fill 
\begin{equation*}
\dim {\pi_{\omega_3}} 
- \dim \sigma_{3\eta_1}\widehat\otimes\, \sigma_{\eta_3'}'
= \left( \tbinom{2n}{3}-2n\right)
-4\tbinom{n}{3}
= \tfrac{2n(n-2)(n+2)}{3}
\end{equation*}
with the values 
$
2\dim \sigma_{\eta_1}\widehat\otimes \, \sigma_{\eta_1'}'$,
$2\dim \sigma_{\eta_1}\widehat\otimes \, \sigma_{\eta_1'+\eta_2'}'$,
$\dim \sigma_{\eta_1}\widehat\otimes \, \sigma_{3\eta_1'}'$, 
$\dim \sigma_{\eta_1}\widehat\otimes  \,\sigma_{\eta_3'}'$, 
$\dim \sigma_{3\eta_1}\widehat\otimes \, \sigma_{\eta_1'}'$,
$\dim \sigma_{3\eta_1}\widehat\otimes \, \sigma_{\eta_1'+\eta_2'}'
$ is with 
$\dim\sigma_{\eta_1}\widehat\otimes \,\sigma_{\eta_1'+\eta_2'}'$, and the proof is complete. 
\end{proof}

\begin{theorem}\label{thm:flia4}
If $G$ and $H$ are as in \eqref{eq:G/H-fliaIV}, then  
$$
\lambda_1(G/H,g_{\st})
=\lambda^{\pi_{2\omega_2}} 
=\frac{2n+1}{n+1}
=2-\frac{1}{n+1}
.
$$ 
Its multiplicity 
is equal to $\dim V_{\pi_{2\omega_2}}=\frac{4n^4-7n^2+3n}{3}$ for $n\neq 4$ and to $2\cdot \dim V_{\pi_{2\omega_2}}=616$ for $n= 4$.
\end{theorem}

\begin{proof}
By Lemma~\ref{lem-flia4:branchings} and Remark~\ref{rem:rhoxrho^*},  $1_{\fh_\C'}=\sigma_0\widehat\otimes \sigma_{0}'$ occurs in $\big(\pi_{\omega_2} \otimes\pi_{\omega_2} \big)|_{\fh_\C'}$ twice for $n\neq4$ and three times for $n=4$.
Indeed, its two (respectively three) irreducible constituents are self-adjoint and non-equivalent. 
Lemma~\ref{lem3:dimV_pi^H} implies $\dim V_{\pi_{\omega_2} \otimes \pi_{\omega_2}}^{\fh_\C'}$ is equal to $2$ (resp.\ $3$). 
Proceeding similarly, we obtain the following identities at the left, and  Remark~\ref{rem:KoikeTerada} ensures the tensor decompositions at the right: 
\begin{align*}
	\dim V_{\pi_{\omega_2}\otimes \pi_{\omega_2}}^{\fh_\C'} &=
\begin{cases}
	2
	&\text{ if }n\neq 4,\\
	3
	&\text{ if }n=4. 
\end{cases},
 &
	\pi_{\omega_2}\otimes\pi_{\omega_{2}}
	&\simeq \pi_{\omega_4}
	\oplus \pi_{\omega_1+\omega_3}
	\oplus \pi_{2\omega_2}
	\oplus \pi_{2\omega_1}
	\oplus \pi_{\omega_2}
	\oplus \pi_{0},
\\
\dim V_{\pi_{2\omega_1}\otimes \pi_{2\omega_1}}^{\fh_\C'} &=
\begin{cases}
	3
	&\text{ if }n\neq 4,\\
	4
	&\text{ if }n=4. 
\end{cases},&
	\pi_{2\omega_1}\otimes\pi_{2\omega_{1}}
	&\simeq \pi_{4\omega_1}
	\oplus \pi_{2\omega_1+\omega_2}
	\oplus \pi_{2\omega_2}
	\oplus \pi_{2\omega_1}
	\oplus \pi_{\omega_2}
	\oplus \pi_{0},
\\
\dim V_{\pi_{2\omega_1}\otimes \pi_{\omega_2}}^{\fh_\C'} &=0,&
	\pi_{2\omega_1}\otimes\pi_{\omega_{2}}
	&\simeq  \pi_{2\omega_1+\omega_2} \oplus \pi_{\omega_1+\omega_3}
	\oplus \pi_{2\omega_1}
	\oplus \pi_{\omega_2},
\\
\dim V_{\pi_{\omega_3}\otimes \pi_{\omega_1}}^{\fh_\C'} &=0,&
	\pi_{\omega_3}\otimes\pi_{\omega_{1}}
	&\simeq  \pi_{\omega_4} \oplus \pi_{\omega_1+\omega_3}
	\oplus \pi_{\omega_2},
\\
\dim V_{\pi_{2\omega_1}\otimes \pi_{\omega_1}}^{\fh_\C'} &=0, &
	\pi_{2\omega_1}\otimes\pi_{\omega_{1}}
	& \simeq   \pi_{3\omega_1} \oplus \pi_{\omega_1+\omega_2}
	\oplus \pi_{\omega_1}.
\end{align*}
It follows from the last three rows that $V_\pi^{\fh_\C'}=0$ for all irreducible representations $\pi$ occurring in the corresponding tensors. 
Furthermore, the first two rows imply
$
1=\dim V_{\pi_{2\omega_2}}^{\fh_\C'}
$
and 
$2=\dim V_{\pi_{4\omega_1}}^{\fh_\C'}\oplus \dim V_{\pi_{2\omega_2}}^{\fh_\C'}$, so $\dim V_{\pi_{4\omega_1}}^{\fh_\C'}=1$ for $n\neq 4$, and $\dim V_{\pi_{2\omega_2}}^{\fh_\C'}=2$ and  $\dim V_{\pi_{4\omega_1}}^{\fh_\C'}=1$ for $n=4$. 
In addition, $V_{\pi_{\omega_3}}^{\fh_\C'}=0$ follows immediately from Lemma~\ref{lem-flia4:branchings}.

We next check \eqref{eq:sufficient-lambda_1} to ensure $\lambda_1(G/H,g_{\st})=\lambda^{\pi_{2\omega_2}}$. 
Lemma~\ref{lem:Sp-lambda^pi<lambda^2omega2} below lists all non-zero $\Lambda\in \PP^+(G)=\PP^+(\Sp(n)/\Z_2)$ such that $\lambda^{\pi_{\Lambda}} <\lambda^{\pi_{2\omega_2}}$. 
We have already shown that $V_{\pi_{\Lambda}}^{\fh_\C'}=0$ (so $\pi_{\Lambda}\notin \widehat G_H$) for all these $\Lambda$ except for $\Lambda=\omega_6$ when $n=6$. 
By \cite[pp.\ 345]{LieART}, 
$\pi_{\omega_6}|_{\fh_\C'}\simeq
	\sigma_{2\eta_1}\widehat\otimes \sigma_{2\eta_2'+2\eta_3'}'
	\oplus
	\sigma_{4\eta_1}\widehat\otimes \sigma_{2\eta_1'}'
	\oplus
	\sigma_{6\eta_1}\widehat\otimes \sigma_{0}'
	\oplus
	\sigma_{0}\widehat\otimes \sigma_{4\eta_2'}'
	\oplus
	\sigma_{0}\widehat\otimes \sigma_{4\eta_3'}'
$, thus $V_{\pi_{\omega_6}}^{\fh_\C'}=0$.

Theorem~\ref{thm:Spec(standard)} and Lemma~\ref{lem:Sp-lambda^pi<lambda^2omega2} imply that the multiplicity of $\lambda_1(G/H,g_{\st})$ in $\Spec(G/H,g_{\st})$ is equal to 
$
\dim V_{\pi_{2\omega_2}}\cdot \dim V_{\pi_{2\omega_2}}^H = \dim V_{\pi_{2\omega_2}}= \frac{4n^4-7n^2+3n}{3}
$ 
for any  $n\neq  7$. 
For the remaining case $n=7$, since $\dim V_{\pi_{\omega_6}}^H=0$ by \cite[pp.\ 350]{LieART}, the multiplicity equals $\dim V_{\pi_{2\omega_2}}$, and the proof is complete. 
\end{proof}

\begin{lemma}\label{lem:Sp-lambda^pi<lambda^2omega2}
Let $n$ be an integer $\geq3$. 
For $\Lambda\in\PP^+(\Sp(n)/\Z_2)$, we have  $\lambda^{\pi_{\Lambda}} <\lambda^{\pi_{2\omega_2}}$ if and only if $\Lambda$ is one of 
$0$, $2\omega_1$, $\omega_2$, $\omega_1+\omega_3$, 
and, in addition,  
$\omega_4$ if $n\geq 4$, or
$\omega_6$ if $n=6$.
Moreover, $\lambda^{\pi_{\Lambda}} =\lambda^{\pi_{2\omega_2}}$ if and only if  $\Lambda=2\omega_2$, or $\Lambda=\omega_6$ if $n=7$.
\end{lemma}

\begin{proof}
Let $\Lambda=\sum_{i=1}^na_i\ee_i\in\PP^+(\Sp(n)/\Z_2)$. Then
$a_i\in\Z$ for all $i$, 
$a_1\geq a_2\geq\dots\geq a_n$, and
$\sum_{i=1}^m a_i$ is even (see e.g.\ \cite[Rmk.~2.8]{LauretTolcachier-nu-stab}). 
A direct computation using \eqref{eq:Freudenthal} shows that 
$$
\lambda^{\pi_{\Lambda}}=\frac{1}{4(n+1)}\sum_{i=1}^n a_i(a_i+2n+2-2i).
$$ 
In particular, 
$
\lambda^{\pi_{2\omega_2}} 
= \lambda^{\pi_{2\ee_1+2\ee_2}}
=\frac{8n+4}{4(n+1)}
=\frac{2n+1}{n+1}
$.

Clearly, if $a_2\geq2$, $\lambda^{\pi_{\Lambda}} \geq \frac{8n+4}{4(n+1)}= 
\lambda^{\pi_{2\omega_2}}$ with equality only if $\Lambda=2(\ee_1+\ee_2)=2\omega_2$. 

If $a_2=0$, then $\Lambda=2a\ee_1$ with $a\in\Z_{\geq0}$, and a straightforward computation shows that $\lambda^{\pi_{\Lambda}}\leq \lambda^{\pi_{2\omega_2}}$ only if $\Lambda=0$ or $\Lambda=2\ee_1=2\omega_1$.

We now assume $a_2=1$. 
Since 
$
\lambda^{\pi_{\Lambda}}
\geq \frac{a_1(a_1+2n) + 2n-1}{4(n+1)} 
,
$
it follows that $\lambda^{\pi_{\Lambda}}> \lambda^{\pi_{2\omega_2}}$ whenever $a_1\geq3$. 
If $a_1=2$, 
$
4(n+1)\lambda^{\pi_{\Lambda}}=
	6n+3
	+\sum_{i=3}^n a_i(a_i+2n+2-2i)	
	\leq 8n+4=4(n+1)\lambda^{\pi_{2\omega_2}}
$ 
if and only if  
$\Lambda=2\ee_1+\ee_2+\ee_3$, and the equality cannot occur.

We now assume in addition that $a_1=1$, thus $\Lambda=\ee_1+\dots+\ee_{2p}=\omega_{2p}$ for some $p\leq \lfloor\tfrac n2\rfloor$. 
We have 
$
4(n+1)\lambda^{\pi_{\omega_{2p}}}=
	2p(2n+2-2p)
$. 
If $2p=4$, then $n\geq4$ and $4(n+1)\lambda^{\pi_{\omega_{2p}}}= 8n-8<8n+4=4(n+1)\lambda^{\pi_{2\omega_2}}$. 
If $2p=6$, then $n\geq6$ and 
\begin{align*}
4(n+1)\lambda^{\pi_{\omega_{2p}}}= 6(2n-4)\leq 8n+4=4(n+1)\lambda^{\pi_{2\omega_2}}
\iff n\leq 7,
\end{align*} 
with equality attained only if $n=7$.
If $2p\geq 8$, then $n\geq8$ and $4(n+1)\lambda^{\pi_{\omega_{2p}}}= 8(2n-6)<8n+4=4(n+1)\lambda^{\pi_{2\omega_2}}
\iff 8n<52$, a contradiction. 
\end{proof}

\subsection{Family VIII} 

We set $H'=\Sp(1)\times \Sp(n)$ for some $n\geq2$ and
$\rho= \sigma_{\eta_1}\widehat\otimes\sigma_{\eta_1'}'$, which is of real type and has dimension $N:=4n$. 
The space $\SO(N)/\rho(H')$ turns out to be isotropy irreducible with universal cover given by $G/H$, where 
\begin{align}\label{eq:G/H-flia8}
	G&=\SO(4n)/\Z_2,
	&
	H&\simeq 
	(\Sp(1)/\Z_2)\otimes(\Sp(n)/\Z_2).
\end{align}

The procedure for this family is very similar to the one in Family IV, so we will omit most of the details, including the proof of the next lemma.

\begin{lemma}\label{lem-flia8:branchings}
For $\fg_\C$ and $\fh_\C'$ defined as above, we obtain the following decompositions: 
\begin{align*}
\pi_{2\omega_1}|_{\fh_\C'}
&\simeq  
	\sigma_{2\eta_1}\widehat\otimes\, \sigma_{2\eta_1'}' 
	\, \oplus \, 
	\sigma_0\widehat\otimes\, \sigma_{\eta_2'}',
\\
\pi_{\omega_2}|_{\fh_\C'}
&\simeq  
	\sigma_{0}\widehat\otimes\,  \sigma_{2\eta_1'}' 
	\, \oplus \,
	\sigma_{2\eta_1}\widehat\otimes \, \sigma_{\eta_2'}' 
	\, \oplus \, 
	\sigma_{2\eta_1}\widehat\otimes\, \sigma_{0}',
\\
\pi_{3\omega_1}|_{\fh_\C'}
&\simeq 
	\sigma_{3\eta_1}\widehat\otimes\, \sigma_{3\eta_1'}'
	\, \oplus \, 
	\sigma_{\eta_1}\widehat\otimes\, \sigma_{\eta_1'+\eta_2'}'.
	\end{align*}
\end{lemma}

\begin{theorem}\label{thm:flia8}
If $G$ and $H$ are as in \eqref{eq:G/H-flia8}, then  
$$
\lambda_1(G/H,g_{\st})
=\lambda^{\pi_{\ee_1+\ee_2+\ee_3+\ee_4}}
=  \frac{4(n-1)}{2n-1}
=  2-\frac{2}{2n-1}
.
$$ 
Its multiplicity in $\Spec(G/H,g_{\st})$ equals $ \dim V_{\pi_{\omega_4}} =\binom{4n}{4}$ for $n\geq 3$ and $\dim V_{\pi_{\ee_1+\ee_2+\ee_3+\ee_4}}=35$ for $n=2$.
\end{theorem}

\begin{proof}
First assume that $n\geq3$.
One can check using Lemma~\ref{lem-flia4:branchings}, Remark~\ref{rem:rhoxrho^*}, Lemma~\ref{lem3:dimV_pi^H}, and Remark~\ref{rem:KoikeTerada} that 
\begin{align*}
\dim V_{\pi_{\omega_2}\otimes \pi_{\omega_2}}^{\fh_\C'} &=3, &
	\pi_{\omega_2}\otimes\pi_{\omega_{2}}
	&\simeq 
	\pi_{\omega_4}
	\oplus \pi_{\omega_1+\omega_3}
	\oplus \pi_{2\omega_2}
	\oplus \pi_{2\omega_1}
	\oplus \pi_{\omega_2}
	\oplus \pi_{0},
\\
\dim V_{\pi_{2\omega_1}\otimes \pi_{2\omega_1}}^{\fh_\C'} &=2,&
	\pi_{2\omega_1}\otimes\pi_{2\omega_{1}}
	&\simeq 
	\pi_{4\omega_1}
	\oplus \pi_{2\omega_1+\omega_2}
	\oplus \pi_{2\omega_2}
	\oplus \pi_{2\omega_1}
	\oplus \pi_{\omega_2}
	\oplus \pi_{0},
\\
\dim V_{\pi_{2\omega_1}\otimes \pi_{\omega_2}}^{\fh_\C'} &=0,&
	\pi_{2\omega_1}\otimes\pi_{\omega_{2}}
	&\simeq  
	\pi_{2\omega_1+\omega_2} 
	\oplus \pi_{\omega_1+\omega_3}
	\oplus \pi_{2\omega_1}
	\oplus \pi_{\omega_2},
\\
\dim V_{\pi_{3\omega_1}\otimes \pi_{\omega_1}}^{\fh_\C'} &=0,&
	\pi_{\omega_3}\otimes\pi_{\omega_{1}}
	&\simeq 
	\pi_{4\omega_1} 
	\oplus \pi_{2\omega_1+\omega_2}
	\oplus \pi_{2\omega_1},
\\
\dim V_{\pi_{2\omega_1}\otimes \pi_{\omega_1}}^{\fh_\C'} &=0, &
	\pi_{2\omega_1}\otimes\pi_{\omega_{1}}
	& \simeq   \pi_{3\omega_1} \oplus \pi_{\omega_1+\omega_2}
	\oplus \pi_{\omega_1}.
\end{align*}
The last three rows yield $V_\pi^{\fh_\C'}=0$ for every irreducible term appearing in them, and thus the first two rows give  
$
\dim V_{\pi_{\omega_4}}^H =  \dim V_{\pi_{2\omega_2}}^H = 
1$.

Equation \eqref{eq:sufficient-lambda_1} holds by Lemma~\ref{lem:SO-lambda^pi<lambda^omega4} below, so we conclude that $\lambda_1(G/H,g_{\st})=\lambda^{\pi_{\omega_4}}$. 
Moreover,
Theorem~\ref{thm:Spec(standard)} and Lemma~\ref{lem:SO-lambda^pi<lambda^omega4} immediately imply that the multiplicity of $\lambda_1(G/H,g_{\st})$ in $\Spec(G/H,g_{\st})$ equals 
$
\dim V_{\pi_{\omega_4}}\cdot \dim V_{\pi_{\omega_4}}^H = \dim V_{\pi_{\omega_4}} = \dim {\textstyle\bigwedge}^4(\pi_{\omega_1})=\binom{4n}{4}
$.

For $n=2$, one must replace $\omega_4$ by $\ee_1+\dots+\ee_4$, obtaining the same results, including the tensor product formulas from Remark~\ref{rem:KoikeTerada}. From \cite[pp.\ 213]{LieART}, we have $\dim V_{\pi_{2\omega_1}}^H=V_{\pi_{2\omega_3}}^H=0$,  which implies that the multiplicity equals $\dim V_{\pi_{\ee_1+\ee_2+\ee_3+\ee_4}}=\frac{1}{2}\dim {\textstyle\bigwedge}^4(\pi_{\omega_1})=35$, and the proof is complete. 
\end{proof}

\begin{lemma}\label{lem:SO-lambda^pi<lambda^omega4}
Let $n$ be an integer $\geq 2$. 
For $\Lambda\in\PP^+(\SO(4n)/\Z_2)$, we have  $\lambda^{\pi_{\Lambda}} <\lambda^{\pi_{\ee_1+\ee_2+\ee_3+\ee_4}}$ if and only if $\Lambda$ is one of 
$0$, $\omega_2$, and, in addition,
$2\omega_1$ if $n\geq3$.
Moreover,
$\lambda^{\pi_{\Lambda}}
 =\lambda^{\pi_{\ee_1+\ee_2+\ee_3+\ee_4}}$ if and only if  $\Lambda={\ee_1+\ee_2+\ee_3+\ee_4}$, or $\Lambda=2\omega_1$ or $\Lambda=2\omega_3={\ee_1+\ee_2+\ee_3-\ee_4}$ when $n=2$.
\end{lemma}

\begin{proof}
The proof is very similar (and simpler) to the proofs of Lemmas~\ref{lem:SOlambda^pi<lambda^omega3},  \ref{lem:lambda^pi<lambda^3omega1}, and \ref{lem:Sp-lambda^pi<lambda^2omega2}, and is therefore left to the reader. 
\end{proof}

\begin{remark}\label{rem8:e1+..+e4=omega4}
The highest weight $\ee_1+\ee_2+\ee_3+\ee_4$ of $\pi_{\ee_1+\ee_2+\ee_3+\ee_4}= {\textstyle\bigwedge^4}(\pi_{\omega_1})$ in Theorem~\ref{thm:flia8} is equal to $\omega_4$ for $n\geq3$. 
When $n=2$, $\ee_1+\dots+\ee_4=2\omega_4$ and ${\textstyle\bigwedge^4}(\pi_{\omega_1})=\pi_{\ee_1+\dots+\ee_4}\oplus \pi_{\ee_1+\ee_2+\ee_3-\ee_4}$ (see Remark~\ref{rem:ExtSym-standard}). 
\end{remark}

\subsection{Families I and II}
In this subsection, we treat two families simultaneously, since the procedures are very similar. 

We set $H'=\SU(n)$ for some $n\geq3$, 
$\rho_1= \sigma_{\eta_2}$, 
$\rho_2= \sigma_{2\eta_1}$ irreducible representations of $H'$. 
The underlying vector spaces of these representations are $V_{\rho_1}\simeq \textstyle{\bigwedge^2}(\pi_{\omega_1}) \simeq \textstyle{\bigwedge^2}(\C^n)$ and $V_{\rho_2}\simeq \Sym^2(\pi_{\omega_1})\simeq \Sym^2(\C^n)$, thus both are of complex type (since $\rho_i^*\not\simeq\rho_i$ for $i=1,2$) and $N_1:=\dim\rho_1= \binom{n}{2}$, $N_2:=\dim\rho_2= \binom{n+1}{2}$. 
The spaces $\SU(N_i)/\rho_i(H')$ for $i=1,2$ turn out to be isotropy irreducible with universal cover given by $G_i/H_i$, where 
\begin{align}\label{eq:G/H-flia1-2}
	G_i&=\SU(N_i)/\Z_d,
	&
	H_i&\simeq \SU(n)/\Z_n,&
	d&=\begin{cases}
	\tfrac{n}{2}&\text{ if $n$ is even,}\\
	n&\text{ if $n$ is odd}. 
	\end{cases}
\end{align}
For $i=1$, we assume that $n\geq5$ in order to avoid symmetric spaces.

\begin{theorem}\label{thm:flia1-2}
If $G_i$ and $H_i$ are as in \eqref{eq:G/H-flia1-2}, then  
$$
\lambda_1(G_i/H_i,g_{\st})=
\begin{cases}
\lambda^{\pi_{2\omega_1+2\omega_{N_i-1}}}
= \frac{2N_i+2}{N_i}= 2+ \frac{2}{N_i} 
&\quad\text{if $i=1$ and $n\neq 6$, or $i=2$ and $n\neq 3$}, 
\\
\lambda^{\pi_{3\omega_1}}
= \frac{42}{25}=1.68
&\quad\text{if $i=1$ and $n=6$},
\\
\lambda^{\pi_{3\omega_1}}
= \frac{15}{8}=1.875
&\quad\text{if $i=2$ and $n=3$}.
\end{cases}
$$ 
Its multiplicity in $\Spec(G/H,g_{\st})$ equals 
	$\dim V_{\pi_{2\omega_1+2\omega_{N_i-1}}}=\frac{N_i^4+2N_i^3-3N_i^2}{4}$ for the first row, 
	$\dim V_{\pi_{3\omega_1}}+\dim V_{\pi_{3\omega_{14}}}=1360$ for the second row, 
	and $\dim V_{\pi_{3\omega_1}}+\dim V_{\pi_{3\omega_{5}}}=112$ for the third row. 
\end{theorem}

\begin{proof}
The decomposition into irreducible components of the restriction of the adjoint representation 
of $\fg_i$ to $\fh_i'$ is given by the Cartan decomposition $\fg_i=\su(N_i)=\fh_i\oplus\fp_i$, where $\fh_i$ is the adjoint representation of $\fh'$ and $\fp_i$ is an irreducible real representation of $\fh'$ 
($(\fp_1)_\C=\sigma_{\eta_2+\eta_{n-2}}$ and $(\fp_2)_\C=\sigma_{2\eta_1+2\eta_{n-1}}$ by \cite[Thm.~11.1]{Wolf68}). 
Now, Remark~\ref{rem:ExtSym(suma)o(tensor)} implies 
\begin{align*}
\Sym^2(\fg_i)|_{\fh'}&
	= \Sym^2(\fh_i\oplus\fp_i)
	\simeq \Sym^2(\fh_i) \oplus \Sym^2(\fp_i) \oplus \fh_i\otimes\fp_i,
\\
\textstyle\bigwedge^2(\fg_i)|_{\fh'}&
	= \textstyle\bigwedge^2(\fh_i\oplus\fp_i)
	\simeq \textstyle\bigwedge^2(\fh_i) \oplus \textstyle\bigwedge^2(\fp_i) \oplus \fh_i\otimes\fp_i
. 
\end{align*}
Since $(\fh_i)_\C$ and $(\fp_i)_\C$ are non-equivalent complex irreducible representations of $\fh_\C'$ of real type, 
Remark~\ref{rem:rhoxrho^*} ensures that $1_{\fh_\C'}$ occurs once in $\Sym^2(\fh_i)$ and in $\Sym^2(\fp_i)$, and it does not occur in $\textstyle\bigwedge^2(\fh_i)$, $\textstyle\bigwedge^2(\fp_i)$, or $\fh_i\otimes\fp_i$.
Lemma~\ref{lem3:dimV_pi^H} gives 
\begin{align*}\label{eq-flia1-2:Sym^2(ad)^H}
\dim \Sym^2(\pi_{\omega_1+\omega_{N-1}})^{\fh_\C'}&=2,
&
\dim {\textstyle\bigwedge^2} (\pi_{\omega_1+\omega_{N-1}})^{\fh_\C'}&=0
. 
\end{align*}

In general, the adjoint representation $\pi_{\omega_1+\omega_{N-1}}$ of $\su(N)_\C$ satisfies (see e.g.\ \cite[p.~3]{Hannabuss})
\begin{equation*}\label{eq-flia1-2:Sym^2(Ad)}
\begin{aligned}	
\Sym^2(\pi_{\omega_1+\omega_{N-1}})
		&\simeq \pi_{2(\omega_1+\omega_{N-1})} 
		\oplus \pi_{\omega_2+\omega_{N-2}}
		\oplus \pi_{\omega_1+\omega_{N-1}}
		\oplus \pi_0
,
\\
\textstyle\bigwedge^2(\pi_{\omega_1+\omega_{N-1}})
		&\simeq \pi_{\omega_2+2\omega_{N-1}} 
		\oplus \pi_{2\omega_1+\omega_{N-2}}
		\oplus \pi_{\omega_1+\omega_{N-1}} .
\end{aligned}
\end{equation*}
It follows immediately that 
$V_{\pi_{\omega_2+2\omega_{N_i-1}}}^{\fh_\C'}
=V_{\pi_{2\omega_1+\omega_{N_i-2}}}^{\fh_\C'}
=V_{\pi_{\omega_1+\omega_{N_i-1}}}^{\fh_\C'}
=0$ and 
\begin{equation}\label{eq-flia1-2:2=dim+dim+0+1}
2
=\dim V_{\pi_{2(\omega_1+\omega_{N_i-1})}}^{\fh_\C'} 
+\dim V_{\pi_{\omega_2+\omega_{N_i-2}}}^{\fh_\C'}
+0+1
.
\end{equation}
Our next goal is to show that $ V_{\pi_{\omega_2+\omega_{N_i-2}}}^{\fh_\C'}=0$. 

By Pieri's formula (see e.g.~\cite[Prop.~15.25]{FultonHarris-book}), we have that
$
\pi_{\omega_2}\otimes\pi_{\omega_{N_i-2}}
	\simeq \pi_{\omega_2+\omega_{N_i-2}} 
	\oplus \pi_{\omega_1+\omega_{N_i-1}}
	\oplus \pi_0
	. 
$
Then 
\begin{equation}\label{eq-flia1-2:util}
	\dim V_{\pi_{\omega_2+\omega_{N-2}}}^{\fh_\C'}
	= \dim \big(V_{\pi_{\omega_2}}\otimes V_{\pi_{\omega_{N-2}}}\big)^{\fh_\C'}
	-1.
\end{equation}
Furthermore, 
\begin{align*}
(\pi_{\omega_2}\otimes\pi_{\omega_{N-2}})|_{\fh_\C'}&
\simeq \big(
	{\textstyle\bigwedge^2}(\pi_{\omega_1}) 
	\otimes 
	{\textstyle\bigwedge^2}(\pi_{\omega_{1}})^*
\big)|_{\fh_\C'}
\simeq {\textstyle\bigwedge}^2(\rho_i) \otimes {\textstyle\bigwedge^2}(\rho_i)^*
.
\end{align*}
We now claim that ${\textstyle\bigwedge^2}(\rho_i)$ is $\fh_\C'$ irreducible, which implies that $\dim V_{\pi_{\omega_2}\otimes\pi_{\omega_{N-2}}}^{\fh_\C'}=1$ by Remark~\ref{rem:rhoxrho^*}. Thus,  $\dim V_{\pi_{\omega_2+\omega_{N-2}}}^{\fh_\C'}=0$ by \eqref{eq-flia1-2:util}, and consequently $\dim V_{\pi_{2(\omega_1+\omega_{N_i-1})}}^{\fh_\C'} =1$ by \eqref{eq-flia1-2:2=dim+dim+0+1}, which means that $\lambda^{\pi_{2(\omega_1+\omega_{N_i-1})}}\in\Spec(G_i/H_i,g_{\st})$.

We now prove separately, for $i=1$ and $i=2$, the claim that ${\textstyle\bigwedge^2}(\rho_i)$ is $\fh_\C'$-irreducible. 
It follows immediately from  \cite[Ex.~15.32]{FultonHarris-book} that ${\textstyle\bigwedge^2}(\rho_1)
={\textstyle\bigwedge^2}(\sigma_{\eta_2})=\sigma_{\eta_1+\eta_3}$. 
For $i=2$, $\sigma_{\eta_1}\otimes \sigma_{\eta_1} \simeq \sigma_{2\eta_1} \oplus \sigma_{\eta_2}$ by \cite[Prop.~15.25]{FultonHarris-book}, thus 
$\Sym^2(\sigma_{\eta_1})\simeq \sigma_{2\eta_1}$,  ${\textstyle\bigwedge}^2(\sigma_{\eta_1})\simeq \sigma_{\eta_2}$ (see Remark~\ref{rem:rhoxrho^*}), and 
$${\textstyle\bigwedge}^2(\sigma_{\eta_1}\otimes \sigma_{\eta_1}) \simeq {\textstyle\bigwedge}^2(\sigma_{2\eta_1} \oplus \sigma_{\eta_2}).$$
By applying Remark~\ref{rem:ExtSym(suma)o(tensor)} to both sides, we get 
\begin{align*}
2\;\Big(
{\textstyle\bigwedge}^2(\sigma_{\eta_1}) \otimes \Sym^2(\sigma_{\eta_1})
\Big)&
	= {\textstyle\bigwedge}^2(\sigma_{2\eta_1}) \oplus {\textstyle\bigwedge}^2(\sigma_{\eta_2}) \oplus \sigma_{2\eta_1}\otimes \sigma_{\eta_2}
\\ \Longrightarrow \qquad 
	2 (\sigma_{\eta_2} \otimes \sigma_{2\eta_1})&
	= {\textstyle\bigwedge}^2(\sigma_{2\eta_1}) \oplus \sigma_{\eta_1+\eta_3} \oplus \sigma_{\eta_2}\otimes \sigma_{2\eta_1}
\\ \Longrightarrow \qquad 
	\sigma_{\eta_2} \otimes \sigma_{2\eta_1}&
	= {\textstyle\bigwedge}^2(\sigma_{2\eta_1}) \oplus \sigma_{\eta_1+\eta_3} 
.
\end{align*}
Since $\sigma_{\eta_2} \otimes \sigma_{2\eta_1} = \sigma_{\eta_1+\eta_3} \oplus \sigma_{2\eta_1+\eta_{2}}$  (see \cite[Prop.~ 15.25]{FultonHarris-book}), we get ${\textstyle\bigwedge}^2(\rho_2)= {\textstyle\bigwedge}^2(\sigma_{2\eta_1})
\simeq \sigma_{2\eta_1+\eta_{2}}$, as claimed. 

The cases $n\leq 8$ were handled computationally in \cite{LauretTolcachier-nu-stab} (see Remark~\ref{rem:SU-flia1y2-n<9} for further details). 
For the rest of the proof we assume $n \geq 9$. In fact, the same argument also applies for $n=5$ and $n=7$.

Next, we check \eqref{eq:sufficient-lambda_1} to ensure $\lambda_1(G_i/H_i,g_{\st}) =\lambda^{\pi_{2\omega_1+2\omega_{N_i-1}}}$. 
We have already shown that $V_{\pi_\Lambda}^{\fh_\C}=0$ for $\Lambda$ equal to $0$, $\omega_1+\omega_{N_i-1}$, $\omega_2+\omega_{N_1-2}$, $2\omega_1+\omega_{N_1-2}$, or $\omega_2+2\omega_{N_1-1}$. 
Lemma~\ref{lem:SU-flia1-lambda^pi<lambda^2omega1+2omegaN-1} ensures that this is enough to check \eqref{eq:sufficient-lambda_1}, thus $\lambda_1(G_i/H_i,g_{\st}) =\lambda^{\pi_{2\omega_1+2\omega_{N_i-1}}}$. 
Moreover, the second part of Lemma~\ref{lem:SU-flia1-lambda^pi<lambda^2omega1+2omegaN-1} implies that its multiplicity is equal to $\dim V_{\pi_{2\omega_1+2\omega_{N_i-1}}}$. 
\end{proof}

\begin{lemma}\label{lem:SU-flia1-lambda^pi<lambda^2omega1+2omegaN-1}
Let $n$ be an integer $\geq 9$.
Set $d=n$ if $n$ is odd, or $d=\frac n2$ if $n$ is even,
and write $N_1=\frac{n(n-1)}{2}$ and  $N_2=\frac{n(n+1)}{2}$.
Fix $i\in\{1,2\}$.
	
If $\Lambda\in\PP^+(\SU(N_i)/\Z_d)$ satisfies $\lambda^{\pi_{\Lambda}} <\lambda^{\pi_{2\omega_1+2\omega_{N_i-1}}}$, then $\Lambda$ is one of the following:
$0$, $\omega_1+\omega_{N_i-1}$, $\omega_2+\omega_{N_i-2}$, $2\omega_1+\omega_{N_i-2}$, $\omega_2+2\omega_{N_i-1}$.
Moreover, $\lambda^{\pi_{\Lambda}} =\lambda^{\pi_{2\omega_1+2\omega_{N_i-1}}}$ if and only if $\Lambda = 2\omega_1+2\omega_{N_i-1}$.
\end{lemma}

\begin{proof}
We recall that 
\begin{equation}\label{eq:PP(SU(N)/Z_d)}
	\PP(\SU(N_i)/\Z_d)
=\bigcup_{1\leq p\leq N_i:\, d\mid p} \mathbb (\omega_p+ \mathbb A_{N_i-1}),
\end{equation}
where $\mathbb A_{N_i-1}=\{\sum_{j=1}^{N_i} b_j\ee_j: \sum_{j=1}^{N_i}b_j=0,\; b_j\in\Z\;\forall\, j\}$ is the root lattice of $\su(N_i)_\C$, and $\omega_{N_i}=0$ by convention. 

Write $\Lambda= \sum_{j=1}^{N_i-1} b_j\omega_j \in\PP^+(\su(N_i-1)_\C)$. 
One can check by using \eqref{eq:Freudenthal} that
\begin{align*}
\lambda^{\pi_{2\omega_1+2\omega_{N_i-1}}} = \frac{2(N_i+1)}{N_i} \qquad \mbox{ and } \qquad \lambda^{\pi_{\omega_k}} = \frac{k(N_i+1)(N_i-k)}{2N_i^2}
\quad \text{$\forall\,1\leq k\leq N_i-1$}.
\end{align*}

We have  $\lambda^{\pi_{\Lambda}}> \sum_{j=1}^{N_i-1} b_j\lambda^{\pi_{\omega_j}}$ by Remark~\ref{rem:autovalor_suma_dominantes}.
It is a simple matter to check that 
$\lambda^{\pi_{\omega_k}}> \lambda^{\pi_{2\omega_1+2\omega_{N_i-1}}}$ for all $5\leq k \leq N_i-5$ provided $N_i\geq 16$. 
Moreover, one can check that 
$$ \sum_{j=1}^4 (b_j\lambda^{\pi_{\omega_j}}+b_j'\lambda^{\pi_{\omega_{N_i-j}}}) >\lambda^{\pi_{2\omega_1+2\omega_{N_i-1}}}
\quad\text{if }\sum_{j=1}^4 j(b_j+b_j')>4.
$$ 
Consequently, we only have to consider those $\Lambda$'s satisfying $\sum_{j=1}^4 j(b_j+b_j')\leq 4$, which are listed in Table~\ref{table:pesosA_n+omega_p}. 

\begin{table}
\caption{All weights $\Lambda=\sum_{j=1}^{N-1} b_j\omega_j\in\PP^+(\su(N)_\C)$, with $N\geq 16$, satisfying $\ell_{\Lambda}:=\sum\limits_{j=1}^4j(b_j+b_j')\leq 4$.}\label{table:pesosA_n+omega_p}
$
\begin{array}{cccccc}
\ell_{\Lambda} 
& \mathbb A_{N-1} & \omega_1+\mathbb A_{N-1} & \omega_2+\mathbb A_{N-1} & \omega_3+\mathbb A_{N-1} & \omega_4+\mathbb A_{N-1} 
\\ \hline \hline
0 & 
	0
\\ \hline 
1 
	& 
	& 
	\begin{array}{c}
	\omega_1 \\
	\omega_{N-1}
	\end{array}
	& 
	& 
	& 
\\ \hline 
2 & 
	\omega_1+\omega_{N-1}
	&
	&
	\begin{array}{c}
	2\omega_1\\
	2\omega_{N-1}\\
	\omega_2\\
	\omega_{N-2}
	\end{array}
	& 
	& 
\\ \hline 
3 
	& 
	& 
	\begin{array}{c} 
	2\omega_1+\omega_{N-1} \\
	\omega_1+2\omega_{N-1} \\
	\omega_1+\omega_{N-2} \\
	\omega_2+\omega_{N-1}
	\end{array}
	& 
	& 
	\begin{array}{c}
	3\omega_1\\
	3\omega_{N-1}\\
	\omega_3\\
	\omega_{N-3}\\
	\omega_1+\omega_{2}\\
	\omega_{N-1}+\omega_{N-2}
	\end{array}
	& 
\\ \hline 
4 
	& 
	\begin{array}{c}
	2\omega_1+2\omega_{N-1} \\ 
	\omega_2+\omega_{N-2}\\
	2\omega_1+\omega_{N-2} \\ 
	\omega_2+2\omega_{N-1}
	\end{array}
	& 
	& 
	\begin{array}{c}
	3\omega_1+\omega_{N-1}\\
	\omega_1+3\omega_{N-1}\\
	2\omega_2\\
	2\omega_{N-2}\\
	\omega_1+\omega_{N-3}\\
	\omega_3+\omega_{N-1}
	\end{array}
	& 
	& 
	\begin{array}{c}
	4\omega_1\\
	4\omega_{N-1}\\
	\omega_1+\omega_{3}\\
	\omega_{N-1}+\omega_{N-3}\\
	\omega_4\\
	\omega_{N-4}
	\end{array}
\\ \hline 
\end{array}
$
\end{table}

Since $n\geq9$ by hypothesis, we have $d>4$. 
Hence, it is sufficient to check whether  $\lambda^{\pi_{\Lambda}}<\lambda^{\pi_{2\omega_1+2\omega_{N_i-2}}}$ or $\lambda^{\pi_{\Lambda}}=\lambda^{\pi_{2\omega_1+2\omega_{N_i-2}}}$ for those $\Lambda\in\mathbb A_{N_i-1}$ in Table~\ref{table:pesosA_n+omega_p}. 
One can check that $\lambda^{\pi_{\Lambda}}<\lambda^{\pi_{2\omega_1+2\omega_{N_i-2}}}$ in all cases except for $\Lambda=2\omega_1+2\omega_{N_i-1}$ where the equality obviously holds. 
\end{proof}

\begin{remark}\label{rem:SU-flia1y2-n<9}
The cases $n\leq 8$ were omitted in the proof of Theorem~\ref{thm:flia1-2} for brevity reasons.  
Although there is a computational proof of them, here we explain how to fill the details. 

For instance, we assume first $n=7$, thus $d=7$ (and $N_1=21$, $N_2=28$). 
According to the proof of Lemma~\ref{lem:SU-flia1-lambda^pi<lambda^2omega1+2omegaN-1}, since $d=7>4$, Table~\ref{table:pesosA_n+omega_p} ensures that $\lambda^{\pi_{\Lambda}}<\lambda^{\pi_{2\omega_1+2\omega_{N_i-1}}}$ for $\Lambda$ in 
$\{
	\omega_1+\omega_{N_i-1},
	\omega_2+2\omega_{N_i-2},
	2\omega_1+\omega_{N_i-2},
	\omega_2+2\omega_{N_i-1},
\}$.
Furthermore, $V_{\pi_\Lambda}^{\fh_\C'}=0$ was already established for each of these $\Lambda$'s in the proof of Theorem~\ref{thm:flia1-2}, and the proof is complete. 
The case $n=5$ is very similar. 

We now assume $n=8$, thus $d=4$ (and $N_1=28$ and $N_2=36$). 
Following the proof of  Lemma~\ref{lem:SU-flia1-lambda^pi<lambda^2omega1+2omegaN-1}, \eqref{eq:PP(SU(N)/Z_d)} implies that it is also necessary to check that $V_{\pi_{\Lambda}}^{\fh_\C'}=0$ for each $\Lambda$ in the column $\omega_4+\mathbb A_{N_i-1}$ in Table~\ref{table:pesosA_n+omega_p}. 
These finitely many branching rules can be done computationally. 

The cases $n=3,4,6$ are similar to the previous one, although the elements from Table~\ref{table:pesosA_n+omega_p} added depends on $d$ via \eqref{eq:PP(SU(N)/Z_d)}. 
An important difference occurs when $i=1$ and $n=6$ ($N_1=15$), or $i=2$ and $n=3$ ($N_2=6$). 
In both cases $d=3$, and the dual representations $\pi_{3\omega_1}, \pi_{3\omega_{N_i-1}}$ are spherical (i.e.\ $V_{\pi_{3\omega_1}}^{\fh_\C'}\neq0$ and $V_{\pi_{3\omega_{N_i-1}}}^{\fh_\C'}\neq0$) unlike previously. 
To check the multiplicity of $\lambda^{\pi_{3\omega_1}}$ in $\Spec(G_i/H_i,g_{\st})$, one can prove that $\lambda^{\pi_{\Lambda}}\leq \lambda^{\pi_{3\omega_1}}$ only for 
$\Lambda$ in 
$\{0, \omega_3, \omega_{12}, \omega_1+\omega_2, \omega_{13}+\omega_{14}, \omega_1+\omega_{14}\}$ for $i=1$, in 
$\{0, \omega_3, \omega_1+\omega_2, \omega_4+\omega_5, \omega_1+\omega_{5}, \omega_2+\omega_4\}$
for $i=2$, 
and the equality holds only for $3\omega_1$ and $3\omega_{N_i-1}$. 
\end{remark}

\subsection{Family III}
For integers $p\geq q\geq2$ with $p+q\neq 4$, we set $H'=\SU(p)\times\SU(q)$ and the irreducible representation $\rho= \sigma_{\eta_1}\widehat  \otimes \,\sigma_{\eta_1'}'$ of $H'$ which is complex type and has dimension $N:=pq$. 
It turns out that $\SU(pq)/\rho(H')$ is isotropy irreducible with universal cover given by $G/H$, where 
\begin{align}\label{eq:G/H-flia3}
G&=\SU(pq)/\Z_d,
&
H&\simeq (\SU(p)/\Z_p)\times (\SU(q)/\Z_q),&
d&=\op{lcm}(p,q). 
\end{align}

\begin{theorem}\label{thm:flia3}
If $G$ and $H$ are as in \eqref{eq:G/H-flia3}, then  
$$
\lambda_1(G/H,g_{\st})
=\lambda^{\pi_{\omega_2+\omega_{pq-2}}}
=\frac{2pq-2}{pq}=2-\frac{2}{pq} 
.
$$ 
	Its multiplicity in $\Spec(G/H,g_{\st})$ is equal to $\dim V_{\pi_{\omega_2+\omega_{pq-2}}}= \binom{pq}{2}^2-p^2q^2$ for $(p,q)\neq(3,3)$, and 
$\dim V_{\pi_{\omega_2+\omega_{7}}}
+\dim V_{\pi_{3\omega_1}}
+\dim V_{\pi_{3\omega_8}}= 1545$ for $(p,q)=(3,3)$.
\end{theorem}

\begin{proof}
Using Remarks~\ref{rem:ExtSym-standard} and \ref{rem:ExtSym(suma)o(tensor)}, we obtain the following equivalences of $\fh_\C'$-modules: 
\begin{align*}
\big(
	\pi_{\omega_2}\otimes\pi_{\omega_{pq-2}}
\big)|_{\fh_\C'}
&= \big(
	{\textstyle\bigwedge^2}(\pi_{\omega_1}) \otimes {\textstyle\bigwedge^2}(\pi_{\omega_{pq-1}})
\big)|_{\fh_\C'}
\simeq {\textstyle\bigwedge}^2(\rho) \otimes {\textstyle\bigwedge}^2(\rho^*) 
\\
&\simeq {\textstyle\bigwedge}^2 \big(\sigma_{\eta_1}\widehat\otimes \, \sigma_{\eta_1'}'\big)
	\otimes {\textstyle\bigwedge}^2 \big(\sigma_{\eta_{p-1}} \widehat\otimes \, \sigma_{\eta_{q-1}'}'\big)
\\
&\simeq \big({\textstyle\bigwedge}^2(\sigma_{\eta_1})\widehat\otimes \, \Sym^2(\sigma_{\eta_1'}')
\, \oplus \,  \Sym^2(\sigma_{\eta_1})\widehat\otimes\, {\textstyle\bigwedge}^2(\sigma_{\eta_1'}') \big)
\\ & \quad \otimes
\big({\textstyle\bigwedge}^2(\sigma_{\eta_{p-1}})\widehat\otimes\, \Sym^2(\sigma_{\eta_{q-1}'}')
\, \oplus \,  \Sym^2(\sigma_{\eta_{p-1}}) \widehat\otimes\, {\textstyle\bigwedge}^2(\sigma_{\eta_{q-1}'}')\big)
\\
&\simeq 
	\big(\sigma_{\eta_2} \otimes \sigma_{\eta_{p-2}}\big) 
	\widehat\otimes   \big(\sigma_{2\eta_{1}'}'\otimes \sigma_{2\eta_{q-1}'}' \big) 
\, \oplus \, 
	\big(\sigma_{\eta_2} \otimes \sigma_{2\eta_{p-1}}\big) 
	\widehat\otimes   \big(\sigma_{2\eta_{1}'}'\otimes \sigma_{\eta_{q-2}'}' \big)
\\ &\quad 
\oplus 
	\big(\sigma_{2\eta_1} \otimes \sigma_{2\eta_{p-1}}\big) 
	\widehat\otimes    \big(\sigma_{\eta_{2}'}'\otimes \sigma_{\eta_{q-2}'}' \big) 
\, \oplus \, 
	\big(\sigma_{2\eta_1} \otimes \sigma_{\eta_{p-2}}\big) 
	\widehat\otimes    \big(\sigma_{\eta_{2}'}'\otimes \sigma_{2\eta_{q-1}'}' \big)
.
\end{align*}
We have that $\sigma_0$ appears exactly once in 
$\sigma_{\eta_2} \otimes \sigma_{\eta_{p-2}}$ and once in $\sigma_{2\eta_1} \otimes \sigma_{2\eta_{p-1}}$ (see Remark~\ref{rem:rhoxrho^*}).
Similarly, $\sigma_0'$ occurs once in 
$\sigma_{\eta_{2}'}'\otimes \sigma_{\eta_{q-2}'}'$ and in 
$\sigma_{2\eta_{1}'}'\otimes \sigma_{2\eta_{q-1}'}'$. 
Hence, $\sigma_0\widehat\otimes\, \sigma_0'$ appears twice in  $(\pi_{\omega_2}\otimes\pi_{\omega_{pq-2}})|_{\fh_\C'}$, and consequently $\dim V_{\pi_{\omega_2}\otimes\pi_{\omega_{pq-2}}}^{\fh_\C'}=2$ by Lemma~\ref{lem3:dimV_pi^H}.

Since 
$
	\pi_{\omega_2}\otimes\pi_{\omega_{pq-2}}
	\simeq \pi_{\omega_2+\omega_{pq-2}}
	\oplus \pi_{\omega_1+\omega_{pq-1}}
	\oplus \pi_{0}
$ (see e.g.\ \cite[Prop 15.25]{FultonHarris-book}), it follows that  
$$
\dim V_{\pi_{\omega_2+\omega_{pq-2}} }^{\fh_\C'}
= \dim V_{\pi_{\omega_2}\otimes\pi_{\omega_{pq-2}} }^{\fh_\C'}
- \dim V_{\pi_{\omega_1+\omega_{pq-1}} }^{\fh_\C'}
-\dim V_{\pi_{0} }^{\fh_\C'}
=2-0-1
=1. 
$$
Here we use Lemma~\ref{lem:omega2-no-esferica}\eqref{item:Ad^H=0} to ensure $V_{\pi_{\omega_1+\omega_{pq-1}}}^{\fh_\C'}=0$ because $\pi_{\omega_1+\omega_{pq-1}}$ is equivalent to the adjoint representation. 

We have noted that $\pi_{\omega_2+\omega_{pq-2}}\in\widehat G_H$. 
In order to establish $\lambda_1(G/H,g_{\st})=\lambda^{\pi_{\omega_2 +\omega_{pq-2}}}$ we check  \eqref{eq:sufficient-lambda_1}. 
We already proved that $V_{\pi_{\omega_1+\omega_{pq-1}}}^{\fh_\C'}=0$.
When $pq> 16$, Lemma~\ref{lem:SU-flia3-lambda^pi<lambda^omega2+omegaN-2} yields $\lambda_1(G/H,g_{\st}) =\lambda^{\pi_{\omega_2+\omega_{pq-2}}}$ with multiplicity in $\Spec(G/H,g_{\st})$ equal to the dimension of ${\pi_{\omega_2+\omega_{pq-2}}}$. 

The cases $pq \le 16$ were treated computationally in \cite{LauretTolcachier-nu-stab}, although they could also be verified using arguments similar to those above.
We only discuss their multiplicity here.

By Theorem~\ref{thm:Spec(standard)} and Lemma~\ref{lem:SU-flia3-lambda^pi<lambda^omega2+omegaN-2}, the multiplicity of $\lambda_1(G/H,g_{\st})$ in $\Spec(G/H,g_{\st})$ is 
$$
\dim V_{\pi_{\omega_2+\omega_{pq-2}}}
=\dim( {\textstyle\bigwedge^2}(\pi_{\omega_1}) 
	\otimes 
	{\textstyle\bigwedge^2}(\pi_{\omega_{1}})^* )
-\dim(V_{\pi_{\omega_1+\omega_{pq-1}}})
-\dim(V_{\pi_0})
=\tbinom{pq}{2}^2-p^2q^2
$$
if $(p,q)\neq(3,3)$.
For $(p,q)=(3,3)$, 
$\dim V_{\pi_{3\omega_1} }^{\fh_\C'}=\dim V_{\pi_{3\omega_8} }^{\fh_\C'}=1$ and $\dim V_{\pi_{\omega_1+\omega_5} }^{\fh_\C'}=\dim V_{\pi_{\omega_4+\omega_8} }^{\fh_\C'}=0$ by \cite[pp.\ 168]{LieART}, thus the multiplicity of $\lambda_1(G/H,g_{\st})$ 
equals
$\dim V_{\pi_{\omega_2+\omega_{7}}}+ \dim V_{\pi_{3\omega_1}} + \dim V_{\pi_{3\omega_8}} = 1215+165+165=1545,$
and the proof is complete.
\end{proof}

\begin{lemma}\label{lem:SU-flia3-lambda^pi<lambda^omega2+omegaN-2}
Let $p\geq q\geq 2$ be integers with $p+q\neq  4$, and set $d=\op{lcm}(p,q)$. 
Let $\Lambda\in\PP^+(\SU(pq)/\Z_d)$. 
If $\lambda^{\pi_{\Lambda}} <\lambda^{\pi_{\omega_2+\omega_{pq-2}}}$, then $\Lambda=0$ or $\Lambda=\omega_1+\omega_{pq-1}$, or additionally:
$\omega_4$ if $(p,q)=(4,2)$,
$\omega_3,\omega_6$ if $(p,q)=(3,3)$,
$\omega_6$ if $(p,q)=(6,2)$,
$\omega_4,\omega_{12}$ if $(p,q)=(4,4)$.

Moreover, 
when $(p,q)\neq (3,3)$, $\lambda^{\pi_{\Lambda}} =\lambda^{\pi_{\omega_2+\omega_{pq-2}}}$ if and only if $\Lambda = \omega_2+\omega_{pq-2}$. 
When $(p,q)=(3,3)$, $\lambda^{\pi_{\Lambda}} =\lambda^{\pi_{\omega_2+\omega_{7}}}$ if and only if $\Lambda$ is one of  $\omega_2+\omega_{7}$, $3\omega_1$, $3\omega_8$, $\omega_1+\omega_5$, or $\omega_4+\omega_8$.
\end{lemma}

\begin{proof}
The argument is essentially the same as in the proof of Lemma~\ref{lem:SU-flia1-lambda^pi<lambda^2omega1+2omegaN-1}. 
We consider $\Lambda= \sum_{j=1}^{pq-1} b_j\omega_j \in\PP^+(\su(pq-1)_\C)$. 
By \eqref{eq:Freudenthal}, we get $\lambda^{\pi_{\omega_2+\omega_{pq-2}}} = \frac{2(pq-1)}{pq}$ and 
$ \lambda^{\pi_{\omega_k}} = \frac{k(pq+1)(pq-k)}{2p^2q^2}$
for all $1\leq k\leq pq-1$, and also $\lambda^{\pi_{\Lambda}}> \sum_{j=1}^{pq-1} b_j\lambda^{\pi_{\omega_j}}$ by Remark~\ref{rem:autovalor_suma_dominantes}.

We assume $pq\geq16$. 
As in the proof of Lemma~\ref{lem:SU-flia1-lambda^pi<lambda^2omega1+2omegaN-1}, 
$\lambda^{\pi_{\omega_k}}> \lambda^{\pi_{\omega_2+\omega_{pq-2}}}$ for all $5\leq k \leq pq-5$, and 
$$ 
\sum_{j=1}^4 (b_j\lambda^{\pi_{\omega_j}}+b_j'\lambda^{\pi_{\omega_{pq-j}}}) >\lambda^{\pi_{\omega_2+\omega_{pq-2}}}
\quad\text{if }\sum_{j=1}^4 j(b_j+b_j')>4.
$$ 
Consequently, we only need to consider those $\Lambda$ in Table~\ref{table:pesosA_n+omega_p}. 
When $d>4$ (i.e.\ $pq\geq 17$ or $(p,q)=(8,2)$), it suffices to check the inequality $\lambda^{\pi_{\Lambda}}\leq \lambda^{\pi_{\omega_2+\omega_{pq-2}}}$ for $\Lambda$ in the first column. 
For $(p,q)=(4,4)$, we have $d=4$, so in this case it is enough to check the weights in the first and last columns to complete the proof.

For $pq<16$ the lemma can be verified case by case.
For instance, suppose $(p,q)=(6,2)$, so $d=6$.  
A straightforward computation shows that $\lambda^{\pi_{\omega_k}}< \lambda^{\pi_{\omega_2+\omega_{10}}}$ for all $1\leq k\leq 11$, but $2\lambda^{\pi_{\omega_k}}> \lambda^{\pi_{\omega_2+\omega_{10}}}$ for $3\leq k\leq 9$, $3\lambda^{\pi_{\omega_2}}> \lambda^{\pi_{\omega_2+\omega_{10}}}$, and $4\lambda^{\pi_{\omega_1}}> \lambda^{\pi_{\omega_2+\omega_{10}}}$. Hence, we only need to consider those $\Lambda$ satisfying $\sum_{j=1}^4 j(b_j+b_j')\leq 4$, together with $\Lambda=\omega_5$ and $\Lambda=\omega_6$.
Since $d=6$, we have $\PP(\SU(12)/\Z_6)	= \mathbb  A_{11} \cup (\omega_6+ \mathbb A_{11})$. Therefore, one verifies that the only weights satisfying the inequality are $0$, $\omega_1+\omega_{11}$, and $\omega_6$.

We now consider the case $(p,q)=(3,3)$. 
As above, it is sufficient to consider those $\Lambda$ satisfying $\sum_{j=1}^4 j(b_j+b_j')\leq 4$.
Since $d=3$, we have $\PP(\SU(9)/\Z_3)
= \mathbb  A_{8} \cup (\omega_3+ \mathbb A_{8})\cup (\omega_6+ \mathbb A_{8})$. 
In this case, several weights in $(\omega_3+ \mathbb A_{8}) \cup (\omega_6+ \mathbb A_{8})$ turn out to have the same eigenvalue as $\omega_2+\omega_{7}$, namely, $3\omega_1, 3\omega_8, \omega_1+\omega_7,\omega_4+\omega_8$.
\end{proof}

\subsection{Isolated cases}

For the non-symmetric simply connected strongly isotropy irreducible spaces not belonging to the Families I--X, i.e.\ the isolated cases in Tables~\ref{table4:isotropyirredexcepcion-classical}-\ref{table4:isotropyirredexcepcion-exceptional}, the smallest positive eigenvalue of the Laplace-Beltrami operator 
was obtained computationally in \cite{LauretTolcachier-nu-stab} 
except for Nos.\ 4, 14--18. 
Theorem~\ref{thm:lambda1(SO(N)/Ad(H))} already provides the value for the Isolated cases Nos.\ 4, 14, 16 and 18 (see Remark~\ref{rem:isolated-Ad(H)}). 
In this section we complete the cases Nos.\ 15 and 17.

We set $H_1'=\SU(8)$ and $H_2'=\Spin(16)$.
The irreducible representations  $\rho_1=\sigma_{\eta_4}$ and 
$\rho_2=\sigma_{\eta_7}$ of $H_1'$ and $H_2'$ respectively are of real type and have dimensions $N_1:=70$ and $N_2:=128$. 
The spaces $\SO(N_i)/\rho_i(H_i')$ for $i=1,2$ are isotropy irreducible with universal cover given by $G_i/H_i$, where 
\begin{equation}\label{eq:G/H-isolatedcases}
\begin{aligned}
G_1&=\SO(70)/\Z_2,&\qquad \qquad
G_2&=\Spin(128)/\Z_2,\\
H_1&\simeq \SU(8)/\Z_8,&
H_2&\simeq \SO(16)/\Z_2. 
\end{aligned}
\end{equation}

\begin{remark}\label{rmk:pesosSpin}
There is an element $z$ in $\Spin(128)$ (namely $z=e_1\dots e_{2n}$ in the notation of \cite[Ex.~1.35]{Sepanski})
such that $z^2=1$ and the center of $\Spin(128)$ is given by $\{\pm1,\pm z\}\simeq \langle -1\rangle\oplus \langle z\rangle\simeq\Z_2\oplus\Z_2$. 
The subgroup $\Z_2$ of $\Spin(128)$ indicated for $G_2$ is $\{1,z\}$ or $\{1,-z\}$. Consequently  $\Spin(128)/\Z_2\not\simeq \SO(128)=\Spin(128)/\{\pm1\}$. 

The set of weights of $G_2=\Spin(128)/\Z_2$ is given by 
$$
\PP\big(\Spin(128)/\{1,\pm z\}\big) = \Span_\Z\big(
	\DD_{64},\tfrac12(\ee_1+\dots+\ee_{63} \pm\ee_{64}) 
\big)
=: \DD_{64}^{\pm}
,
$$ 
where 
$
\DD_{64}=
\left\{{\textstyle \sum_{i=1}^{64}}a_i\ee_i: a_i\in\Z\;\forall \,i,\; \; {\textstyle \sum_{i=1}^{64}} a_i\text{ is even}\right\}
$ is the root lattice of $\so(128)_\C$. 
\end{remark}

\begin{theorem}\label{thm:isolatedcases}
If $G_i$ and $H_i$ are as in \eqref{eq:G/H-isolatedcases}, then  
$$
\lambda_1(G_i/H_i, g_{\st})
=\lambda^{\pi_{2\omega_2}}
=
\begin{cases}
\frac{69}{34}=2+\frac{1}{34} &\text{ if }i=1,\\[1mm]
\frac{127}{63}=2+\frac{1}{63} &\text{ if }i=2.
\end{cases}
$$ 
Its multiplicity in $\Spec(G_i/H_i, g_{\st})$ is equal to 
$$
\dim V_{\pi_{2\omega_2}} 
= 
\begin{cases}
1997940&\quad\text{for $i=1$},\\
22360000&\quad\text{for $i=2$}.
\end{cases}
$$
\end{theorem}

\begin{proof}
As representations of $(\fg_i)_\C=\so(N_i)_\C$, we have  $V_{\pi_{\omega_p}}\simeq {\textstyle\bigwedge^p}(\pi_{\omega_1})$ for $p=1,2,3$ (see Remark~\ref{rem:ExtSym-standard}). 
Thus
\begin{align*}
\big(
	\pi_{\omega_2}\otimes\pi_{\omega_{2}} 
\big)|_{(\fh_i')_\C}
	&\simeq {\textstyle\bigwedge}^2(\pi_{\omega_1}|_{(\fh_i')_\C})
	\otimes {\textstyle\bigwedge}^2(\pi_{\omega_1}|_{(\fh_i')_\C}) 
	\simeq 
	{\textstyle\bigwedge}^2(\rho_i)
	\otimes 
	{\textstyle\bigwedge}^2(\rho_i) 
,
\\
\big(
	\pi_{\omega_3}\otimes\pi_{\omega_{1}} 
\big)|_{(\fh_i')_\C}
	&\simeq {\textstyle\bigwedge}^3(\pi_{\omega_1}|_{(\fh_i')_\C})
	\otimes (\pi_{\omega_1}|_{(\fh_i')_\C}) 
	\simeq 
	{\textstyle\bigwedge}^3(\rho_i)
	\otimes 
	\rho_i
. 
\end{align*}
One checks that
\begin{align*}
	 {\textstyle\bigwedge}^2(\rho_1)&= {\textstyle\bigwedge}^2(\sigma_{\eta_4})
	 \simeq \sigma_{\eta_3+\eta_5} \oplus \sigma_{\eta_1+\eta_7}, 
\\
	 {\textstyle\bigwedge}^3(\rho_1)&= {\textstyle\bigwedge}^3(\sigma_{\eta_4})
	 \simeq \sigma_{\eta_1+\eta_3} \oplus \sigma_{\eta_5+\eta_7}
	 \oplus \sigma_{\eta_1+\eta_4+\eta_7} \oplus \sigma_{\eta_2+2\eta_5} \oplus \sigma_{2\eta_3+2\eta_6}, 
 \\
	 {\textstyle\bigwedge}^2(\rho_2)&= {\textstyle\bigwedge}^2(\sigma_{\eta_7})
	 \simeq \sigma_{\eta_2} \oplus \sigma_{\eta_6}, 
\\
	 {\textstyle\bigwedge}^3(\rho_2)&= {\textstyle\bigwedge}^3(\sigma_{\eta_7})
	 \simeq \sigma_{\eta_2+\eta_7} \oplus \sigma_{\eta_1+\eta_8}
	 \oplus \sigma_{\eta_5+\eta_8}.
\end{align*}
Since all irreducible constituents appearing in ${\textstyle\bigwedge}^2(\rho_i)$ are self-conjugate, we obtain $\dim V_{\pi_{\omega_2}\otimes\pi_{\omega_{2}}}^{(\fh_i')_\C}=2$ (see Remark~\ref{rem:rhoxrho^*}). 
Similarly, $\dim V_{\pi_{\omega_3}\otimes\pi_{\omega_{1}}}^{(\fh_i')_\C}=0$ because $\rho_i^*$ does not occur in the decomposition of ${\textstyle\bigwedge}^3(\rho_i)$.

We summarize these computations in the left column of the following decomposition, while the tensor decompositions shown in the right column follow from Remark~\ref{rem:KoikeTerada}:
\begin{align*}
\dim V_{\pi_{\omega_2}\otimes\pi_{\omega_{2}}}^{\fh_\C'} &=2,&
	\pi_{\omega_2}\otimes\pi_{\omega_{2}}
	&\simeq \pi_{\omega_4}
	\oplus \pi_{\omega_1+\omega_3}
	\oplus \pi_{2\omega_2}
	\oplus \pi_{2\omega_1}
	\oplus \pi_{\omega_2}
	\oplus \pi_{0},
\\
\dim V_{\pi_{\omega_3}\otimes\pi_{\omega_{1}}}^{\fh_\C'} &=0,&
	\pi_{\omega_3}\otimes\pi_{\omega_{1}}
	&\simeq  \pi_{\omega_4} \oplus \pi_{\omega_1+\omega_3}
	\oplus \pi_{\omega_2}
.
\end{align*}
From the second row we see that 
$V_{\pi_{\omega_4}}^{(\fh_i')_\C} = V_{\pi_{\omega_1+\omega_3}}^{(\fh_i')_\C} = V_{\pi_{\omega_2}}^{(\fh_i')_\C}=0$. 
We know that $V_{\pi_{2\omega_1}}^{(\fh_i')_\C} =0$ from Lemma~\ref{lem:omega2-no-esferica}\eqref{item:omega_2^H=2omega1=0}.
Furthermore, the first row implies 
$
\dim V_{\pi_{2\omega_2}}^{(\fh_i')_\C}=1. 
$

In order to establish $\lambda_1(G/H,g_{\st})=\lambda^{\pi_{2\omega_2}}$, it remains to check \eqref{eq:sufficient-lambda_1}, which follows immediately from Lemma~\ref{lem:SO-isolated} and the above conclusions. 

Finally, by Lemma~\ref{lem:SO-isolated} the multiplicity of $\lambda_1(G/H,g_{\st})$ in $\Spec(G/H,g_{\st})$ is $\dim V_{\pi_{2\omega_2}}$ .
\end{proof}

\begin{lemma}\label{lem:SO-isolated}
For $\Lambda\in\PP^+(G_i)$ for any $i=1,2$, we have $\lambda^{\pi_{\Lambda}} <\lambda^{\pi_{2\omega_2}}$ if and only if $\Lambda$ is one of the following:
$0$, $2\omega_1$, $\omega_2$, $\omega_1+\omega_3$, $\omega_4$. 
Moreover, $\lambda^{\pi_{\Lambda}} =\lambda^{\pi_{2\omega_2}}$ if and only if $\Lambda = 2\omega_2$.
\end{lemma}

\begin{proof}
Arguing as in the proof of Lemma~\ref{lem:SOlambda^pi<lambda^omega3}, we obtain that the weights $\Lambda$ in $\PP^+(\so(70)_\C)$ or $\PP^+(\so(128)_\C)$ satisfying $\lambda^{\pi_{\Lambda}} \leq \lambda^{\pi_{2\omega_2}}$ are $0$, $\omega_1$, $2\omega_1$, $3\omega_1$, $\omega_2$,  $\omega_1+\omega_2$, $\omega_3$, $2\omega_2$, $\omega_1+\omega_3$, $\omega_4$, and the equality holds only for $2\omega_2$.
Among these, the $G_i$-integral ones 
(see \cite[Rmk.~2.8]{LauretTolcachier-nu-stab} and  Remark~\ref{rmk:pesosSpin})
are $0$, $2\omega_1$,  $\omega_2$, $2\omega_2$,  $\omega_1+\omega_3$, and $\omega_4$. 
\end{proof}

\section{Uniform bounds for the first eigenvalue}\label{sec:bounds}

The goal of this section is to provide lower and upper bounds for $\lambda_1(G'/H',g_{\st})$, where $G'/H'$ is an arbitrary compact strongly isotropy irreducible space, including the irreducible compact symmetric spaces. 

It turns out that $G'/H'$ is covered by a simply connected strongly isotropy irreducible space $G/H$ (see Remark~\ref{rem:universalcover} for details for the non-symmetric case). 
Moreover, $E(G/H,g_{\st})=E(G'/H',g_{\st})$. 
Since $\lambda_1(M,g)\leq \lambda_1(N,g')$ for any Riemannian covering $(M,g)\to (N,g')$, we have  
$$
\lambda_1(G/H,g_{\st})\leq \lambda_1(G'/H',g_{\st}).
$$
Moreover, $\lambda_1(G/H,g_{\st})$ has been computed by Urakawa~\cite{Urakawa86} for any compact irreducible symmetric space and in Section~\ref{sec:lambda_1(simplyconnected)} for the remaining cases.

\subsection{Non-symmetric strongly isotropy irreducible spaces}
By the comments above, the lower bound for $\lambda_1(G'/H',g_{\st})$ follows by a simple inspection of Tables~\ref{table4:isotropyirred-families2}--\ref{table4:isotropyirredexcepcion-exceptional}. 

\begin{proposition}\label{prop4:lowerbound-non-symmetric}
If $G'/H'$ is any non-symmetric strongly isotropy irreducible space, then 
$$
\lambda_1(G'/H',g_{\st})\geq  
\lambda_1(\op{G}_2/\SU(3),g_{\st})
=\frac12.
$$
Moreover, the equality holds only for $G'/H'=\op{G}_2/\SU(3)$. 
\end{proposition}

We now focus on the upper bound, which is more delicate. 
We denote by $Z(G)$ and $N_G(H)$ the center of $G$ and the normalizer of $H$ in $G$, respectively. 

\begin{lemma}\label{lem4:Z(G)}
Let $G$ be a compact connected semisimple Lie group and $\pi$ a representation of $G$. 
Then $(\pi\otimes\pi^*)(z)=\Id_{V_\pi\otimes V_\pi^*}$ for all $z\in Z(G)$. 
In particular, $\pi'(z)= \Id_{V_{\pi'}}$ for all $z\in Z(G)$, for any irreducible constituent $\pi'$ of $\pi\otimes\pi^*$. 
\end{lemma}

\begin{proof}
Fix $z\in Z(G)$. 
Since $\pi(z)$ commutes with $\pi(a)$ for all $a\in G$, Schur's Lemma yields $\pi(z)=\xi\, \Id_{V_\pi}$ for some $\xi\in\C$. 
Since the order of $z$ is finite, because $Z(G)$ is finite, then $|\xi|=1$.

Note that $\pi^*(z)=\bar\xi\, \Id_{V_\pi^*}$. 
Indeed, for $\varphi\in V_\pi^*$ and $v\in V_\pi$, we compute:  $(\pi^*(z)\cdot \varphi)(v) = \varphi(\pi(z^{-1})\cdot v)=\varphi(\xi^{-1} \, v)=\bar\xi\varphi(v)$, where  
in the second step we use that $\pi(z^{-1})=\xi^{-1} \, \Id_{V_\pi} =\bar \xi \,\Id_{V_\pi}$ since $\pi(z)\circ\pi(z^{-1}) = \Id_{V_\pi}$. 

We conclude that $(\pi\otimes \pi^*)(z)\cdot v\otimes\varphi = (\pi(z)\cdot v)\otimes (\pi^*(z)\cdot \varphi)=  (\xi\, v)\otimes (\bar \xi\,\varphi)=v\otimes\varphi$ for all $v\in V_\pi$ and $\varphi\in V_\pi^*$. 
\end{proof}

\begin{lemma}\label{lem4:N/H-finite}
Let $G$ be a compact connected Lie group, and let $H,K$ be closed subgroups of $G$ satisfying that $H\subset K\subset N_G(H)$ and the group $K/H$ is finite. 
Let $\pi_{\Lambda}$ be a non-trivial irreducible representation of $G$ with highest weight $\Lambda$ satisfying $V_{\pi_\Lambda}^H\neq0$ (i.e.\ $\pi_{\Lambda}\in\widehat G_H$) and let $\Lambda^*$ denote the highest weight of $\pi_{\Lambda}^*$, i.e.\ $\pi_{\Lambda^*}\simeq\pi_{\Lambda}^*$. 
If $K/H$ is abelian or $\dim V_\pi^H=1$, then there is $\pi'\in\widehat G\smallsetminus\{1_G\}$ such that $V_{\pi'}^{K}\neq0$ and $\lambda^{\pi'}\leq\lambda^{\pi_{\Lambda+\Lambda^*}}$. 
\end{lemma}

\begin{proof}
We abbreviate $\pi=\pi_\Lambda$. 
We first show that $\pi(k)$ preserves $V_{\pi}^H$ for all $k\in K$. 
Indeed, for $k\in K$ and $v\in V_{\pi}^H$, we have 
$$
\pi(h)\cdot\big(\pi(k)\cdot v\big)
= \pi(k)\cdot\big( \pi(k^{-1}hk)\cdot v\big)
= \pi(k)\cdot v
\qquad\forall \, h\in H,
$$
thus $\pi(k)\cdot v\in V_\pi^H$.

Let $m$ be the order of $K/H$.
For any $k\in K$, $k^m\in H$, thus $\pi(k)^m=\Id_{V_\pi^H}$, and consequently the eigenvalues of $\pi(k)$ are $m$-th roots of unity.

We claim that there exist a non-zero $v\in V_\pi^H$ and a character $\xi:K\to S^1$ such that $\pi(k)\cdot v=\xi(k)\, v$ and $\xi(k)^m=1$ for all $k\in K$. 
This is obvious if $\dim V_\pi^H=1$. 
We now prove the claim under the assumption that $K/H$ is abelian. 
Choose $k_1,\dots,k_m\in K$ such that $k_ik_j^{-1}\notin H$ for any $i\neq j$, so $K=\bigcup_{i=1}^m k_iH$. 
Since $\pi$ is unitary and $\{\pi(k_i)|_{V_\pi^H}: 1\leq i\leq m\}$ is a commuting family of linear maps (because $K/H$ is abelian), they can be simultaneously diagonalized. 
Thus, we pick $v$ a non-zero vector $v$ that is a joint eigenvector of $\{\pi(k_i)|_{V_\pi^H}:1\leq i\leq m\}$. 

It is well known that $\dim (V_\pi^*)^H=\dim V_\pi^H$. 
Similarly as above, we can pick $\varphi\in (V_\pi^*)^H$ with $\varphi(v)\neq0$ satisfying that $\pi^*(k)\cdot \varphi=\zeta(k)\, \varphi$ for some character $\zeta:K\to S^1$. 
Note that, for any $k\in K$, $\zeta(k)\, \varphi(v)=\big(\pi^*(k)\cdot \varphi\big)(v)
=\varphi\big(\pi(k^{-1})\cdot v\big)
=\varphi\big( \xi(k^{-1})\,  v\big)
=\overline{\xi(k)}\varphi\big(  v\big)
$.
Hence $\zeta(k)=\overline{\xi(k)}$ for every $k\in K$. 
We deduce that 
\begin{align*}
(\pi\otimes\pi^*)(k)\cdot v\otimes \varphi&
=\big(\pi(k)\cdot v\big)\otimes \big(\pi^*(k)\cdot \varphi\big)
=\big(\xi(k)\, v\big)\otimes \big(\overline{\xi(k)}\, \varphi\big)
=v\otimes \varphi
\end{align*}
for all $k\in K$, that is, $v\otimes \varphi\in (V_\pi\otimes V_\pi^*)^{K}$. 

We decompose $\pi\otimes\pi^*$ in irreducible factors:
there are $G$-invariant subspaces $W_0,\dots,W_p$ of $V_\pi\otimes V_\pi^*$ such that $V_\pi\otimes V_\pi^*=W_0\oplus\dots\oplus W_p$ with $\pi_j:=\pi\otimes \pi^*|_{W_j}$ irreducible for all $j$. 
By Remark~\ref{rem:rhoxrho^*}, we can assume that $\pi_0\simeq 1_G$ and $\pi_{j}\not\simeq 1_G$ for all $1\leq j\leq p$. 
Clearly, 
$$
(V_\pi\otimes V_\pi^*)^{K}
=W_{0}^{K}\oplus W_{1}^{K}\oplus\dots\oplus W_{p}^{K}
.
$$
The element $v\otimes \varphi\in (V_\pi\otimes V_\pi^*)^{K}$ is obviously not $G$-invariant (because $\pi\not\simeq 1_G$), i.e.\ $v\otimes \varphi\notin W_0^{K}=W_0$.
It follows that $v\otimes \varphi$ has a non-zero component in some $W_{j}^{K}$ for some $1\leq j\leq p$, or in other words, $V_{\pi_j}^{K}\neq0$. 

Let us denote by $\Lambda'$ the highest weight of $\pi':=\pi_j$, which is an irreducible representation occurring in  $\pi\otimes\pi^*=\pi_\Lambda\otimes\pi_{\Lambda^*}$. 
It only remains to show that $\lambda^{\pi_{\Lambda'}}\leq \lambda^{\pi_{\Lambda+\Lambda^*}}$. 

It is well known that (see e.g.\ \cite[Ex.~25.33]{FultonHarris-book}) $\Lambda'$ can be written as $\Lambda'=\Lambda+\mu$ for some weight $\mu$ of $\pi^*$. 
In particular, $\gamma:=\Lambda+\Lambda^*-\Lambda'=\Lambda^*-\mu$ is zero or a sum of positive simple roots.
Hence
\begin{align*}
\lambda^{\pi_{\Lambda+\Lambda^*}} - \lambda^{\pi_{\Lambda'}}&
= \inner{\Lambda+\Lambda^*+2\rho}{\Lambda+\Lambda^*} 
- \inner{\Lambda'+2\rho}{\Lambda'} 
\\ & 
= \inner{\Lambda'+\gamma +2\rho}{\Lambda'+\gamma } 
- \inner{\Lambda'+2\rho}{\Lambda'} 
\\ & 
= 
 \inner{\Lambda' +2\rho}{\gamma } 
+ \inner{\gamma}{\Lambda'} 
+ \inner{\gamma}{\gamma} 
\geq0
\end{align*}
since $\Lambda'$ and $2\rho$ are dominant. 
\end{proof}

\begin{remark}\label{rem:universalcover}
Let $M'=G'/H'$ be a (not necessarily simply connected) non-symmetric strongly isotropy irreducible space. 
Without loss of generality, we may assume that $G'/H'$ is an effective presentation with $G'$ connected and $H_0'$ (the connected component of $H'$) acting irreducibly on the tangent space of $M'$. 
According to the second paragraph of \cite[Thm.~11.1]{Wolf68}, there are:
\begin{itemize}
\item a simply connected strongly isotropy irreducible space $G/H$ as in Tables~\ref{table4:isotropyirred-families},%
\ref{table4:isotropyirredexcepcion-classical},%
\ref{table4:isotropyirredexcepcion-exceptional}, 

\item a central subgroup $Q$ of the center $Z(G)$ of $G$, 
\item a subgroup $H''$ of the normalizer $N_G(H)$ of $H$ in $G$ with $H=H_0''=N_G(H)_0$,
\end{itemize}
such that 
\begin{align*}
	G'&=G/Q,&
	H'=(Q\cdot H'')/Q. 
\end{align*}
Clearly, $G'/H'$ is diffeomorphic to $G/(Q\cdot H'')$, and $(G'/H',g_{\st})$ is isometric to $(G/Q\cdot H'', g_{\st})$. 
Consequently, without losing generality, we can assume $G'=G$ and $H'=Q\cdot H''$. 
\end{remark}

\begin{remark}\label{rem4:Z=eN_G(H)=H}
Family VIII and the Isolated cases Nos.\ 4,5,6,8,9,20,22,25,26,29,31,32 satisfy that $Z(G)$ and $N_G(H)/H$ are simultaneously trivial.
In any of these cases, $G/H$ as in Tables~\ref{table4:isotropyirred-families},%
\ref{table4:isotropyirredexcepcion-classical},%
\ref{table4:isotropyirredexcepcion-exceptional}
is the only strongly isotropy irreducible space associated.  
\end{remark}

We are now in a position to give an upper bound for $\lambda_1(G'/H',g_{\st})$.

\begin{theorem}\label{thm4:Lambda+Lambda^*}
Let $G'/H'$ be a non-symmetric strongly isotropy irreducible space and let $G/H$ be its universal cover as in Remark~\ref{rem:universalcover}. 
We assume that $G/H\not\simeq (\Sp(4)/\Z_2)/(\Sp(1)/\Z_2\times \SO(4)/\Z_2)$ (Family IV with $n=4$).
Write $\lambda_1(G/H,g_{\st})=\lambda^{\pi_{\Lambda}}$, where $\Lambda$ is explicitly indicated in Tables~\ref{table4:isotropyirred-families2}--\ref{table4:isotropyirredexcepcion-exceptional} in each case. 
Then, 
$$
\lambda_1(G'/H',g_{\st}) \leq \lambda^{\pi_{\Lambda+\Lambda^*}},
$$
where $\Lambda^*$ denotes the highest weight of $\pi_{\Lambda}^*$. 
\end{theorem}

\begin{proof}
According to Remark~\ref{rem:universalcover}, we can assume that $G'=G$ and $H'=Q\cdot H''$ for some subgroup $Q$ of $Z:=Z(G)$ and $H''$ a subgroup of $N_G(H)$ satisfying $H=H_0''=N_G(H)_0$. 

We set $K=Z\cdot N_G(H)=N_G(H)$, thus $H\subset H'\subset K$. 
Since $V_{\pi'}^K\subset V_{\pi'}^{H'}$, it is sufficient to find $\pi'\in\widehat G$ satisfying that $\lambda^{\pi'}\leq \lambda^{\pi_{\Lambda+\Lambda^*}}$ and $V_{\pi'}^K\neq0$.

In each case, $Z(G)$ and $N_G(H)$ are specified in Tables~\ref{table4:isotropyirred-families}, \ref{table4:isotropyirredexcepcion-classical}, \ref{table4:isotropyirredexcepcion-exceptional}, originally computed in \cite[pages 107--110]{Wolf68} and \cite[page 142]{Wolf84-errata-isot-irred}. 
It turns out that $N_G(H)/H$ is always abelian except for $n=4$ in Family IV, which is precisely the omitted case in the statement. 
Consequently, since $Z$ is finite (because $G$ is a compact simple Lie group), $K/H=N_G(H)/H$ is a finite abelian group. 
Now, Lemma~\ref{lem4:N/H-finite} proves the claim. 
\end{proof}

We now consider the omitted case in Theorem~\ref{thm4:Lambda+Lambda^*}.

\begin{remark}\label{rem4:upperbound-Sp(4)/Sp(1)xSO(4)}
Let $G'/H'$ be a non-symmetric strongly isotropy irreducible space whose universal cover $G/H$ belongs to Family~IV with $n=4$, that is, 
$$ 
G=\Sp(4)/\Z_2 \qquad\text{and}\qquad H=\Sp(1)/\Z_2 \times \SO(4)/\Z_2. 
$$

We explain first why this case does not satisfy the conclusion of Theorem~\ref{thm4:Lambda+Lambda^*}.  
On the one hand, $N_G(H)/H\simeq \mathbb S_3$, the symmetric group on three letters, which is non-abelian.  
On the other hand, $\lambda_1(G/H,g_{\st})=\lambda^{\pi_{2\omega_2}}$ and $\dim V_{\pi_{2\omega_2}}^H=2$ (see the proof of Theorem~\ref{thm:flia4}).

From the proof of Theorem~\ref{thm:flia4}, we also know that $\dim V_{\pi_{4\omega_1}}^{H}=1$.  
Applying Lemma~\ref{lem4:N/H-finite} with $K=N_G(H)$, there exists $\pi'\in\widehat{G}$ such that $\lambda^{\pi'} \leq \lambda^{\pi_{8\omega_1}}$ and $0 \neq V_{\pi'}^{K} \subset V_{\pi'}^{H'}$. 
We conclude that
$$
\lambda_1(G'/H',g_{\st}) \leq \lambda^{\pi_{8\omega_1}}.
$$
\end{remark}

The upper bound from Theorem~\ref{thm4:Lambda+Lambda^*} can be improved in many cases. 
Here is an example. 

\begin{proposition}\label{prop4:N=ZH+pi(Z)trivial}
Let $G'/H'$ be a non-symmetric strongly isotropy irreducible space and let $G/H$ be its universal cover as in Remark~\ref{rem:universalcover}. 
If $N_G(H)= Z(G)\cdot H$ and $\pi \in\widehat G_H\smallsetminus\{1_G\}$ (thus $\lambda^{\pi}\in\Spec(G/H,g_{\st})$) satisfies $\pi(z)=\Id_{V_\pi}$ for all $z\in Z$, then 
$$
\lambda_1(G'/H',g_{\st}) \leq \lambda^{\pi}.
$$
\end{proposition}

\begin{proof}
By Remark~\ref{rem:universalcover}, since $N_G(H)=Z(G)\cdot H$, we can assume that $G'=G$ and $H'=Q\cdot H$ for some subgroup $Q$ of $Z(G)$.
The hypotheses imply $V_\pi^{H'}\supset V_\pi^{Z(G)\cdot H}=V_\pi^H\neq0$, thus $\lambda^{\pi}\in\Spec(G'/H',g_{\st})$ by Theorem~\ref{thm:Spec(standard)}, and consequently $\lambda_1(G'/H',g_{\st}) \leq \lambda^{\pi}$.
\end{proof}

We next observe some cases that can be dealt with Proposition~\ref{prop4:N=ZH+pi(Z)trivial}.

\begin{remark}[Improved upper bound]
\label{rem4:improved-upperbound}
Suppose $G'/H'$ is a non-symmetric strongly isotropy irreducible space associated with Family I or II, and let $G/H$ be its universal cover with $G/H$ as in Table~\ref{table4:isotropyirred-families}. 

We next verify that the hypotheses in Proposition~\ref{prop4:N=ZH+pi(Z)trivial} are satisfied. 
We established in the proof of Theorem~\ref{thm:flia1-2}  that $\pi_{2\omega_1+2\omega_{N-1}}\in \widehat G_H$.
(Actually $\lambda_1(G/H,g_{\st})$ is attained in $\lambda^{\pi_{2\omega_1+2\omega_{N-1}}}$ in most cases.) 
We have that $N_G(H)\simeq Z(G)\cdot H$, see Table~\ref{table4:isotropyirred-families}. 
Furthermore, $\pi_{2\omega_1+2\omega_{N-1}}$ occurs with multiplicity one in $\pi_{2\omega_1}\otimes \pi_{2\omega_{N-1}}\simeq \pi_{2\omega_1}\otimes \pi_{2\omega_{1}}^*$, thus $\pi_{2\omega_1+2\omega_{N-1}}(z)$ acts trivially by Lemma~\ref{lem4:Z(G)} for every $z\in Z(G)$. 
We conclude by Proposition~\ref{prop4:N=ZH+pi(Z)trivial} that
$$
\lambda_1(G'/H',g_{\st})\leq \lambda^{2\omega_1+2\omega_{N-1}}. 
$$

Similarly, when $G'/H'$ corresponds to Family III or to the Isolated cases Nos.\ 11, 13 and 17, the hypotheses in Theorem~\ref{thm4:Lambda+Lambda^*} are again satisfied. 
We check here only that $\pi_{\omega_2+\omega_{pq-2}}\subset  \pi_{\omega_2}\otimes \pi_{\omega_{pq-2}}\simeq \pi_{\omega_2}\otimes \pi_{\omega_{2}}^*$ for Family III, and
$\pi_{2\omega_2}\subset  \pi_{\omega_2}\otimes \pi_{\omega_{2}}\simeq \pi_{\omega_2}\otimes \pi_{\omega_{2}}^*$ for the isolated cases. 
\end{remark}

\subsection{Compact irreducible symmetric spaces}

In this section we deal, for completeness, with strongly isotropy irreducible spaces that are \emph{symmetric}; they are irreducible compact symmetric spaces. 
The situation for them is much more clear because a great amount of bibliography (e.g.\ \cite{Helgason-book-DiffGeomLieGrSymSpaces}, \cite{Loos1}, \cite{Loos2}) and results. 
In particular, for the simply connected spaces, the determination of the first Laplace eigenvalue was completed
by Urakawa~\cite{Urakawa86} (see also \cite{CaoHe15} for tables showing $\lambda_1/E$). 

We now give an upper bound for the first Laplace eigenvalue of an arbitrary irreducible compact symmetric space. 
We start with those of group type. 
As before, we do not lose generality by always considering the standard metric. 

\begin{proposition}\label{prop4:symm-grouptype}
Let $(G',g_{\st})$ be an irreducible compact symmetric space of group type, with $G'$ a compact connected simple Lie group. 
Let $G$ denote the universal cover of $G'$. 
Then
$$
\frac{5}{12}\leq 
\lambda_1(G,g_{\st})\leq 
\lambda_1(G',g_{\st}) \leq 
\lambda^{\Ad_G}=1. 
$$
\end{proposition}

\begin{proof}
The upper bound is obvious because the adjoint representation of $G'$ always lies in $\widehat G'$, and $\lambda^{\Ad_{G'}}=\lambda^{\Ad_G}=1$ is well known. 
The lower bound follows immediately from \cite[Table A.1]{Urakawa86}, and it is attained in $\lambda_1(\Spin(5),g_{\st})=\frac{5}{12}$. 
\end{proof}

We now continue with irreducible compact symmetric spaces of non-group type. 

\begin{proposition}\label{prop4:symm-non-grouptype}
Let $G'/H'$ be an irreducible compact symmetric space of non-group type (i.e.\ $H'$ is not trivial and $G'$ is simple). 
Let $G/H$ be its universal cover, which is again a symmetric space, and write $\lambda_1(G/H,g_{\st})=\lambda^{\pi_{\Lambda}}$ ($\Lambda$ can be explicitly obtained from \cite[Table A.2]{Urakawa86} in each case). 
Then, 
$$
\frac12<
\lambda_1(G/H,g_{\st})\leq 
\lambda_1(G'/H',g_{\st}) \leq \lambda^{\pi_{\Lambda+\Lambda^*}}
\leq \frac{22}{5},
$$
where $\Lambda^*$ denotes the highest weight of $\pi_{\Lambda}^*$. 
\end{proposition}

\begin{proof}
The lower bound follows immediately from \cite[Table A.2]{Urakawa86} and
is sharp because $\lambda_1(\SO(n+1)/\SO(n),g_{\st})=\frac{n}{2(n-1)}\to\frac12$ as $n\to+\infty$. 

It is well known that any irreducible compact symmetric space $(G'/H',g_{\st})$ of non-group type is associated with an orthonormal symmetric Lie algebra $(\fg,\theta)$ of the compact type  such that $\fg'=\fg$ and the fixed point set $\fh\simeq\fh'$ of the involution $\theta:\fg\to\fg$ does not contain any non-trivial ideal. 
Let $G$ denote the simply connected Lie group with Lie algebra $\fg$, let $\vartheta:G\to G$ be the homomorphism of Lie groups satisfying $d\vartheta=\theta$, and set $H=\{x\in G: \vartheta(x)=x\}$.

From  \cite[\S{}VII.9]{Helgason-book-DiffGeomLieGrSymSpaces}, we have that $G'/H'$ is necessarily covered by the simply connected compact symmetric space associated with $(\fg,\theta)$, namely $G/H$, and furthermore, it covers the \emph{adjoint space} associated with $(\fg,\theta)$, namely, the compact symmetric space $G/K$ with 
$$K=\{k\in G: k^{-1}\vartheta(k)\in Z(G) \}.$$

The Riemannian coverings
$
(G/H,g_{\st})
\to
(G'/H',g_{\st})
\to 
(G/K,g_{\st})
$
give 
\begin{align*}
\lambda_1(G/H,g_{\st})
\leq
\lambda_1(G'/H',g_{\st})
\leq
\lambda_1(G/K,g_{\st}).
\end{align*}
It remains to show that $\lambda_1(G/K,g_{\st})\leq \lambda^{\pi_{\Lambda+\Lambda^*}}$.

We claim that $K\subset N_G(H)$. 
Indeed, for $k\in K$ and $h\in H$, we have $khk^{-1}\in H$ because
\begin{align*}
\vartheta(khk^{-1})  &
= \big(\vartheta(k) h k^{-1}\big) \cdot  \big(k\vartheta(k^{-1})\big) 
\\ & 
= \big(k\vartheta(k^{-1})\big) \cdot \big(\vartheta(k) h k^{-1}\big) 
	\qquad\text{(since $k\vartheta(k^{-1})\in Z(G)$)}
\\ & 
=  k  h k^{-1}.
\end{align*}
Moreover, it is well known that $N_G(H)/H$ is a finite abelian group (see e.g.~\cite[Ex.~C 1--4 in Ch.~X]{Helgason-book-DiffGeomLieGrSymSpaces} or \cite[\S{}IV.4]{Loos1}).
Lemma~\ref{lem4:N/H-finite} forces $\lambda_1(G/K,g_{\st})\leq \lambda^{\pi_{\Lambda+\Lambda^*}}$, as claimed. 

One can check that the maximum of $\lambda^{\pi_{\Lambda+\Lambda^*}}$ is attained at Type E VIII: $G/H=\op{E}_8/\SO(16)$, $\Lambda=\Lambda^*=2\omega_8$, and $\lambda^{\pi_{4\omega_8}}=22/5$. 
\end{proof}

\subsection{Upper bound for the first eigenvalue over the Einstein constant}

We are in a position to prove the main result of the article. 

\begin{proof}[Proof of Theorem~\ref{thm1:Saloff-Coste}]

For any connected standard Einstein manifold $(G/H,g_{\st})$, its Einstein constant $E(G/H,g_{\st})$ satisfies (see \cite[Cor.~1.6]{WangZiller85})
\begin{equation}\label{eq:E}
	\frac14\leq E(G/H,g_{\st})\leq \frac12. 
\end{equation}
Furthermore, $E(G/H,g_{\st})=\frac14$ if and only if $H=e$.
When $H\neq e$, $E(G/H,g_{\st})=\frac12$ if and only if $G/H$ is locally symmetric. 

We now assume that $G'/H'$ is a strongly isotropy irreducible space. 
In particular, the standard manifold $(G'/H',g_{\st})$ is Einstein, so \eqref{eq:E} holds.  
We abbreviate $\lambda_1=\lambda_1(G'/H',g_{\st})$ and $E=E(G'/H',g_{\st})$. 

If $H'=e$, then $(G'/H',g_{\st})$ is symmetric space of group type and $E=1/4$.
Proposition~\ref{prop4:symm-grouptype} yields
$
\frac{5}{3}\leq 
\lambda_1/E
\leq 4.
$

If $G'/H'$ is a symmetric space of non-group type, then $E=\frac12$, and Proposition~\ref{prop4:symm-non-grouptype} gives 
$1<\lambda_1/E\leq \frac{44}{5}<16$. 

We now assume that $G'/H'$ is not symmetric. 
It follows from
\eqref{eq:E} and Proposition~\ref{prop4:lowerbound-non-symmetric} that $\lambda_1/E\geq 2\lambda_1>1$ if $G'/H'\neq \op{G}_2/\SU(3)$. 
Furthermore, $\lambda_1(\op{G}_2/\SU(3),g_{\st})/E=\frac12/\frac5{12}=\frac{6}{5}>1$. 

Regarding the upper bound for $\lambda_1$, Tables~\ref{table4:isotropyirred-families2}--\ref{table4:isotropyirredexcepcion-exceptional} provide explicit values for each case, obtained either from Theorem~\ref{thm4:Lambda+Lambda^*}, or from Remarks~\ref{rem4:Z=eN_G(H)=H}, \ref{rem4:upperbound-Sp(4)/Sp(1)xSO(4)} or \ref{rem4:improved-upperbound}.
A simple inspection of these values shows that 
$$\lambda_1/E\leq  16=\tfrac{32}5/\tfrac25=\lambda^{\pi_{8\omega_1}}/E\big( \tfrac{\Sp(4)/\Z_2}{\Sp(1)/\Z_2\times\SO(4)/\Z_2},g_{\st}\big),
$$ 
where $\lambda^{\pi_{8\omega_1}}$ is the upper bound (see Remark~\ref{rem4:upperbound-Sp(4)/Sp(1)xSO(4)}) for $\lambda_1(G'/H',g_{\st})$ when $G'/H'$ is covered by $(\Sp(4)/\Z_2)/(\Sp(1)/\Z_2\times\SO(4)/\Z_2)$. 
This completes the proof. 
\end{proof}

\begin{remark}
One can check from Tables~\ref{table4:isotropyirred-families2}--\ref{table4:isotropyirredexcepcion-exceptional} that, among simply connected non-symmetric strongly isotropy irreducible spaces $G/H$ with $(G/H,g_{\st})$ not isometric to a round sphere (i.e.\ $\Spin(7)/\op{G}_2\simeq S^7$ and $\op{G}_2/\SU(3)\simeq S^6$ are omitted), one has 
$$
{\lambda_1(G/H,g_{\st})}/{E(G/H,g_{\st})} \geq {60}/{31}\approx 1.93,
$$
attained only at $G/H=(\SU(6)/\Z_6)/(\SU(3)/\Z_3\times\SU(2)/\Z_2)$ (Family III, $(p,q)=(2,3)$). 
\end{remark}

\bibliographystyle{plain}

\end{document}